\documentclass[a4paper,11pt]{article}

\usepackage[T1]{fontenc}
\usepackage[utf8]{inputenc}
\usepackage{lmodern}
\usepackage[text={6.5in,9.0in},centering]{geometry}
\usepackage{microtype}
\usepackage{amsmath,amssymb,amsthm,mathtools,mathrsfs}
\usepackage{enumitem}
\usepackage{needspace}
\usepackage{xcolor}
\usepackage{authblk}
\usepackage{amsrefs}
\usepackage[
colorlinks=true,
linkcolor=blue,
citecolor=blue,
urlcolor=blue,
pdfauthor={Juan J. Jimenez},
pdftitle={Stochastic heat equations with non-Lipschitz coefficients
	driven by symmetric alpha-stable Levy white noise}
]{hyperref}
\renewcommand{\PrintDOI}[1]{\href{https://doi.org/#1}{doi:\nolinkurl{#1}}}

\numberwithin{equation}{section}
\numberwithin{figure}{section}
\setlist[enumerate]{leftmargin=2em,itemsep=0.15em,topsep=0.3em}

\theoremstyle{plain}
\newtheorem{theorem}{Theorem}[section]
\newtheorem{proposition}[theorem]{Proposition}
\newtheorem{lemma}[theorem]{Lemma}
\newtheorem{corollary}[theorem]{Corollary}
\newtheorem{definition}[theorem]{Definition}
\theoremstyle{definition}

\theoremstyle{remark}

\newcommand{\bbB}{\mathbb{B}}
\newcommand{\bbD}{\mathbb{D}}
\renewcommand{\H}{\mathbb{H}}
\renewcommand{\L}{\mathbb{L}}
\renewcommand{\P}{\mathbb{P}}
\newcommand{\bbQ}{\mathbb{Q}}
\newcommand{\calA}{\mathcal{A}}
\newcommand{\calB}{\mathcal{B}}
\newcommand{\calC}{\mathcal{C}}
\newcommand{\calD}{\mathcal{D}}
\newcommand{\calE}{\mathcal{E}}
\newcommand{\calG}{\mathcal{G}}
\newcommand{\calI}{\mathcal{I}}
\newcommand{\calJ}{\mathcal{J}}
\newcommand{\calK}{\mathcal{K}}
\newcommand{\calL}{\mathcal{L}}
\newcommand{\calM}{\mathcal{M}}
\newcommand{\calN}{\mathcal{N}}
\newcommand{\calP}{\mathcal{P}}
\newcommand{\calQ}{\mathcal{Q}}
\newcommand{\calR}{\mathcal{R}}
\newcommand{\calS}{\mathcal{S}}
\newcommand{\calT}{\mathcal{T}}
\newcommand{\calV}{\mathcal{V}}
\newcommand{\calZ}{\mathcal{Z}}
\newcommand{\frakB}{\mathfrak{B}}
\newcommand{\frakC}{\mathfrak{C}}
\newcommand{\frakS}{\mathfrak{S}}
\newcommand{\fraks}{\mathfrak{s}}
\newcommand{\scrL}{\mathscr{L}}
\newcommand{\scrM}{\mathscr{M}}
\newcommand{\R}{\mathbb{R}}
\newcommand{\N}{\mathbb{N}}
\newcommand{\E}{\mathbb{E}}

\newcommand{\Fil}{\mathcal{F}}
\newcommand{\m}{\mathfrak{m}}
\newcommand{\dd}{\,\mathrm{d}}
\newcommand{\ind}[1]{\mathbf{1}_{\{#1\}}}

\newcommand{\Path}[1]{\bbD([0,T];#1)}
\newcommand{\Nt}{\widetilde N}
\newcommand{\LK}{L^{K}}
\newcommand{\tauK}{\tau_{K}}

\newcommand{\Cgrow}{C_{\mathrm{gr}}}
\newcommand{\Cbound}[2]{\calC_0(#1,#2)}

\title{Stochastic heat equations with non-Lipschitz coefficients\\
	driven by symmetric $\alpha$-stable L\'evy white noise}
\author{Juan J. Jim\'enez\thanks{Email:
		\href{mailto:juj0003@auburn.edu}{juj0003@auburn.edu}}}
\affil{Department of Mathematics and Statistics\\
	Auburn University, Auburn, AL 36849, USA}
\date{}

\begin{document}
	\maketitle
	\begin{abstract}
		We study the stochastic heat equation on a bounded interval with zero Dirichlet boundary conditions, driven by symmetric $\alpha$-stable space--time L\'evy white noise, where $1<\alpha<2$. Both the drift and noise coefficients are continuous and have at most linear growth. For bounded deterministic initial data and every $r>1/2$, we construct a weak solution with c\`adl\`ag paths in the negative Sobolev space $\H_{-r}$. We also construct a predictable mild solution on the same stochastic basis and with the same noise. At each fixed deterministic time, this mild solution is a density of the distribution-valued weak solution.
	\end{abstract}

	
	\vspace{1mm}
	

	\section{Introduction}\label{sec:introduction}
	\subsection{The problem and its main difficulty}
	
	A jump of the space--time noise adds a weighted Dirac mass at the jump location to the solution of the stochastic heat equation (SHE).
	Although the heat kernel smooths this mass at later times, the jump of
	the solution itself, that is, the difference between the state and its
	left limit, is a multiple of a Dirac mass, not a function. So, unlike
	for the SHE driven by Gaussian noise, the c\`adl\`ag paths of the
	solution must be studied in a space of distributions, and this changes
	both the natural path space and the compactness argument. Chong--Dalang--Humeau \cite{CDH}*{Section~2.1.2, p.~129} illustrate this obstruction for a symmetric pure-jump noise with finite intensity. Their example rules out c\`adl\`ag paths in $\L^p$ on a bounded interval when a nonzero jump is present. It also rules out a jointly continuous space--time version. They show that $\H_{-r}$, with $r>1/2$, is the natural negative Sobolev space for c\`adl\`ag paths on an interval; see \cite{CDH}*{Theorem~2.5 and Remark~2.6}. The threshold is sharp when nonzero Dirac jumps occur.
	
	For the SHE driven by Gaussian space--time white noise, a jointly
	continuous version exists at positive times in the globally Lipschitz
	case; see Dalang--Sanz-Sol\'e \cite{DSS26}*{Theorem~4.3.4(a)}. Joint
	continuity is used in many existence and uniqueness results
	for non-Lipschitz coefficients; see, e.g., \cites{MP11,DKZ19,FN21,CFHS25}.
	These arguments do not apply directly to the SHE when the noise has jumps.
	
	For jump noise, several works prove the existence of a solution to the SHE with non-Lipschitz coefficients under additional assumptions on the coefficients or the noise. Mytnik \cite{Mytnik02} constructs nonnegative weak solutions for one-sided stable noise and a power-type
	noise coefficient.
	
	Hausenblas \cite{Hausenblas07} develops a Banach-space approach under quantitative H\"older assumptions and uses localization to handle large jumps. The stable-noise example in that work has a power-type
	coefficient and only positive jumps. Continuity and linear
	growth alone do not imply these H\"older assumptions. Brze\'zniak--Hausenblas--Razafimandimby \cite{BHR18} study martingale solutions to abstract reaction--diffusion equations under different dissipativity and integrability assumptions. Their application to space--time white noise assumes bounded jumps and a finite $p$-moment of the L\'evy measure; see \cite{BHR18}*{condition~(5.1)} and the erratum \cite{BHR25erratum}. 
	
	Kavin--Majee \cite{KM25} prove well-posedness for the SHE on a bounded
	domain with the logarithmic drift $u\log|u|$, driven by a multiplicative
	L\'evy noise that is homogeneous in time but is not a space--time noise:
	its jumps occur at random times and act on the whole domain at once.
	The drift grows faster than linearly and is not locally Lipschitz at
	$0$. The noise coefficient is locally Lipschitz, with a constant that
	may grow like $(\log|u|)^{1/2}$, and has sublinear growth in $u$; in the
	jump variable, it is dominated by a bounded function that is integrable
	with respect to the L\'evy measure. This excludes $\alpha$-stable noise.

	In a related hyperbolic setting, Balan--Jim\'enez--Quer-Sardanyons
	\cite{BJQ} prove the existence and uniqueness of solutions to the
	fractional wave equation with locally Lipschitz coefficients of at most
	linear growth, driven by general symmetric pure-jump space--time L\'evy
	noise without additional moment assumptions. Their results include
	symmetric $\alpha$-stable noise. In their range of parameters, the wave
	kernel is regular enough for the solution to have positive Sobolev
	regularity, and hence to be bounded in space. This allows them to
	truncate the coefficients using Sobolev norms, as is done in the
	Gaussian setting. For the SHE, this approach does not apply directly:
	the Dirac masses created by nonzero jumps of the solution rule out the
	spatial bounds it requires.

	For globally Lipschitz noise coefficients, Balan \cite{Balan14} constructs predictable mild solutions to a class of SPDEs on bounded domains driven by full $\alpha$-stable space--time noise, including noises with non--symmetric tails. The construction uses Picard
	iteration for equations with truncated noise and then pastes their solutions up to stopping times determined by large jumps. Chong--Dalang--Humeau \cite{CDH} study the path properties of the solution to the SHE driven by general pure-jump L\'evy noise, including noise with infinite variance. On a bounded interval, they establish modifications with $\H_{-r}$-valued c\`adl\`ag paths
	for every $r>1/2$, and investigate spatial regularity at fixed times and temporal regularity at fixed spatial points.
	
	The closest result is due to Wang--Yan--Zhou \cite{WYZ}, who prove weak existence for the SHE on a bounded interval with zero drift, a continuous noise coefficient of at most linear growth, and truncated $\alpha$-stable noise. Their results are formulated in spaces of measures and functions. Here we construct weak solutions with $\H_{-r}$-valued c\`adl\`ag paths for every $r>1/2$, where $\H_{-r}$ is a negative fractional Sobolev space. We treat the full symmetric $\alpha$-stable noise and allow both the drift and noise coefficients to be continuous with at most linear growth. We also construct a predictable mild solution with the same distribution-valued state as the corresponding weak solution.
	
	We prove weak existence of a solution for the full symmetric stable noise by combining compactness in negative Sobolev spaces with convergence of the space--time fields. This allows us to remove the noise cutoff and pass from $b_n$ and $\sigma_n$ to $b$ and $\sigma$, without requiring
	uniform moment bounds for the full noise. Our approximation arguments share some initial ideas with Wang--Yan--Zhou \cite{WYZ}, while our tightness arguments are related to those of Chong--Delerue \cite{ChongDelerue20}. To identify the limit, we develop complex exponential martingale identities and reconstruct the noise using ideas similar to those of Mytnik \cite{Mytnik02}.
	
	\subsection{Setting and solution concepts}
	
	Let $T>0$, $D=(0,1)$, and $\calZ_t=(0,t]\times D$ for
	$t\in[0,T]$. On a stochastic basis
	$\bbB=(\Omega,\Fil,(\Fil_t)_{t\ge0},\P)$ satisfying the usual conditions,
	we study the SHE
	\begin{equation}\label{eq:SHE}
		\partial_tu=\tfrac12\partial_x^2u+b(u)+\sigma(u)\,\dot L
		\quad\text{on }[0,T]\times(0,1),\qquad u(t,0)=u(t,1)=0,
	\end{equation}
	with deterministic initial condition $u(0,\cdot)=u_0\in \L^\infty(D)$.
	Whenever $u_0\in\L^\infty(D)$, we use a bounded Borel representative satisfying
	$|u_0(x)|\le\|u_0\|_{\L^\infty(D)}$ for every $x\in D$.
	The coefficients $b,\sigma:\R\to\R$ are continuous functions with
	at most linear growth: for some constant $\Cgrow>0$,
	\begin{equation}\label{eq:linear-growth}
		|b(\xi)|+|\sigma(\xi)|\le\Cgrow(1+|\xi|),\qquad \xi\in\R.
	\end{equation}
	The forcing $\dot L$ denotes \emph{symmetric $\alpha$-stable space--time
		L\'evy white noise}, where $1<\alpha<2$. To describe it, let $N$ be an
	$(\Fil_t)$-Poisson random measure on
	$(0,\infty)\times D\times(\R\setminus\{0\})$ with intensity
	\[
	\m(\dd t,\dd x,\dd z)=\dd t\,\dd x\,\nu_\alpha(\dd z),\qquad
	\nu_\alpha(\dd z)=c_\alpha|z|^{-1-\alpha}\dd z,
	\quad c_\alpha>0,
	\]
	and write $\Nt=N-\m$ for its compensated
	version. The associated independently scattered random measure $L$ has
	the \emph{L\'evy--It\^o decomposition}
	\begin{equation}\label{eq:Levy-Ito}
		L(\dd t,\dd x)
		=\int_{0<|z|\le1}z\,\Nt(\dd t,\dd x,\dd z)
		+\int_{|z|>1}z\,N(\dd t,\dd x,\dd z).
	\end{equation}
	
	The noise $L$ is a \emph{homogeneous L\'evy basis} with zero drift and zero Gaussian
	part and jump intensity $\m$, satisfying the structural assumptions
	(1)--(7) in \cite{ChongVolterra}*{Section~2} with respect to the filtration
	$(\Fil_t)$ and the space--time predictable $\sigma$-field defined below.
	Let $\calP_0(\bbB)$ be the predictable $\sigma$-field of
	$(\Fil_t)$ on $\Omega\times[0,T]$, and let
	$\calP(\bbB):=\calP_0(\bbB)\otimes\calB(D)$.
	We call $\calP(\bbB)$ the predictable $\sigma$-field on
	$\Omega\times[0,T]\times D$. As usual, we also write
	$H\in\calP(\bbB)$ when a real-valued field $H$ on
	$\Omega\times[0,T]\times D$ is $\calP(\bbB)$-measurable.
	We use this notation for each stochastic basis; in particular,
	$\calP(\widetilde{\bbB})$ denotes the predictable fields on
	the final basis $\widetilde{\bbB}$ constructed in Section \ref{sec:main-proof}.
	
	Throughout, we write $\bbB=(\Omega,\Fil,(\Fil_t)_{t\ge0},\P)$
	for a stochastic basis: $(\Omega,\Fil,\P)$ is a probability space
	and $(\Fil_t)_{t\ge0}$ is a filtration satisfying the usual conditions.
	We write $\calB(E)$ for the Borel $\sigma$-field generated by
	the topology on $E$ understood from the context, in particular the
	usual Euclidean topology or the norm topology when $E$ is a normed space.
	
	For a measure space $(E,\calA,\mu)$ and $1\le p\le\infty$,
	we write $\L^p(E,\mu)$ for the usual real-valued Lebesgue space, with norm
	$\|\cdot\|_{\L^p(E,\mu)}$, omitting $\calA$ when understood.
	For a separable Banach space $V$, we write $\L^p(E,\mu;V)$
	for the space of measurable maps $f:E\to V$, identified up to
	$\mu$-a.e. equality, such that $\|f(\cdot)\|_V\in\L^p(E,\mu)$.
	If $\mu$ is random, this definition is applied separately for each $\omega$.
	
	We simplify the notation for the spaces used throughout the paper:
	on a Euclidean domain $E$, we write
	$\L^p(E):=\L^p(E,\operatorname{Leb})$, where $\operatorname{Leb}$
	denotes Lebesgue measure on $E$; in particular,
	$\L^p(D)$ and $\L^p((0,T)\times D)$
	(also denoted by $\L^p(\calZ_T)$);
	for random variables, we write
	$\L^p(\Omega):=\L^p(\Omega,\P)$.
	The same notation is used for complex-valued variables when needed.
	Finally, $\L^0(\Omega):=\L^0(\Omega,\Fil,\P)$ denotes
	the finite measurable random variables modulo $\P$-almost-sure
	equality. We use the notation
	\[
	\|X\|_{\L^0(\Omega)}:=\E[|X|\wedge1].
	\]
	The associated metric $d(X,Y):=\|X-Y\|_{\L^0(\Omega)}$
	is complete and induces convergence in probability.
	
	We use the same notation on the probability spaces
	$(\overline\Omega,\overline\Fil,\overline\P)$ and
	$(\widetilde\Omega,\widetilde\Fil,\widetilde\P)$ introduced in
	Sections~\ref{subsec:joint-convergence}
	and~\ref{subsec:weighted-stochastic-integration}.
	
	We follow the Daniell approach to stochastic integration developed by
	Bichteler, in the random-measure framework of Bichteler--Jacod
	\cite{BichtelerJacod83}. We use the notation of
	\cite{ChongKlueppelberg}*{Section~2}.
	Let $\calS_L$ denote the bounded simple
	$\calP(\bbB)$-measurable integrands on $\calZ_T$,
	whose integrals are defined by linearity from the L\'evy basis;
	see \cite{CDH}*{Appendix~A}.
	For $H\in\calP(\bbB)$, the \emph{Daniell mean} is
	\[
	\|H\|_L
	:=\sup_{\substack{S\in\calS_L\\ |S|\le |H|}}
	\left\|\int_{\calZ_T}S(s,y)L(\dd s,\dd y)\right\|_{\L^0(\Omega)}.
	\]
	We call $H$ \emph{$\dot L$-integrable} if some sequence $S_n\in\calS_L$ satisfies
	$\|H-S_n\|_L\to0$; its integral is then the limit in probability
	of the integrals of $S_n$. For $\dot L$-integrable $H$, we use the c\`adl\`ag
	version of $t\mapsto\int_{\calZ_t}H(s,y)L(\dd s,\dd y)$. A general criterion for $\dot L$-integrability is given in
	\cite{ChongKlueppelberg}*{Theorem~4.1 and Remark~4.4}.
	For the symmetric $\alpha$-stable noise $L$ considered here,
	a predictable field $H$ is $\dot L$-integrable if and only if
	\[
	\int_{\calZ_T}|H(s,y)|^\alpha\dd s\dd y<\infty,
	\qquad\P\text{-a.s.}
	\]
	Balan \cite{Balan14}*{Sections~2 and~4} gives a detailed construction of $\alpha$-stable L\'evy noise and develops its stochastic integration theory for random fields.
	
	The family
	$\scrL=\{\scrL(A):A\in\calB([0,T]\times D)\}$,
	defined by
	\[
	\scrL(A):=\int_{\calZ_T}\mathbf1_A(s,y)L(\dd s,\dd y),
	\]
	is an independently scattered symmetric $\alpha$-stable random
	measure on $[0,T]\times D$ in the sense of Rajput and Rosi\'nski \cite{RR89}. The law of $\scrL(A)$ is characterized by
	\begin{equation}\label{charact-L}
		\E[e^{i\theta  \scrL(A)}]=\exp\bigl(-\operatorname{Leb}(A)|\theta|^\alpha\bigr),
		\qquad \theta\in\R.
	\end{equation}
	Thus $\scrL$ is a symmetric $\alpha$-stable (S$\alpha$S) random measure with Lebesgue control measure, in the sense of \cite{ST94}*{Definition~3.3.1}.
	The rate of jumps larger than $R$ in absolute value, per unit time
	and space, is
	\[
	\nu_\alpha(\{|z|>R\})=\frac{2c_\alpha}{\alpha}R^{-\alpha},
	\qquad R>0,
	\qquad c_\alpha:=\frac{\Gamma(1+\alpha)\sin(\pi\alpha/2)}{\pi},
	\]
	with $c_\alpha$ corresponding to the normalization above.
	For $b=0$ and $\sigma(u)=u$, Balan and Jim\'enez \cite{BJseries}
	studied the SHE on $\R^d$ driven by symmetric
	$\alpha$-stable space--time L\'evy noise and obtained series expansions
	of mild solutions.
	
	Let $e_k(x)=\sqrt2\sin(k\pi x)$ and $\lambda_k=k^2\pi^2/2$,
	$k\ge1$, be the eigenfunctions and eigenvalues of
	$-\tfrac12\partial_x^2$ with Dirichlet boundary conditions. We write $\varphi''$ for the second spatial derivative $\partial_x^2\varphi$. Set
	$\calE=\operatorname{span}\{e_k:k\ge1\}$. For $r\in\R$, the
	Dirichlet fractional Sobolev space $\H_r$ is the completion of
	$\calE$ under the norm
	\[
	\|v\|_{\H_r}^2=\sum_{k\ge1}\lambda_k^r|\langle v,e_k\rangle|^2.
	\]
	In particular, $\H_0=\L^2(D)$ and $\H_{-r}$ is the dual of $\H_r$
	with respect to the $\L^2(D)$ pairing.
	
	The associated Dirichlet heat kernel has the spectral representation
	\begin{equation}\label{eq:heat-kernel}
		\rho_t(x,y)=\sum_{k\ge1}e^{-\lambda_kt}e_k(x)e_k(y),
		\qquad t>0,\quad x,y\in D.
	\end{equation}
	The corresponding heat semigroup on $\L^2(D)$ is denoted by $(P_t)_{t\ge0}$, where
	\[
	P_tf(x)=\int_D\rho_t(x,y)f(y)\dd y,\qquad t>0,
	\qquad P_0f=f.
	\]
	For $v\in \H_r$, the same semigroup is defined by
	$P_tv=\sum_{k\ge1}e^{-\lambda_kt}\langle v,e_k\rangle e_k$;
	this series defines a smooth function on $D$ for $t>0$.
	We impose the zero boundary condition by using the Dirichlet heat
	semigroup $P_t$.
	
	For $r\in\R$, let $\Path{\H_r}$ denote the space of functions from
	$[0,T]$ to $\H_r$ that are right-continuous and have left limits,
	equipped with the Skorokhod $\mathrm{J}_1$ topology.
	
	We follow the approach of \cite{Hughes}*{Definition~1.1 and Remark~1.2}
	by distinguishing a predictable field from a state process,
	here taking values in $\H_r$, and identifying them through test functions.
	The following definitions are stated on a given stochastic basis; in the
	existence theorem, the basis and noise are part of the construction;
	see also \cite{WYZ}*{Definition~2.2}.
	
	\pagebreak[3]
	\begin{definition}[Weak solution in $\H_r$]\label{def:weak-solution}
		A field $u\in\calP(\bbB)$ is a weak solution of \eqref{eq:SHE}
		in $\H_r$, for $r\in\R$ and deterministic initial state $u_0\in \H_r$,
		if the following conditions hold on a stochastic basis $\bbB$
		carrying the noise $L$ of \eqref{eq:Levy-Ito}.
		\begin{enumerate}[label=(\roman*)]
			\item There exists an adapted process $(u(t))_{t\in[0,T]}$ with
			$(u(t))_{t\in[0,T]}\in\Path{\H_r}$ and $u(0)=u_0$, $\P$-a.s., such that
			\[
			\P\!\left(
			\langle u(t),\varphi\rangle
			=\langle u(t,\cdot),\varphi\rangle
			\right)=1,
			\qquad t\in(0,T],\quad \varphi\in\calE.
			\]
			
			\item For every $\varphi\in\calE$, with probability one, the
			following identity holds for all $t\in[0,T]$:
			\begingroup
			\medmuskip=3mu plus 1mu minus 1mu
			\begin{equation}\label{eq:weak-form}
				\langle u(t),\varphi\rangle
				= \langle u_0,\varphi\rangle
				+\tfrac12\int_0^t\langle u(s),\varphi''\rangle\dd s
				+\int_{\calZ_t} b(u(s,y))\varphi(y)\dd s\dd y
				+\int_{\calZ_t}\varphi(y)\sigma(u(s,y))L(\dd s,\dd y).
			\end{equation}
			\endgroup
			\item The field satisfies
			\begin{equation}\label{eq:weak-integrability}
				\int_{\calZ_T}
				\bigl(|u(s,y)|+|b(u(s,y))|+|\sigma(u(s,y))|^\alpha\bigr)\dd s\dd y<\infty,
				\qquad \P\text{-a.s.}
			\end{equation}
		\end{enumerate}
	\end{definition}

	Here $u(t,x)$ denotes the predictable field $u\in\calP(\bbB)$, and $u(t)$ is the $\H_r$-valued state. The notation $u(t,x)$ refers to the jointly measurable map $(\omega,t,x)\mapsto u(\omega,t,x)$, rather than to pointwise
	evaluation of the distribution $u(t)$ at $x$. In (i), $\langle u(t,\cdot),\varphi\rangle$ denotes the finite Lebesgue integral $\int_D u(t,x)\varphi(x)\dd x$. For each $t\in(0,T]$ and $\varphi\in\calE$, this integral equals the dual pairing $\langle u(t),\varphi\rangle$ on an event of probability one that may depend on $t$ and $\varphi$. This identification alone does not imply that the field $u$ satisfies the mild formulation. On the other hand, since $\varphi,\varphi''\in\calE
	\subset \H_{-r}$ and $u$ has c\`adl\`ag state paths, the first time integral in \eqref{eq:weak-form} is well defined. Condition~\eqref{eq:weak-integrability} ensures the same for the drift and stochastic integrals.
	
	We write $\operatorname{Leb}\text{-a.e.}$ and
	$(\P\otimes\operatorname{Leb})\text{-a.e.}$ for almost everywhere
	with respect to Lebesgue measure and its product with $\P$, respectively.
	For jointly measurable fields $X$ and $Y$ on the same stochastic basis,
	$Y$ is a \emph{modification} of $X$ if
	\[
	\P\bigl(Y(t,x)=X(t,x)\bigr)=1
	\qquad\text{for every fixed }(t,x)\in[0,T]\times D.
	\]
	In this paper, we call $Y$ a \emph{weak modification} of $X$ if
	\[
	\P\!\left(Y(t,x)=X(t,x)
	\ \text{for }\operatorname{Leb}\text{-a.e. }(t,x)\in\calZ_T\right)=1.
	\]
	Equivalently, $Y=X$, $(\P\otimes\operatorname{Leb})\text{-a.e.}$
	on $\Omega\times\calZ_T$.

	\begin{definition}[Mild solution]\label{def:mild-solution}
		A predictable field $u\in\calP(\bbB)$ is a mild solution of
		\eqref{eq:SHE} if, for each fixed
		$(t,x)\in(0,T]\times D$, the integrals below are well defined and
		\begin{equation}\label{eq:mild}
			u(t,x)=P_tu_0(x)
			+\int_{\calZ_t}\rho_{t-s}(x,y)b(u(s,y))\dd s\dd y
			+\int_{\calZ_t}\rho_{t-s}(x,y)\sigma(u(s,y))L(\dd s,\dd y),
			\quad\P\text{-a.s.}
		\end{equation}
		We call a field $v \in \calP (\bbB)$ a density of an $\H_{-r}$-valued
		process $\widetilde u$ if, for each fixed $t\in(0,T]$,
		$\langle\widetilde u(t),\varphi\rangle
		=\int_D v(t,x)\varphi(x)\dd x$ for all $\varphi\in\calE$,
		$\P$-a.s. If $v$ also satisfies the mild
		formulation, we call it a mild density of $\widetilde u$.
	\end{definition}
	
	\subsection{Main results and proof strategy}
	
	The usual existence results for SPDEs start with a given stochastic
	basis $\bbB$ and driving noise $L$, from which a solution is
	constructed, for example by Picard iteration when the coefficients
	are globally Lipschitz. Our proof works the other way around:
	we first use tightness and Skorokhod's representation theorem to obtain
	a limit on a new probability space, where the stochastic basis and
	driving noise have not yet been defined. We then construct a stochastic
	basis and a driving noise with respect to which this limit satisfies
	the SHE.
	
	The following theorem gives 
	weak existence for the full symmetric stable
	noise under the assumptions stated above.
	
	\begin{theorem}\label{thm:weak-existence}
		Let $1<\alpha<2$, $u_0\in \L^\infty(D)$ be deterministic, and
		$b,\sigma:\R\to\R$ be continuous and satisfy \eqref{eq:linear-growth}.
		For every $r>\tfrac12$, there exists a weak solution $\widetilde u$
		in $\H_{-r}$ of \eqref{eq:SHE}, defined on a stochastic basis
		$\widetilde{\bbB}$ satisfying the usual conditions and carrying
		a symmetric $\alpha$-stable space--time L\'evy basis $\widetilde L$
		as in \eqref{eq:Levy-Ito}, in the sense of
		Definition~\ref{def:weak-solution}.
	\end{theorem}
	
	The next theorem identifies a predictable mild density of the
	solution constructed in Theorem~\ref{thm:weak-existence}.
	\begin{theorem}\label{thm:mild-existence}
		Under the assumptions of Theorem~\ref{thm:weak-existence}, fix $r>1/2$
		and let $\widetilde u$ be the weak solution constructed there.
		Then $\widetilde u$ has a mild density $v\in\calP(\widetilde{\bbB})$ on the same basis $\widetilde{\bbB}$ and driven by the same noise $\widetilde L$.
	\end{theorem}

	\noindent\emph{Organization of the paper.} The proof of Theorem~\ref{thm:weak-existence} spans Sections~\ref{sec:preliminaries} and~\ref{sec:main-proof}. In Section~\ref{sec:preliminaries}, we approximate the coefficients by globally Lipschitz functions, truncate the noise, and obtain bounds that do not depend on the Lipschitz constants. We construct the full-noise approximation by pasting together solutions with truncated noise. We then prove joint tightness of the state paths and the space--time fields, and use the Skorokhod representation theorem to obtain an almost surely convergent sequence on a new probability space.

	In Section~\ref{sec:main-proof}, we identify the limit through exponential martingales, without using a second moment of the stable noise. We then reconstruct the driving noise from the limiting martingale identities, adapting the characteristic-function and division arguments of \cite{Mytnik02}*{Proposition~4.1, (4.2), and Lemma~5.7} to symmetric noise. In this way we do not need a representation with the same noise for all approximations, which cannot be guaranteed in general; see Ondrej\'at--Seidler \cite{OndrejatSeidler25}.
	
	In Section~\ref{sec:mild-representation}, we prove Theorem~\ref{thm:mild-existence}. We first use the truncated estimates to show that the heat-kernel integrals are finite at each point. We then derive a regularized mild identity from finite spectral sums and remove the regularization at each fixed space--time point.
	
	Appendix~\ref{app:constants} gives explicit constants for the estimates, and Appendix~\ref{app:predictable-representatives} constructs the predictable representative used throughout the proof.
	
	\medskip 
	\noindent \emph{Open problems and future directions.}
	Continuity and linear growth alone do not give uniqueness in general: already for $\sigma\equiv0$, \eqref{eq:SHE} with $b(z)=\sqrt{z_+}$ and $u_0=0$ has solutions other than $u\equiv0$, see \cite{LRS17}, and
	shifting such a solution in time gives infinitely many different solutions. In the Gaussian setting on $\R$, pathwise uniqueness holds among continuous solutions under linear growth assumptions when $b$ is globally Lipschitz and $\sigma$ is H\"older continuous of order greater than $3/4$; see \cite{MP11}*{Theorem~1.2}. For stable noise, the uniqueness problem is open even for globally Lipschitz $\sigma$ when the domain is unbounded. On $\R^d$, Chong \cite{Chong17} constructs a solution to the stochastic heat equation with globally Lipschitz $\sigma$, driven by heavy-tailed noise that includes $\alpha$-stable noise with  $\alpha < 1 +2/d$, but uniqueness is not known merely under the global Lipschitz condition; the solution obtained there is the limit of the solutions with truncated large jumps. On $\R^d$, large jumps occur in every time interval so the stopping-time argument that gives existence and uniqueness on bounded domains (see \cite{Balan14}, \cite{CDH}*{Proposition~2.1} and Section~\ref{subsec:full-noise-approximations}) is not available. A related question is strong existence: whether \eqref{eq:SHE} admits a solution determined by the given noise $L$ and initial condition on a given stochastic basis. The Yamada--Watanabe theorem of Kurtz \cite{Kurtz07}*{Theorem~3.14} would yield this from Theorem~\ref{thm:weak-existence} and pathwise uniqueness, once its hypotheses are verified for the solution constructed here.
	
	Another direction concerns coefficients with superlinear growth. In the Gaussian case, this problem has been widely studied: global solutions exist when $| b(u)| = O( |u|\log|u|)$ and $\sigma(u)=o(|u|(\log|u|)^{1/4})$ \cite{DKZ19}, and Osgood-type conditions on the drift have been used to study global existence and finite-time blow-up \cites{FN21,CFHS25}. These results rely on estimates of the solution in the supremum norm. For stable noise, each jump adds a Dirac mass to the state, so the supremum of the solution in space cannot be bounded uniformly in time. In addition, the moment bounds in Section~\ref{sec:preliminaries} use the linear growth condition~\eqref{eq:linear-growth}. We leave the extension of our existence result to coefficients with superlinear growth for future work.
	
	Our proof is restricted to one space dimension. For the noise with jumps
	truncated at size $K$, our second-moment estimates require the heat
	kernel to be square integrable in space and time, which holds only in dimension $d=1$. These second moments are used at several steps of the proof. In higher dimensions, the heat kernel remains $\alpha$-integrable when $\alpha < 1 + 2/d$, which is the condition under which Chong \cite{Chong17} constructs solutions on $\R^d$. In this case, one would have to work with moments of order $q$ with $\alpha < q < 1 + 2/d$ instead of second moments, and with negative Sobolev spaces of order larger than $d/2$, where Dirac masses live in dimension $d$. We leave the extension to higher spatial dimensions for future work.
	
	\section{Approximation and compactness}\label{sec:preliminaries}
	\label{subsec:proof-weak-existence}
	
	This section begins the proof of Theorem~\ref{thm:weak-existence}.
	We first construct Lipschitz approximations and derive estimates that
	do not depend on their Lipschitz constants. We then remove the noise
	cutoff and pass to a joint limit of the state paths and random fields.
	
	\subsection{Fixed-cutoff solutions and estimates}
	\label{subsec:fixed-cutoff}
	
	The estimates in this subsection are standard; we include them to
	track the dependence of the constants on $K$ and to show that they do
	not depend on $n$.
	
	For $K\ge1$, define
	\[
	\LK(\dd s,\dd y)
	:=\int_{0<|z|\le K}z\,\Nt(\dd s,\dd y,\dd z),
	\qquad
	m_q(K):=\int_{0<|z|\le K}|z|^q\nu_\alpha(\dd z)<\infty
	\quad(q>\alpha).
	\]
	As in \cite{CDH}*{Proposition~2.1}, define
	\[
	\tauK:=\inf\bigl\{t\in[0,T]:
	N\bigl([0,t]\times D\times[-K,K]^c\bigr)\ne0\bigr\},
	\qquad K\in\N,
	\]
	with $\inf\varnothing=\infty$. Then $(\tauK)_{K\ge1}$ is an increasing
	sequence of stopping times such that, $\P$-a.s., $\tauK>0$ for every
	$K\ge1$ and $\tauK=\infty$ for all sufficiently large $K$.
	
	Direct integration gives the truncated moments and the large-jump rate:
	\begin{equation}\label{eq:noise-constants}
		m_p(K)=\frac{2c_\alpha}{p-\alpha}K^{p-\alpha}\quad(p>\alpha),
		\qquad
		\nu_\alpha(\{|z|>K\})
		=\frac{2c_\alpha}{\alpha}K^{-\alpha}.
	\end{equation}
	By symmetry, $L$ and $\LK$ agree before $\tauK$. By the local property
	\cite{BichtelerJacod83}*{p.~5, (2.22)} and $\P(\tauK=t)=0$, for every
	fixed $t\in[0,T]$ and predictable field $H$ satisfying
	$\int_{\calZ_T}|H(s,y)|^\alpha\dd s\dd y<\infty$, $\P$-a.s., both
	stochastic integrals are defined and
	\begin{equation}\label{eq:noise-local-property}
		\mathbf1_{\{t\le\tauK\}}\int_{\calZ_t}H(s,y)L(\dd s,\dd y)
		=\mathbf1_{\{t\le\tauK\}}\int_{\calZ_t}H(s,y)\LK(\dd s,\dd y),
		\qquad\P\text{-a.s.}
	\end{equation}
	
	We next approximate the coefficients by globally Lipschitz functions.
	Fix a nonnegative $\varrho\in C_c^\infty((-1,1))$ with
	$\int_\R\varrho(\xi)\dd\xi=1$. For $n\ge1$, set
	\[
	\varrho_{1/n}(\xi):=n\varrho(n\xi),
	\qquad \pi_n(\xi):=\max\{-n,\min\{\xi,n\}\},
	\]
	and define, with convolution on $\R$,
	\begin{equation}\label{eq:approx-coefficients}
		b_n:=(b\circ \pi_n)*\varrho_{1/n},
		\qquad \sigma_n:=(\sigma\circ \pi_n)*\varrho_{1/n}.
	\end{equation}
	Continuity of $b$ and $\sigma$ gives $b_n\to b$ and $\sigma_n\to\sigma$
	locally uniformly. By \eqref{eq:linear-growth},
	\begin{equation}\label{eq:approx-growth}
		|b_n(\xi)|+|\sigma_n(\xi)|
		\le\Cgrow(2+|\xi|)\le2\Cgrow(1+|\xi|),
		\qquad \xi\in\R,\quad n\ge1.
	\end{equation}
	Since $b\circ\pi_n$ and $\sigma\circ\pi_n$ are bounded, their mollifications are globally
	Lipschitz, with a common Lipschitz constant
	\[
	\Lambda_n:=\Cgrow(1+n)n\|\varrho'\|_{\L^1(\R)}.
	\]
	Thus, the growth bound is independent of $n$, although $\Lambda_n$ may grow with $n$.

	For $n\ge1$ and $K\ge1$, we consider the truncated equation in mild form:
	for each fixed $(t,x)\in[0,T]\times D$,
	\begin{equation}\label{eq:mildK}
		\begin{aligned}
			u_n^K(t,x)
			&=P_tu_0(x)
			+\int_{\calZ_t}\rho_{t-s}(x,y)b_n\bigl(u_n^K(s,y)\bigr)\dd s\dd y\\
			&\quad+\int_{\calZ_t}\rho_{t-s}(x,y)
			\sigma_n\bigl(u_n^K(s,y)\bigr)\LK(\dd s,\dd y),
			\qquad\P\text{-a.s.}
		\end{aligned}
	\end{equation}
	
	For each fixed $n$ and $K$, the global Lipschitz continuity of
	$b_n$ and $\sigma_n$, together with $m_2(K)<\infty$, gives by
	standard Picard iteration a unique predictable mild solution
	of \eqref{eq:mildK} in the class
	\[
	\sup_{t\in[0,T],\,x\in D}\E|u_n^K(t,x)|^2<\infty.
	\]
	The estimates below are uniform in $n$ and independent of the Lipschitz constants $\Lambda_n$.

	\begin{lemma}\label{lem1}
		Let $\alpha<q\le2$ and $K\ge1$. Any family of predictable mild solutions
		$u_n^K$ of \eqref{eq:mildK} satisfies
		\begin{equation}\label{eq:uniform-moment-bound}
			\sup_{n\ge1}\ \sup_{(t,x)\in[0,T]\times D}
			\E\bigl[|u_n^K(t,x)|^q\bigr]\le\Cbound{q}{K},
		\end{equation}
		where $\Cbound{q}{K}$ depends only on
		$q,T,\Cgrow,\|u_0\|_{\L^\infty(D)}$ and $m_q(K)$, and is independent of
		$n$ and $\Lambda_n$. An explicit choice is given in
		Appendix~\ref{app:constants}, \eqref{eq:C1-explicit}. The constant depends on $\alpha$ and $K$ only through $m_q(K)$, which tends to infinity as $K\to\infty$ or $q\downarrow\alpha$.
	\end{lemma}

	\begin{proof}
		The moment argument is standard; see Balan \cite{Balan14}*{Sections~5.2 and~6.1}
		and Chong--Dalang--Humeau \cite{CDH}*{Proposition~2.7} for related estimates.
		We give the calculation to track the dependence of $\Cbound{q}{K}$ and,
		in particular, its independence of $n$.
		
		Using the notation and martingale constant $B_q$ of Appendix~\ref{app:constants},
		fix $n$ and $K$, and put
		\[
		F_n(t):=1+\sup_{x\in D}\E|u_n^K(t,x)|^q,
		\qquad \beta_q:=\frac{q-1}{2}=1-\eta_q.
		\]
		For fixed $n$ and $K$, boundedness of the coefficients and continuity
		of the heat kernel give mean-square continuity of
		$x\mapsto u_n^K(t,x)$ for every $t>0$. Since $q\le2$, the function
		$x\mapsto\E|u_n^K(t,x)|^q$ is continuous. Its supremum over $D$
		therefore equals its supremum over $D\cap\bbQ$, which makes
		$F_n$ measurable on $(0,T]$. At $t=0$, it is a deterministic value.
		
		For fixed $n$, boundedness of $b_n$ and $\sigma_n$ gives
		$\sup_{t\le T}F_n(t)<\infty$.
		Using the sub-Markov property of $P_t$, the bound
		$\sup_x\int_D\rho_t(x,y)^q\dd y\le h_qt^{-\beta_q}$,
		and the compensated-integral estimate of \cite{CDH}*{Lemma~A.2(2)},
		in the form \eqref{eq:compensated-maximal} with constant $B_q$,
		Jensen's inequality and
		\eqref{eq:approx-growth} yield
		\[
		\begin{aligned}
			F_n(t)\le a_q
			&+3^{q-1}2^{q-1}(2\Cgrow)^q
			\bigg[T^{q-1}\int_0^t F_n(s)\dd s\\
			&\hspace{35mm}
			+B_qm_q(K)h_q\int_0^t(t-s)^{-\beta_q}F_n(s)\dd s\bigg].
		\end{aligned}
		\]
		Since $1\le T^{\beta_q}(t-s)^{-\beta_q}$ for $0\le s<t\le T$, this implies
		\[
		F_n(t)\le a_q+d_{q,K}\int_0^t(t-s)^{\eta_q-1}F_n(s)\dd s,
		\]
		with $a_q$ and $d_{q,K}$ as in \eqref{eq:C1-explicit}.
		Iterating the convolution inequality and using the beta integral yields
		\[
		F_n(t)\le a_q\sum_{j=0}^\infty
		\frac{\bigl[d_{q,K}\Gamma(\eta_q)t^{\eta_q}\bigr]^j}
		{\Gamma(1+j\eta_q)}
		=a_qE_{\eta_q}\bigl(d_{q,K}\Gamma(\eta_q)t^{\eta_q}\bigr).
		\]
		Taking the supremum over $t\le T$ and $n$ proves
		\eqref{eq:uniform-moment-bound}. Every constant in this last bound depends
		only on $q,T,\Cgrow,\|u_0\|_{\L^\infty(D)}$ and $m_q(K)$; the Lipschitz
		constants $\Lambda_n$ do not enter the estimate.
	\end{proof}
	
	By Lemma~\ref{lem1} with $q=2$, $u_n^K(t,\cdot)\in \L^2(D)$,
	$\P$-a.s., for each fixed $t\in[0,T]$. For $k\ge1$, write
	$f_{n,k}^K(t):=\langle u_n^K(t,\cdot),e_k\rangle$.
	The same moment bound permits stochastic Fubini
	\cite{CDH}*{Theorem~A.3(1), with $p=2$}; hence, by \eqref{eq:heat-kernel},
	for each fixed $t\in[0,T]$ and $k\ge1$,
	\begin{equation}\label{eq:mode}
		f_{n,k}^K(t)=e^{-\lambda_kt}\langle u_0,e_k\rangle
		+\int_0^t e^{-\lambda_k(t-s)}B_{n,k}^K(s)\dd s+X_{n,k}^K(t),
		\qquad\P\text{-a.s.}
	\end{equation}
	where
	\[
	\begin{aligned}
		B_{n,k}^K(s)&:=\langle b_n(u_n^K(s,\cdot)),e_k\rangle,\\
		M_{n,k}^K(t)&:=\int_{\calZ_t}e_k(y)
		\sigma_n\bigl(u_n^K(s,y)\bigr)\LK(\dd s,\dd y),
	\end{aligned}
	\qquad
	X_{n,k}^K(t):=\int_0^t e^{-\lambda_k(t-s)}\dd M_{n,k}^K(s).
	\]
	We use the c\`adl\`ag version of the square-integrable martingale $M_{n,k}^K$.
	As in the calculation of \cite{CDH}*{(2.10)}, $X_{n,k}^K$ has a c\`adl\`ag
	version, while the initial and drift terms are continuous. We henceforth
	use the resulting c\`adl\`ag modification of $f_{n,k}^K$, keeping the same notation.
	
	Since $e^{-\lambda_k(t-s)}=1-\lambda_k\int_s^t e^{-\lambda_k(t-r)}\dd r$,
	stochastic Fubini gives, with probability one, for all $0\le s\le t\le T$,
	\begin{equation}\label{eq:ibp}
		\begin{aligned}
			X_{n,k}^K(t)-e^{-\lambda_k(t-s)}X_{n,k}^K(s)
			&=M_{n,k}^K(t)-M_{n,k}^K(s)\\
			&\quad-\lambda_k\int_s^t e^{-\lambda_k(t-r)}
			\bigl(M_{n,k}^K(r)-M_{n,k}^K(s)\bigr)\dd r.
		\end{aligned}
	\end{equation}
	Consequently, on the same event, for all $0\le s\le t_1\le T$,
	\begin{equation}\label{eq:sup2}
		\sup_{s\le t\le t_1}
		\bigl|X_{n,k}^K(t)-e^{-\lambda_k(t-s)}X_{n,k}^K(s)\bigr|
		\le2\sup_{s\le t\le t_1}\bigl|M_{n,k}^K(t)-M_{n,k}^K(s)\bigr|.
	\end{equation}
	
	To keep track of the constants in the estimates below, for $\alpha<p\le2$ put
	\[
	J_p^K:=2^{p-1}(2\Cgrow)^p\bigl(1+\Cbound{p}{K}\bigr),
	\qquad D_p^K:=2^{p/2}J_p^K.
	\]
	By \eqref{eq:approx-growth} and Lemma~\ref{lem1},
	\[
	\sup_{n\ge1,\,t\le T,\,y\in D}
	\E\bigl|\sigma_n(u_n^K(t,y))\bigr|^p\le J_p^K,
	\qquad
	\sup_{n,k\ge1,\,t\le T}\E\bigl|B_{n,k}^K(t)\bigr|^p\le D_p^K,
	\]
	where the second bound uses Jensen's inequality, $\operatorname{Leb}(D)=1$ and
	$|e_k|\le\sqrt2$. In particular, for $\alpha<q<q'\le2$, uniformly in
	$n,k\ge1$ and $s\in[0,T]$,
	\begin{equation}\label{eq:src}
		\begin{gathered}
			\E\bigl|B_{n,k}^K(s)\bigr|^q\le D_q^K,\qquad
			\E\int_{\calZ_T}\bigl|\sigma_n(u_n^K(s,y))\bigr|^q\dd s\dd y
			\le TJ_q^K,\\
			\E\int_{\calZ_T}|u_n^K(s,y)|^{q'}\dd s\dd y
			\le T\Cbound{q'}{K}.
		\end{gathered}
	\end{equation}
	The constants $J_p^K$ and $D_p^K$ depend only on
	$p,T,\Cgrow,\|u_0\|_{\L^\infty(D)}$ and $m_p(K)$, and are independent of
	$n,k$ and $\Lambda_n$.
	
	\begin{lemma}\label{lem2}
		Let $\alpha<q\le2$ and $K\ge1$.
		\begin{enumerate}[label=(\alph*),leftmargin=2em]
			\item For every $0\le\eta<1/q$,
			\[
			\sup_{n\ge1}\E\int_0^T\|u_n^K(t,\cdot)\|_{\H_\eta}^q\dd t
			\le C_{1,1}^K(q,\eta).
			\]
			\item For every $r>1/q$, the process $t\mapsto u_n^K(t,\cdot)$ has an
			$\H_{-r}$-valued c\`adl\`ag modification, denoted by $u_n^K(t)$, such that
			\[
			\sup_{n\ge1}\E\sup_{t\le T}\|u_n^K(t)\|_{\H_{-r}}^q
			\le C_{1,2}^K(q,r).
			\]
		\end{enumerate}
		Explicit choices are given in
		Appendix~\ref{app:constants}, \eqref{eq:space-sup-constants}.
		Both constants depend only on $q,T,\Cgrow,\|u_0\|_{\L^\infty(D)}$ and
		$m_q(K)$, together with $\eta$ for $C_{1,1}^K$ and $r$ for $C_{1,2}^K$.
		In particular, they are independent of $n,k$ and $\Lambda_n$.
		They depend on $\alpha,c_\alpha,K$ only through
		$m_q(K)=2c_\alpha K^{q-\alpha}/(q-\alpha)$. The constants $C_{1,1}^K(q,\eta)$ and $C_{1,2}^K(q,r)$ tend to infinity as $\eta q\uparrow1$ and $rq\downarrow1$, respectively.
	\end{lemma}
	
	\begin{proof}
		\emph{(a)} By Jensen's inequality and the bounds defining $D_q^K$,
		\[
		\E\left|\int_0^t e^{-\lambda_k(t-s)}B_{n,k}^K(s)\dd s\right|^q
		\le D_q^K\lambda_k^{-q}.
		\]
		The compensated-integral bound \eqref{eq:compensated-maximal},
		based on \cite{CDH}*{Lemma~A.2(2)}, also gives
		\[
		\E|X_{n,k}^K(t)|^q
		\le B_qm_q(K)D_q^K\int_0^t e^{-q\lambda_k(t-s)}\dd s
		\le\frac{B_qm_q(K)D_q^K}{q\lambda_k}.
		\]
		Since $|\langle u_0,e_k\rangle|\le\|u_0\|_{\L^2(D)}\le\|u_0\|_{\L^\infty(D)}$,
		\eqref{eq:mode} and $\lambda_k^{-q}\le\lambda_1^{1-q}\lambda_k^{-1}$ yield
		\[
		\E\int_0^T|f_{n,k}^K(t)|^q\dd t
		\le\frac{3^{q-1}}{\lambda_k}
		\left[\frac{\|u_0\|_{\L^\infty(D)}^q}{q}
		+TD_q^K\lambda_1^{1-q}+\frac{TB_qm_q(K)D_q^K}{q}\right].
		\]
		Using $(\sum_k a_k)^{q/2}\le\sum_k a_k^{q/2}$ for $a_k\ge0$ and Tonelli,
		we sum this bound with weights $\lambda_k^{\eta q/2}$.
		The identity
		$\sum_{k\ge1}\lambda_k^{\eta q/2-1}
		=\lambda_1^{\eta q/2-1}\zeta(2-\eta q)$ proves (a).
		
		\emph{(b)} By \eqref{eq:sup2} with $s=0$ and the maximal
		compensated-integral bound \eqref{eq:compensated-maximal},
		\[
		\E\sup_{t\le T}|X_{n,k}^K(t)|^q
		\le2^q\E\sup_{t\le T}|M_{n,k}^K(t)|^q
		\le2^qB_qm_q(K)TD_q^K.
		\]
		For the drift term, H\"older's inequality gives
		\[
		\E\sup_{t\le T}
		\left|\int_0^t e^{-\lambda_k(t-s)}B_{n,k}^K(s)\dd s\right|^q
		\le T^{q-1}\E\int_0^T|B_{n,k}^K(s)|^q\dd s
		\le T^qD_q^K.
		\]
		Consequently,
		\[
		\E\sup_{t\le T}|f_{n,k}^K(t)|^q
		\le3^{q-1}\left[\|u_0\|_{\L^\infty(D)}^q+T^qD_q^K
		+2^qB_qm_q(K)TD_q^K\right].
		\]
		Summing with weights $\lambda_k^{-rq/2}$ gives $C_{1,2}^K(q,r)$, since
		$\sum_{k\ge1}\lambda_k^{-rq/2}=\lambda_1^{-rq/2}\zeta(rq)<\infty$.
		The same estimate with the tail sum shows that the c\`adl\`ag partial sums
		$\sum_{k\le m}f_{n,k}^K(t)e_k$ are Cauchy uniformly in time in probability
		in $\H_{-r}$. A subsequence converges uniformly in time almost surely,
		and its limit is c\`adl\`ag. At each fixed time, this limit agrees almost
		surely with $u_n^K(t,\cdot)$ by its Fourier expansion in $\L^2(D)$.
		This constructs the modification in (b); Fatou's lemma gives its stated bound.
	\end{proof}
	
	For $0<\vartheta<1$, $1\le q<\infty$ and $r\in\R$, the fractional Sobolev
	space $W^{\vartheta,q}(0,T;\H_{-r})$ consists of the functions
	$v\in \L^q(0,T;\H_{-r})$ for which the following norm is finite:
	\[
	\|v\|_{W^{\vartheta,q}(0,T;\H_{-r})}^q
	:=\int_0^T\|v(t)\|_{\H_{-r}}^q\dd t
	+\iint_{(0,T)^2}
	\frac{\|v(t)-v(s)\|_{\H_{-r}}^q}{|t-s|^{1+\vartheta q}}\dd s\dd t.
	\]
	
	\begin{lemma}\label{lem:time-increments}
		Let $\alpha<q<q'\le2$, $K\ge1$ and $r>1/q$, and set
		\[
		\gamma:=\min\left\{q-1,\ 1-\frac{q}{q'},\ \frac{rq-1}{2}\right\}>0.
		\]
		For every stopping time $\tau\le T$ and $0<\delta\le1$,
		\begin{equation}\label{eq:stopped-increments}
			\sup_{n\ge1}\E\bigl\|u_n^K((\tau+\delta)\wedge T)-u_n^K(\tau)
			\bigr\|_{\H_{-r}}^q
			\le C_{2,1}^K(q,q',r)\delta^\gamma.
		\end{equation}
		Moreover, for $0<\vartheta<\gamma/q$,
		\begin{equation}\label{eq:fractional-time-bound}
			\sup_{n\ge1}\E\|u_n^K\|_{W^{\vartheta,q}(0,T;\H_{-r})}^q
			\le C_{2,2}^K(q,q',r,\vartheta).
		\end{equation}
		Explicit constants are given in Appendix~\ref{app:constants},
		\eqref{eq:time-constants}. They depend only on
		$q,q',r,T,\Cgrow,\|u_0\|_{\L^\infty(D)}$ and the truncated
		moments $m_q(K)$ and $m_{q'}(K)$, and on $\vartheta$ in the case of
		$C_{2,2}^K$; in particular, they do not depend on $n$, $\tau$ or
		$\delta$. The constant $C_{2,2}^K$ tends to infinity as
		$\vartheta q\uparrow\gamma$.
	\end{lemma}
	
	\begin{proof}
		Use $S_{q,r},V_q^K,\frakS_{q,r,\gamma}$ and $N_{q,q'}^K$
		from Appendix~\ref{app:constants}, and set $\tau':=(\tau+\delta)\wedge T$.
		Equations \eqref{eq:mode}--\eqref{eq:sup2}
		give
		\[
		\begin{aligned}
			|f_{n,k}^K(\tau')-f_{n,k}^K(\tau)|
			&\le\min(1,\lambda_k\delta)|f_{n,k}^K(\tau)|
			+\int_\tau^{\tau'}|B_{n,k}^K(s)|\dd s\\
			&\quad+2\sup_{\tau\le t\le\tau'}|M_{n,k}^K(t)-M_{n,k}^K(\tau)|.
		\end{aligned}
		\]
		The proof of Lemma~\ref{lem2}(b) bounds
		$\E\sup_{t\le T}|f_{n,k}^K(t)|^q$ by $V_q^K$.
		H\"older's inequality bounds the $q$-th moment of the drift term by
		$TD_q^K\delta^{q-1}$. For the martingale term, apply
		\eqref{eq:compensated-stopped} with
		$H(s,y)=e_k(y)\sigma_n(u_n^K(s,y))$.
		This is the stopped maximal form of the compensated-integral estimate
		in \cite{CDH}*{Lemma~A.2(2)}, justified in Appendix~\ref{app:constants}.
		Together with H\"older's inequality on $\Omega\times(0,T)\times D$, it yields
		\[
		\begin{aligned}
			\E\sup_{\tau\le t\le\tau'}|M_{n,k}^K(t)-M_{n,k}^K(\tau)|^q
			&\le2^{q/2}B_qm_q(K)
			\E\int_{\calZ_T}\mathbf1_{\{\tau<s\le\tau'\}}
			|\sigma_n(u_n^K(s,y))|^q\dd s\dd y\\
			&\le N_{q,q'}^K\delta^{1-q/q'}.
		\end{aligned}
		\]
		Indeed, $\E\int_{\calZ_T}\mathbf1_{\{\tau<s\le\tau'\}}\dd s\dd y
		\le\delta$, and the $q'$-moment integral is at most $TJ_{q'}^K$.
		Summing these bounds over $k\ge1$ with weights $\lambda_k^{-rq/2}$ gives
		\[
		\begin{aligned}
			\E\|u_n^K(\tau')-u_n^K(\tau)\|_{\H_{-r}}^q
			\le3^{q-1}\bigg[&V_q^K\sum_{k\ge1}
			\lambda_k^{-rq/2}\min(1,\lambda_k\delta)^q\\
			&+\bigl(TD_q^K\delta^{q-1}
			+2^qN_{q,q'}^K\delta^{1-q/q'}\bigr)S_{q,r}\bigg].
		\end{aligned}
		\]
		Splitting the spectral sum at $\lambda_k=\delta^{-1}$ shows that it is bounded
		by a constant depending only on $q,r$ times
		\[
		\begin{cases}
			\delta^{(rq-1)/2},&1<rq<2q+1,\\
			\delta^q(1+|\log \delta|),&rq=2q+1,\\
			\delta^q,&rq>2q+1.
		\end{cases}
		\]
		Since $\gamma\le(rq-1)/2$ and $\gamma<q$, this proves
		$\frakS_{q,r,\gamma}<\infty$ and \eqref{eq:stopped-increments}.
		The term $2^qC_{1,2}^K(q,r)$ in \eqref{eq:time-constants} also ensures
		$\E\|u_n^K(t)-u_n^K(s)\|_{\H_{-r}}^q
		\le C_{2,1}^K(q,q',r)|t-s|^\gamma$ when $|t-s|>1$.
		Finally, integrating the deterministic-time increment bound and using
		Lemma~\ref{lem2}(b), together with
		\[
		\iint_{(0,T)^2}|t-s|^{\gamma-1-\vartheta q}\dd s\dd t
		=\frac{2T^{1+\gamma-\vartheta q}}
		{(\gamma-\vartheta q)(1+\gamma-\vartheta q)},
		\]
		proves \eqref{eq:fractional-time-bound}.
	\end{proof}
	
	\begin{lemma}\label{lem:truncated-tightness}
		Let $K\ge1$, $\alpha<q<2$ and $r>1/q$. The laws of $u_n^K$, $n\ge1$,
		are tight in $\Path{\H_{-r}}$ with the Skorokhod $\mathrm{J}_1$ topology and in
		$\L^q((0,T)\times D)$ with the norm topology.
	\end{lemma}
	\begin{proof}
		Choose
		\[
		q':=\frac{q+2}{2},\qquad r_0:=\frac{r+1/q}{2},\qquad
		\eta:=\frac1{2q},\qquad\vartheta:=\frac{\gamma}{2q},
		\]
		where $\gamma$ is given by Lemma~\ref{lem:time-increments} for $q,q',r$.
		These choices satisfy $q<q'\le2$, $1/q<r_0<r$, $0<\eta<1/q$
		and $0<\vartheta<\gamma/q$.
		
		\emph{Path space.} The embedding $\H_{-r_0}\hookrightarrow \H_{-r}$ is compact.
		By Lemma~\ref{lem2}(b) and Markov's inequality,
		\begin{equation}\label{eq:compact-containment}
			\sup_{n\ge1}\P\left(\sup_{t\in[0,T]}
			\|u_n^K(t)\|_{\H_{-r_0}}>R\right)
			\le C_{1,2}^K(q,r_0)R^{-q}.
		\end{equation}
		Here the modifications in the two Sobolev spaces agree at all times
		almost surely, by equality at rational times and at $T$, and their
		c\`adl\`ag paths. Since bounded balls of $\H_{-r_0}$ have compact
		closure in $\H_{-r}$, this gives compact containment.
		
		To prove tightness, we use Aldous's criterion \cite{Aldous}*{Theorem~1}
		in the metric-space form of \cite{LalleyTang20}*{Theorem~2.2}. Since
		the latter is stated for paths on $[0,\infty)$, we first extend the
		paths after $T$ by keeping them constant. Some care is needed when we
		come back to $[0,T]$. On $[0,\infty)$, the time changes in the
		$\mathrm{J}_1$ topology may move the point $T$, while on $[0,T]$ they
		must keep $T$ fixed. So, if the limit path jumps at $T$, convergence
		on $[0,\infty)$ does not give convergence of the restrictions to
		$[0,T]$. For this reason, we first show that every subsequential
		limit is almost surely continuous at $T$.
		
		Set $\widehat u_n^K(t):=u_n^K(t\wedge T)$ for all $t\ge0$, and keep the
		filtration equal to $\Fil_T$ after time $T$. Since $\widehat u_n^K(t)=u_n^K(t\wedge T)$, the bound
		\eqref{eq:compact-containment} also holds for $\widehat u_n^K$,
		with the supremum taken over $t\ge0$. Let $\tau$ be a bounded stopping time and
		$0<h\le1$. Then $\tau\wedge T$ is a stopping time bounded by $T$, and
		$\widehat u_n^K(\tau+h)-\widehat u_n^K(\tau)
		=u_n^K(((\tau\wedge T)+h)\wedge T)-u_n^K(\tau\wedge T)$. So
		\eqref{eq:stopped-increments} gives
		\[
		\sup_{n\ge1}\E
		\|\widehat u_n^K(\tau+h)-\widehat u_n^K(\tau)\|_{\H_{-r}}^q
		\le C_{2,1}^K(q,q',r)h^\gamma.
		\]
		This bound, \eqref{eq:compact-containment} and Markov's inequality
		give the two conditions of \cite{LalleyTang20}*{Theorem~2.2}, so the
		laws of $\widehat u_n^K$ are tight in $\bbD([0,\infty);\H_{-r})$ with
		the $\mathrm{J}_1$ topology.
		
		By Prokhorov's theorem \cite{EK}*{Chapter~3, Theorem~2.2}, every
		subsequence of $(\widehat u_n^K)$ has a further subsequence that
		converges in distribution in $\bbD([0,\infty);\H_{-r})$. Let
		$\widehat u^K$ be the limit of such a subsequence, still indexed by
		$n$, defined on some probability space with expectation $\E_*$.
		The paths that are constant on $[T,\infty)$ form a closed set in the
		$\mathrm{J}_1$ topology, so $\widehat u^K$ is also constant on
		$[T,\infty)$, almost surely. In particular, $T+1$ is an almost sure
		continuity time of $\widehat u^K$. Choose deterministic times
		$\delta_j\downarrow0$ with $\delta_j<\min\{1,T\}$ such that each
		$T-\delta_j$ is also an almost sure continuity time of
		$\widehat u^K$. This is possible because a c\`adl\`ag process can jump
		with positive probability at only countably many fixed times.
		Since $\widehat u_n^K(T+1)=u_n^K(T)$, \eqref{eq:stopped-increments}
		with $\tau=T-\delta_j$ and $\delta=\delta_j$ gives
		\begin{equation}\label{eq:increment-at-T}
			\E\|\widehat u_n^K(T+1)-\widehat u_n^K(T-\delta_j)\|_{\H_{-r}}^q
			\le C_{2,1}^K(q,q',r)\delta_j^\gamma.
		\end{equation}
		Since $T+1$ and $T-\delta_j$ are continuity times of the limit, the
		difference $\widehat u_n^K(T+1)-\widehat u_n^K(T-\delta_j)$ converges
		in law to $\widehat u^K(T+1)-\widehat u^K(T-\delta_j)$. The function
		$x\mapsto\|x\|_{\H_{-r}}^q$ is continuous and nonnegative, so the
		expectation of its value at the limit is at most the $\liminf$ of the
		expectations; this follows from \cite{EK}*{Chapter~3, Theorem~3.1(e)},
		applied to the open sets $\{x:\|x\|_{\H_{-r}}^q>a\}$, $a\ge0$, and
		Fatou's lemma. By
		\eqref{eq:increment-at-T},
		\[
		\E_*\|\widehat u^K(T+1)-\widehat u^K(T-\delta_j)\|_{\H_{-r}}^q
		\le C_{2,1}^K(q,q',r)\delta_j^\gamma.
		\]
		Since $\widehat u^K(T+1)=\widehat u^K(T)$ and
		$\widehat u^K(T-\delta_j)\to\widehat u^K(T-)$ as $j\to\infty$,
		Fatou's lemma gives
		$\E_*\|\widehat u^K(T)-\widehat u^K(T-)\|_{\H_{-r}}^q=0$. Thus
		$\widehat u^K$ is almost surely continuous at $T$.
		The restriction map from $\bbD([0,\infty);\H_{-r})$ to
		$\Path{\H_{-r}}$ is continuous at every path that is continuous at
		$T$; this follows from \cite{EK}*{Chapter~3, Proposition~5.2}. The restriction of
		$\widehat u_n^K$ to $[0,T]$ is $u_n^K$. By the continuous mapping
		theorem, $u_n^K$ converges in law in $\Path{\H_{-r}}$ along the same
		subsequence. Hence every subsequence of $(u_n^K)$ has a further
		subsequence that converges in law in $\Path{\H_{-r}}$. Since
		$\Path{\H_{-r}}$ is Polish, Prokhorov's theorem \cite{EK}*{Chapter~3, Theorem~2.2} shows that the laws of
		$u_n^K$ are tight in $\Path{\H_{-r}}$.

		\emph{Function space.} Since
		$\H_\eta\hookrightarrow \L^2(D)\hookrightarrow \H_{-r}$, with the first
		embedding compact, the sets
		\[
		\calB_R:=\left\{v:
		\|v\|_{\L^q(0,T;\H_\eta)}\le R,\quad
		\|v\|_{W^{\vartheta,q}(0,T;\H_{-r})}\le R\right\}
		\]
		are relatively compact in $\L^q(0,T;\L^2(D))$ by
		\cite{Simon}*{Corollary~5, p.~86}, applied with
		$X=\H_\eta$, $B=\L^2(D)$ and $Y=\H_{-r}$.
		Both time-integrability exponents in that corollary are chosen equal to
		$q$, and its fractional order is $\vartheta>0$.
		Since $\operatorname{Leb}(D)=1$ and $q<2$, the embedding
		$\L^q(0,T;\L^2(D))\hookrightarrow \L^q((0,T)\times D)$ is continuous.
		Thus the closure of $\calB_R$ in $\L^q((0,T)\times D)$ is compact.
		Lemmas~\ref{lem2}(a) and~\ref{lem:time-increments} give
		\[
		\sup_{n\ge1}\P(u_n^K\notin\calB_R)
		\le\bigl[C_{1,1}^K(q,\eta)+C_{2,2}^K(q,q',r,\vartheta)\bigr]R^{-q},
		\]
		which proves tightness in $\L^q((0,T)\times D)$.
	\end{proof}
	
	\subsection{Full-noise approximations}
	\label{subsec:full-noise-approximations}
	
	We now remove the cutoff and construct the full-noise approximations.
	For each $n\ge1$ we define the predictable full-noise solution $u_n$ by
	\begin{equation}\label{eq:full-approximation}
		u_n(t,x):=u_n^K(t,x)\quad\text{on }\{t\le\tauK\},
		\qquad\P\text{-a.s.},
	\end{equation}
	for every fixed $(t,x)\in[0,T]\times D$ and integer $K\ge1$.
	
	We use the standard construction by truncation and pasting from
	\cite{Balan14}*{Section~6} and
	\cite{CDH}*{Proposition~2.1 and the proof of Proposition~2.3}.
	This family satisfies the consistency property
	$u_n^{K'}(t,x)=u_n^K(t,x)$, $\P$-a.s.\ on $\{t<\tau_{K'}\}$,
	for $1\le K'<K$ and every fixed $(t,x)$, so that
	\eqref{eq:full-approximation} is well defined.
	For each fixed $n$, the resulting $u_n$ is unique up to modification among
	predictable mild solutions $v$ with coefficients $b_n,\sigma_n$, initial
	condition $u_0$ and noise $L$, satisfying
	\[
	\sup_{(t,x)\in[0,T]\times D}
	\E\bigl[|v(t,x)|^q\ind{t\le\tauK}\bigr]<\infty,
	\qquad \alpha<q\le2,\quad K\ge1.
	\]
	Moreover, \eqref{eq:uniform-moment-bound} gives the uniform estimate
	\[
	\sup_{n\ge1}\ \sup_{(t,x)\in[0,T]\times D}
	\E\bigl[|u_n(t,x)|^q\ind{t\le\tauK}\bigr]
	\le\Cbound{q}{K},\qquad \alpha<q\le2,\quad K\ge1.
	\]
	
	We work with the c\`adl\`ag $\H_{-r}$-valued modification of $u_n$, for $r>1/2$,
	given by the argument of \cite{CDH}*{Theorem~2.5 and Remark~2.4},
	after rescaling the interval and heat operator.
	Indeed, $\sigma_n$ is bounded, while $P_tu_0$ and the convolution with
	the bounded drift $b_n(u_n)$ have continuous $\L^2(D)$-valued paths.
	
	The following proposition is inspired in part by the approach of
	Chong--Delerue \cite{ChongDelerue20}, who consider jointly the random field
	and its c\`adl\`ag modification in a negative Sobolev space.
	
	\begin{proposition}\label{prop:full-tightness}
		Let $\alpha<q<2$ and $r>1/q$. The laws of the pairs $(u_n,u_n)$,
		$n\ge1$, are tight in
		\[
		\Path{\H_{-r}}\times \L^q((0,T)\times D),
		\]
		with the Skorokhod $\mathrm{J}_1$ topology in the first factor and the norm
		topology in the second. In the pair $(u_n,u_n)$, the first coordinate is the state process $t\mapsto u_n(t)$, viewed as an element of $\Path{\H_{-r}}$, and the second coordinate is the predictable field $(t,x)\mapsto u_n(t,x)$, viewed as an element of $\L^q((0,T)\times D)$.
		
	\end{proposition}
	\begin{proof}
		By \eqref{eq:full-approximation}, Fubini and the c\`adl\`ag modifications,
		$u_n=u_n^K$ in $\Path{\H_{-r}}$ and in $\L^q((0,T)\times D)$
		on $\{\tauK=\infty\}$, $\P$-a.s.
		These events increase to a set of probability one, so $(u_n,u_n)$
		takes values in $\Path{\H_{-r}}\times \L^q((0,T)\times D)$.
		
		Fix $\varepsilon>0$ and choose an integer $K\ge1$ such that
		$\P(\tauK\le T)<\varepsilon/2$.
		By Lemma~\ref{lem:truncated-tightness}, there are compact sets
		$\calA\subset\Path{\H_{-r}}$ and
		$\calB\subset \L^q((0,T)\times D)$ such that
		\[
		\sup_{n\ge1}\P(u_n^K\notin\calA)<\frac{\varepsilon}{4},
		\qquad
		\sup_{n\ge1}\P(u_n^K\notin\calB)<\frac{\varepsilon}{4}.
		\]
		Consequently,
		\[
		\sup_{n\ge1}\P\bigl((u_n,u_n)\notin\calA\times\calB\bigr)
		\le \P(\tauK\le T)
		+\sup_{n\ge1}\P(u_n^K\notin\calA)
		+\sup_{n\ge1}\P(u_n^K\notin\calB)
		<\varepsilon.
		\]
		The product $\calA\times\calB$ is compact, proving the claim.
	\end{proof}
	
	\begin{lemma}\label{lem:full-integrability-probability}
		Let $\alpha<q'\le2$. For every $R>0$ and integer $K\ge1$,
		\[
		\sup_{n\ge1}\P\left(\int_{\calZ_T}|u_n(s,y)|^{q'}\dd s\dd y>R\right)
		\le \P(\tauK\le T)+\frac{C_{3,1}^K(q')}{R},
		\qquad C_{3,1}^K(q'):=T\Cbound{q'}{K}.
		\]
		The constant $C_{3,1}^K(q')$ has the dependencies of
		$\Cbound{q'}{K}$ and is independent of $n$ and $\Lambda_n$;
		it is also listed in Appendix~\ref{app:constants}.
		Consequently,
		\[
		\lim_{R\to\infty}\sup_{n\ge1}
		\P\left(\int_{\calZ_T}|u_n(s,y)|^{q'}\dd s\dd y>R\right)=0,
		\]
		and $u_n\in \L^{q'}((0,T)\times D)$, $\P$-a.s., for every $n$.
	\end{lemma}
	\begin{proof}
		On $\{\tauK=\infty\}$, the fields $u_n$ and $u_n^K$ agree
		Leb-a.e.\ on $\calZ_T$, $\P$-a.s.
		Splitting according to this event, then applying Markov's inequality
		and Lemma~\ref{lem1} with exponent $q'$, gives the bound in the lemma.
		For any $\varepsilon>0$, choose $K$ so that
		$\P(\tauK\le T)<\varepsilon/2$, then choose
		$R>2T\Cbound{q'}{K}/\varepsilon$.
		This proves the limit and hence the almost sure finiteness of each integral.
	\end{proof}
	
	\subsection{Joint convergence on a new probability space}
	\label{subsec:joint-convergence}
	
	Fix $r>1/2$ and choose $q$ with $\max\{\alpha,1/r\}<q<2$, so that
	Proposition~\ref{prop:full-tightness} applies. Set
	\[
	\calD:=\Path{\H_{-r}},\qquad
	\calL:=\L^q((0,T)\times D).
	\]
	Both spaces are Polish with the topologies of
	Proposition~\ref{prop:full-tightness}. By that proposition and
	Prokhorov's theorem \cite{EK}*{Chapter~3, Theorem~2.2}, there is a subsequence, which we still denote by
	$n$, along which the laws of $(u_n,u_n)$ converge weakly on
	$\calD\times\calL$. Since $\calD\times\calL$ is
	separable, Skorokhod's representation theorem
	\cite{EK}*{Chapter~3, Theorems~1.8 and~3.1(a)--(b)} then gives a probability
	space $(\overline\Omega,\overline\Fil,\overline\P)$ and random
	elements $(\widetilde u_n,\widetilde v_n)$ and $(\widetilde u,\widetilde v)$
	such that
	\begin{equation}\label{eq:skor}
		\begin{gathered}
			(\widetilde u_n,\widetilde v_n)\overset{d}{=}(u_n,u_n),\qquad n\ge1,\\
			\widetilde u_n\longrightarrow\widetilde u\ \text{in }\calD
			\quad\text{and}\quad
			\widetilde v_n\longrightarrow\widetilde v\ \text{in }\calL,
			\qquad\overline\P\text{-a.s.}
		\end{gathered}
	\end{equation}
	
	We now check that the limit field $\widetilde v$ agrees with the limit
	state $\widetilde u$ when tested against functions in $\calE$.
	We first explain why this agreement is preserved by convergence in
	$\calD\times\calL$.
	Let $\calG$ be the set of pairs $(f,g)$ in this product satisfying
	\[
	\langle f(t),e_k\rangle=\int_D g(t,y)e_k(y)\dd y
	\quad\text{for a.e. }t\in(0,T),\quad k\ge1.
	\]
	Here $f(t)$ is an element of $\H_{-r}$, whereas $g$ is a space--time
	function. In other words, $\calG$ consists of the pairs
	$(f,g)\in\calD\times\calL$ in which $g$ is a density of the
	$\H_{-r}$-valued path $f$, in the sense of
	Definition~\ref{def:mild-solution}, for almost every time instead of
	every fixed time. 
	
	For each $k$, the equality holds for all $t$ outside a set
	$\Xi_k\subset(0,T)$ of Lebesgue measure zero. Since there are
	countably many $k$, the union $\Xi:=\bigcup_k\Xi_k$ also has Lebesgue
	measure zero, and for $t\notin\Xi$ the equality holds for all $k$ at
	the same time. By linearity, it then holds for every
	$\varphi\in\calE$, since each $\varphi\in\calE$ is a finite
	linear combination of the $e_k$.
	
	We now show that $\calG$ is closed. Let $(f_j,g_j)\in\calG$
	with $f_j\to f$ in $\calD$ and $g_j\to g$ in $\calL$.
	Convergence in the $\mathrm{J}_1$ topology implies that
	$\sup_j\sup_{t\le T}\|f_j(t)\|_{\H_{-r}}<\infty$ and that
	$f_j(t)\to f(t)$ in $\H_{-r}$ at every time $t$ where $f$ is continuous,
	that is, for all but countably many $t$. For fixed $k$, the bound
	$|\langle h,e_k\rangle|\le\|h\|_{\H_{-r}}\|e_k\|_{\H_r}$ and dominated
	convergence therefore give
	\[
	\int_0^T|\langle f_j(t)-f(t),e_k\rangle|\dd t\longrightarrow0.
	\]
	For the fields, H\"older's inequality and $\operatorname{Leb}(D)=1$ give
	\[
	\int_0^T\left|\int_D(g_j(t,y)-g(t,y))e_k(y)\dd y\right|\dd t
	\le\sqrt2\,T^{1-1/q}\|g_j-g\|_{\calL}\longrightarrow0.
	\]
	The two pairings agree for each $(f_j,g_j)$ for almost every $t$, so
	their limits in $\L^1(0,T)$ also agree, and the pairings for
	$(f,g)$ coincide for almost every $t$. Thus $(f,g)\in\calG$, and
	$\calG$ is closed. In particular, $\calG$ is a Borel set, so
	events of the form $\{(X,Y)\in\calG\}$ are measurable.
	
	For the original solutions, the pairings of the state and of the field
	with each $e_k$ agree almost surely at each fixed time. By Fubini's
	theorem, they agree for almost every $t$, almost surely, and since
	there are countably many $e_k$, we obtain
	$\P((u_n,u_n)\in\calG)=1$ for every $n$. Since $\calG$ is
	a Borel set and $(\widetilde u_n,\widetilde v_n)$ has the same law as
	$(u_n,u_n)$, it follows that
	$\overline\P((\widetilde u_n,\widetilde v_n)\in\calG)=1$.
	There are countably many $n$, so the intersection of these events with
	the event in \eqref{eq:skor} still has probability one. On this
	intersection, the pair $(\widetilde u,\widetilde v)$ belongs to
	$\calG$ because $\calG$ is closed, that is,
	\[
	\langle\widetilde u(t),\varphi\rangle
	=\int_D\widetilde v(t,y)\varphi(y)\dd y
	\quad\text{for a.e. }t\in(0,T)\text{ and all }\varphi\in\calE.
	\]
	
	The elements of $\calL$ are functions defined up to equality
	Leb-a.e.\ in space--time. We choose versions of $\widetilde v_n$ and
	$\widetilde v$ that are measurable in $(\omega,t,y)$, and we denote
	them by $\widetilde u_n(t,y)$ and $\widetilde u(t,y)$. We keep the
	notation $\widetilde u_n(t)$ and $\widetilde u(t)$ for the
	$\H_{-r}$-valued state processes. So far, the state and the field agree
	for almost every time; the identification at each fixed time is proved
	at the end of Section~\ref{subsec:stable-noise-reconstruction}.
	
	We also need convergence of the fields at almost every space--time point.
	To obtain it, choose a further subsequence, independent of $\omega$ and
	again indexed by $n$, such that
	\[
	\overline\P\bigl(
	\|\widetilde v_n-\widetilde v\|_{\calL}^q>2^{-n}\bigr)<2^{-n}.
	\]
	This is possible because the norms tend to zero almost surely, hence
	in probability. The sum of these probabilities is finite.
	The Borel--Cantelli lemma therefore implies that, almost surely,
	$\|\widetilde v_n-\widetilde v\|_{\calL}^q\le2^{-n}$ for all
	sufficiently large $n$. For each such $\omega$, Tonelli's theorem gives
	\[
	\int_{\calZ_T}\sum_{n\ge1}
	|\widetilde u_n(t,y)-\widetilde u(t,y)|^q\dd t\dd y
	=\sum_{n\ge1}\|\widetilde v_n-\widetilde v\|_{\calL}^q<\infty.
	\]
	The series inside the integral is therefore finite for almost every
	$(t,y)$, so its terms tend to zero. We have proved
	\[
	\widetilde u_n(t,y)\longrightarrow\widetilde u(t,y)
	\quad\text{for Leb-a.e. }(t,y)\in\calZ_T,
	\qquad\overline\P\text{-a.s.}
	\]
	For $\alpha<q'\le2$, define
	\[
	\Theta_{q'}(R):=\sup_{n\ge1}\P\left(
	\int_{\calZ_T}|u_n(s,y)|^{q'}\dd s\dd y>R\right),
	\qquad R>0.
	\]
	Lemma~\ref{lem:full-integrability-probability} gives
	$\Theta_{q'}(R)\to0$ as $R\to\infty$.
	The map $g\mapsto\int_{\calZ_T}|g|^{q'}$, with values in
	$[0,\infty]$, is Borel on $\calL$ (in fact, lower semicontinuous).
	Hence, equality in law in \eqref{eq:skor} gives the same tail bound for
	the fields $\widetilde u_n$. Fatou's lemma, applied first in space--time and then
	to the indicators of the tail events, yields
	\begin{equation}\label{eq:limit-tail}
		\overline\P\left(\int_{\calZ_T}
		|\widetilde u(s,y)|^{q'}\dd s\dd y>R\right)
		\le\Theta_{q'}(R).
	\end{equation}
	Indeed, if the limit integral exceeds $R$, the approximating integrals
	eventually exceed $R$ as well.
	Taking $q'=2$ gives
	\begin{equation}\label{eq:limit-L2}
		\int_{\calZ_T}|\widetilde u(s,y)|^2\dd s\dd y<\infty,
		\qquad\overline\P\text{-a.s.}
	\end{equation}

	\begin{lemma}\label{lem:nonlinear-convergence}
		Let $\alpha<q'\le2$ and $0<\beta<q'$. Let
		$\Psi_n,\Psi:D\times\R\to\R$ be measurable and continuous in their
		second variable. Suppose that, for every $a>0$,
		\[
		\sup_{y\in D}\sup_{|\xi|\le a}|\Psi_n(y,\xi)-\Psi(y,\xi)|\longrightarrow0,
		\qquad
		|\Psi_n(y,\xi)|+|\Psi(y,\xi)|\le C_\Psi(1+|\xi|^\beta),
		\]
		where $C_\Psi\ge0$ is independent of $n,y,\xi$. Then
		\[
		\int_{\calZ_T}
		|\Psi_n(y,\widetilde u_n(s,y))-\Psi(y,\widetilde u(s,y))|\dd s\dd y
		\longrightarrow0\quad\text{in }\overline\P\text{-probability}.
		\]
		In particular, taking $q'=2$, for every fixed $\varphi\in\calE$,
		\begin{align*}
			&\sup_{t\le T}\left|\int_{\calZ_t}
			\bigl(|\sigma_n(\widetilde u_n(s,y))\varphi(y)|^\alpha
			-|\sigma(\widetilde u(s,y))\varphi(y)|^\alpha\bigr)\dd s\dd y\right|
			\longrightarrow0,\\
			&\sup_{t\le T}\left|\int_{\calZ_t}
			\bigl(b_n(\widetilde u_n(s,y))-b(\widetilde u(s,y))\bigr)
			\varphi(y)\dd s\dd y\right|\longrightarrow0
			\quad\text{in }\overline\P\text{-probability}.
		\end{align*}
	\end{lemma}
	
	\begin{proof}
		Write $F_n(s,y)=\Psi_n(y,\widetilde u_n(s,y))$ and
		$F(s,y)=\Psi(y,\widetilde u(s,y))$. Since
		$\widetilde u_n\to\widetilde u$ almost everywhere, $\Psi_n\to\Psi$
		uniformly on bounded sets of $\xi$, and $\Psi$ is continuous in $\xi$,
		we have $F_n\to F$ almost everywhere, almost surely. If $C_\Psi=0$
		there is nothing to prove, so let $C_\Psi>0$.
		
		Let $\pi_a(z)=\max\{-a,\min\{z,a\}\}$ be the function that cuts $z$ at
		level $a$, as in Section~\ref{subsec:fixed-cutoff} with $a$ in place of $n$, and let
		\[
		E_{n,R}:=\left\{\int_{\calZ_T}|\widetilde u_n|^{q'}\dd s\dd y\le R,
		\quad\int_{\calZ_T}|\widetilde u|^{q'}\dd s\dd y\le R\right\}.
		\]
		By \eqref{eq:limit-tail} and equality in law,
		$\overline\P(E_{n,R}^c)\le2\Theta_{q'}(R)$.
		
		Take $a\ge2C_\Psi$. If $|F_n|>a$, the growth bound shows that
		$|\widetilde u_n|\ge1$, so $|F_n|\le2C_\Psi|\widetilde u_n|^\beta$ and
		$|\widetilde u_n|^\beta>a/(2C_\Psi)$. Since $\beta<q'$, this gives
		\[
		|F_n|\ind{|F_n|>a}
		\le(2C_\Psi)^{q'/\beta}a^{1-q'/\beta}|\widetilde u_n|^{q'}.
		\]
		The same bound holds for $F$, with $\widetilde u$ in place of
		$\widetilde u_n$. On $E_{n,R}$ we therefore get
		\[
		\int_{\calZ_T}|F_n-F|\dd s\dd y
		\le\int_{\calZ_T}|\pi_a(F_n)-\pi_a(F)|\dd s\dd y
		+2C_{3,2}\,a^{1-q'/\beta}R,
		\]
		where $C_{3,2}:=(2C_\Psi)^{q'/\beta}$. For fixed $a$, the first
		integral on the right goes to zero almost surely by dominated
		convergence.
		
		Now let $\varepsilon>0$. First choose $R$ so that
		$2\Theta_{q'}(R)<\varepsilon$. Then choose $a\ge2C_\Psi$ so that the
		last term is less than $\varepsilon$. Finally, take $n$ large enough
		so that the first integral is larger than $\varepsilon$ with
		probability less than $\varepsilon$. For such $n$,
		\[
		\overline\P\left(\int_{\calZ_T}|F_n-F|\dd s\dd y>2\varepsilon\right)
		<2\varepsilon,
		\]
		which proves the convergence in probability.
		
		It remains to prove the two limits in the last part of the lemma,
		for $\sigma$ and for $b$. We take $q'=2$. For the limit with
		$\sigma$, take $\Psi_n(y,\xi)=|\sigma_n(\xi)\varphi(y)|^\alpha$ and
		$\beta=\alpha$; for the limit with $b$, take
		$\Psi_n(y,\xi)=b_n(\xi)\varphi(y)$ and $\beta=1$. In both cases,
		$\Psi$ is obtained by replacing $\sigma_n$ and $b_n$ by $\sigma$ and
		$b$. All these coefficients are bounded by $2\Cgrow(1+|\xi|)$, so we
		can take $C_\Psi=2^{2\alpha}\Cgrow^\alpha\|\varphi\|_{\L^\infty(D)}^\alpha$
		for $\sigma$ and $C_\Psi=4\Cgrow\|\varphi\|_{\L^\infty(D)}$ for $b$. Finally,
		in each case the supremum over $t\le T$ is at most the integral of
		the absolute difference over $\calZ_T$.
	\end{proof}
	
	\section{Martingale problem and noise reconstruction}\label{sec:main-proof}
	
	The joint limit does not retain the original driving noise. In this section, we first show that the limit $\widetilde u$ satisfies the exponential martingale identities of
	Proposition~\ref{prop:limit-martingale}, which characterize the stable noise term. We then reconstruct a driving noise on an enlarged stochastic basis.
	
	\subsection{Exponential martingales for the approximations}
	\label{subsec:approximation-martingales}
	
	The full stable noise has no finite second moment, so we use its
	characteristic exponent to identify the noise term in the limit.
	We follow the characteristic-function martingale approach of Mytnik
	\cite{Mytnik02}*{Proposition~4.1, (4.2)}, adapted here to symmetric
	stable noise. Multiplying by the integrating factor $e^{|\theta|^\alpha A_n^\varphi}$
	turns the identity given by It\^o's formula into the exponential local
	martingale $Z_n^{\varphi,\theta}$ in \eqref{eq:exponential-martingale}.
	Both quantities in its exponent, $M_n^\varphi$ and $A_n^\varphi$ below, are functions of the solution $u_n$. After stopping, the resulting
	martingales are bounded. So, by equality in law, these martingale identities also hold for the
	processes $\widetilde u_n$ of \eqref{eq:skor}, and they pass to the limit
	$\widetilde u$. We then use them to reconstruct the driving noise.
	
	On the original stochastic basis, fix $n\ge1$ and
	$\varphi\in\calE$, and set, for $t\in[0,T]$,
	\begin{equation}\label{eq:martingale-coordinates}
		\begin{aligned}
			M_n^\varphi(t)&:=\langle u_n(t),\varphi\rangle-\langle u_0,\varphi\rangle
			-\frac12\int_0^t\langle u_n(s),\varphi''\rangle\dd s
			-\int_{\calZ_t}b_n(u_n(s,y))\varphi(y)\dd s\dd y,\\
			A_n^\varphi(t)&:=\int_{\calZ_t}
			|\sigma_n(u_n(s,y))\varphi(y)|^\alpha\dd s\dd y.
		\end{aligned}
	\end{equation}
	Here $u_n(t)$ is the state process and $u_n(s,y)$ is its predictable
	field. The process $A_n^\varphi$ is adapted, continuous and nondecreasing.
	Both $M_n^\varphi$ and $A_n^\varphi$ are well defined because the state
	paths are c\`adl\`ag and, for fixed $n$, $b_n$ and $\sigma_n$ are bounded.
	For $\theta\in\R$, define the complex-valued process
	\begin{equation}\label{eq:exponential-martingale}
		Z_n^{\varphi,\theta}(t)
		:=\exp\bigl(i\theta M_n^\varphi(t)+|\theta|^\alpha A_n^\varphi(t)\bigr),
		\qquad t\in[0,T].
	\end{equation}
	The factor $|\theta|^\alpha$ comes from  \eqref{charact-L}.
	
	\begin{lemma}\label{lem:approximation-martingale}
		For every $n\ge1$, $\varphi\in\calE$ and $\theta\in\R$,
		$Z_n^{\varphi,\theta}$ is a complex-valued $(\Fil_t)$-local martingale
		starting at $1$. For $R>0$, define
		\begin{equation}\label{eq:martingale-clock-stopping}
			\tau_{n,R}^\varphi
			:=\inf\{t\in[0,T]:A_n^\varphi(t)\ge R\}\wedge T,
			\qquad \inf\varnothing=\infty.
		\end{equation}
		Then $Z_n^{\varphi,\theta}(t\wedge\tau_{n,R}^\varphi)$ is a
		bounded complex-valued $(\Fil_t)$-martingale, and
		\[
		\sup_{t\le T}|Z_n^{\varphi,\theta}(t\wedge\tau_{n,R}^\varphi)|
		\le e^{|\theta|^\alpha R},\qquad\P\text{-a.s.}
		\]
		The bound depends only on $\alpha,\theta,R$, and is independent of $n$.
	\end{lemma}
	\begin{proof}
		On $\{\tauK=\infty\}$, the defining relation
		\eqref{eq:full-approximation} gives $u_n(t,y)=u_n^K(t,y)$.
		Their weak formulations therefore agree on this
		event, by \eqref{eq:mode} and the local property
		\eqref{eq:noise-local-property}, giving
		\begin{equation}\label{eq:martingale-stochastic-integral}
			M_n^\varphi(t)=\int_{\calZ_t}
			\sigma_n(u_n(s,y))\varphi(y)L(\dd s,\dd y),
			\qquad\forall t\in[0,T].
		\end{equation}
		Then, \eqref{eq:martingale-stochastic-integral} holds for all
		$t\in[0,T]$ on the event $\bigcup_{K\ge1}\{\tauK=\infty\}$,
		which has probability one.
		
		Put $H_n^\varphi(s,y):=\sigma_n(u_n(s,y))\varphi(y)$.
		For fixed $n$, this is a bounded predictable field.
		The integral over $|z|\le1$ is square-integrable. For $|z|>1$, we have
		$\int_{|z|>1}|z|\nu_\alpha(\dd z)<\infty$ and
		$\int_{|z|>1}z\nu_\alpha(\dd z)=0$.
		Thus replacing $N$ by $\Nt$ in the large-jump integral does not change
		its value. We can therefore write $M_n^\varphi$ as the integral of
		$zH_n^\varphi$ against $\Nt$ over all jump sizes; it is a local martingale.
		Set
		\[
		R_\theta(a):=e^{i\theta a}-1-i\theta a,
		\qquad \calQ_t:=\calZ_t\times(\R\setminus\{0\}).
		\]
		Applying It\^o's formula to the real and imaginary parts of
		$e^{i\theta M_n^\varphi}$ (\cite{JS}*{Theorem~I.4.57}), and writing the jump
		sum as an integral against $N=\Nt+\m$, gives
		\[
		\begin{aligned}
			e^{i\theta M_n^\varphi(t)}={}&1+\calM_n^\varphi(t)\\
			&+\int_{\calZ_t}e^{i\theta M_n^\varphi(s-)}
			\left[\int_{\R\setminus\{0\}}
			R_\theta(zH_n^\varphi(s,y))\nu_\alpha(\dd z)\right]\dd s\dd y,
		\end{aligned}
		\]
		where
		\[
		\calM_n^\varphi(t):=\int_{\calQ_t}e^{i\theta M_n^\varphi(s-)}
		\bigl(e^{i\theta zH_n^\varphi(s,y)}-1\bigr)
		\Nt(\dd s,\dd y,\dd z).
		\]
		This is a square-integrable complex-valued martingale. Indeed,
		$H_n^\varphi$ is bounded for fixed $n$, and the integrand has modulus at most
		$|\theta|\|H_n^\varphi\|_{\infty}|z|$ for $|z|\le1$, and at most $2$ for
		$|z|>1$. These bounds are square-integrable against the intensity
		on $\calQ_T$, since $\alpha<2$.
		For every $h\in\R$, symmetry and the normalization of $L$ give
		\[
		\int_{\R\setminus\{0\}}R_\theta(zh)\nu_\alpha(\dd z)
		=-|\theta h|^\alpha.
		\]
		The integral is absolutely convergent, since its integrand is of
		order $z^2$ near zero and at most of order $1+|z|$ at infinity,
		with $1<\alpha<2$. Its imaginary part vanishes by symmetry, and
		its real part is the stable characteristic exponent. Consequently,
		\begin{equation}\label{eq:ito-complex-martingale}
			e^{i\theta M_n^\varphi(t)}=1+\calM_n^\varphi(t)
			-|\theta|^\alpha\int_0^t e^{i\theta M_n^\varphi(s-)}\dd A_n^\varphi(s).
		\end{equation}
		By \eqref{eq:ito-complex-martingale} and the product rule, it follows that
		\[
		\dd Z_n^{\varphi,\theta}(t)
		=e^{|\theta|^\alpha A_n^\varphi(t)}\dd\calM_n^\varphi(t).
		\]
		Since $M_n^\varphi(0)=A_n^\varphi(0)=0$, $Z_n^{\varphi,\theta}$ has the stochastic-integral
		representation
		\[
		Z_n^{\varphi,\theta}(t)
		=1+\int_0^t e^{|\theta|^\alpha A_n^\varphi(s)}\dd\calM_n^\varphi(s).
		\]
		The integrand is predictable and locally bounded, because $A_n^\varphi$ is
		adapted and continuous. Hence, this stochastic integral is a
		complex-valued local martingale, and $Z_n^{\varphi,\theta}(0)=1$.
		
		Since $A_n^\varphi$ is adapted and continuous,
		$\tau_{n,R}^\varphi$ is a stopping time and
		$A_n^\varphi(t\wedge\tau_{n,R}^\varphi)\le R$.
		Thus
		\[
		|Z_n^{\varphi,\theta}(t\wedge\tau_{n,R}^\varphi)|
		=\exp\bigl(|\theta|^\alpha
		A_n^\varphi(t\wedge\tau_{n,R}^\varphi)\bigr)
		\le e^{|\theta|^\alpha R}.
		\]
		The stopped real and imaginary parts are bounded local martingales,
		hence martingales. This proves the claim.
	\end{proof}
	
	\subsection{The martingale problem for the limit}
	\label{subsec:limit-martingale-problem}
	
	We now establish the exponential martingale identity for the limit
	state process $\widetilde u$ obtained in \eqref{eq:skor}. For each pair $(\widetilde u_n,\widetilde v_n)$ and for the limit pair
	$(\widetilde u,\widetilde v)$, choose the predictable representative from
	Lemma~\ref{lem:predictable-representative}.
	
	We apply the same construction to each original pair $(u_n,u_n)$:
	Lemma~\ref{lem:predictable-representative} applies because
	$\P((u_n,u_n)\in\calG)=1$ and
	$u_n\in\L^2(\calZ_T)$, $\P$-a.s., by
	Lemma~\ref{lem:full-integrability-probability} with $q'=2$.
	The passage from a c\`adl\`ag state process to a predictable density is also discussed in \cite{Hughes}*{Remark~5.2}. In Appendix~\ref{app:predictable-representatives}, we construct a
	predictable weak modification from the sine Fej\'er means of the
	left-limit state; for Fej\'er means, see
	\cite{Katznelson}*{Chapter~I, Sections~2 and~3}. It is predictable for
	the natural filtration of the state and agrees with the field
	Leb-a.e.\ in space--time, almost surely. Hence, all preceding integrals
	and convergence statements remain unchanged. We keep the notation
	$\widetilde u_n(t,y)$ and $\widetilde u(t,y)$ for these representatives,
	and $\widetilde u_n(t)$ and $\widetilde u(t)$ for the unchanged
	state processes.
	
	Let $(\overline\Fil_t)$ be the completion and right-continuous
	augmentation of
	\[
	\overline\Fil_t^0:=\sigma\bigl(\widetilde u(s):0\le s\le t\bigr).
	\]
	Define the limit processes by
	\begin{equation}\label{eq:limit-martingale-coordinates}
		\begin{aligned}
			\widetilde M^\varphi(t)
			:={}&\langle\widetilde u(t),\varphi\rangle-\langle u_0,\varphi\rangle
			-\frac12\int_0^t\langle\widetilde u(s),\varphi''\rangle\dd s
			-\int_{\calZ_t}b(\widetilde u(s,y))\varphi(y)\dd s\dd y,\\
			\widetilde A^\varphi(t)
			:={}&\int_{\calZ_t}|\sigma(\widetilde u(s,y))\varphi(y)|^\alpha
			\dd s\dd y.
		\end{aligned}
	\end{equation}
	These processes are adapted to $(\overline\Fil_t)$ by the choice of
	representative above.
	
	\begin{proposition}\label{prop:limit-martingale}
		We have $\widetilde u(0)=u_0$ and
		$\widetilde A^\varphi(T)<\infty$, $\overline\P$-a.s.
		For every $\varphi\in\calE$ and $\theta\in\R$, the process
		\begin{equation}\label{eq:limit-exponential-martingale}
			\widetilde Z^{\varphi,\theta}(t)
			:=\exp\bigl(i\theta\widetilde M^\varphi(t)
			+|\theta|^\alpha\widetilde A^\varphi(t)\bigr)
		\end{equation}
		is a complex-valued $(\overline\Fil_t)$-local martingale starting at $1$.
	\end{proposition}
	
	\begin{proof}
		Evaluation at $0$ is continuous in the $\mathrm{J}_1$ topology. Thus
		\eqref{eq:skor} and $u_n(0)=u_0$ give $\widetilde u(0)=u_0$.
		The growth bound, $\alpha<2$ and \eqref{eq:limit-L2} imply
		$\widetilde A^\varphi(T)<\infty$ almost surely and integrability of
		the drift term in \eqref{eq:limit-martingale-coordinates}.
		
		Fix $\varphi\in\calE$ and $\theta\in\R$.
		Let $\widetilde M_n^\varphi$ and $\widetilde A_n^\varphi$ be the
		expressions in \eqref{eq:martingale-coordinates} with $u_n$ replaced
		by $\widetilde u_n$, retaining $b_n,\sigma_n$. Define
		\[
		\widetilde E_n^{\varphi,\theta}(t)
		:=e^{i\theta\widetilde M_n^\varphi(t)},\qquad
		\widetilde E^{\varphi,\theta}(t)
		:=e^{i\theta\widetilde M^\varphi(t)}.
		\]
		For each $n$, we use the natural filtration of $\widetilde u_n$, since
		equality in law transfers martingale identities only in this filtration.
		
		\smallskip
		\noindent\textit{1. Bounded martingales.}
		Let $f\in C_c^1([0,\infty);\R)$ and define
		\begin{equation}\label{eq:smooth-cutoff-martingale}
			\begin{aligned}
				\widetilde Y_n^{f,\varphi,\theta}(t)
				:={}&f(\widetilde A_n^\varphi(t))
				\widetilde E_n^{\varphi,\theta}(t)-f(0)\\
				&-\int_0^t\widetilde E_n^{\varphi,\theta}(s-)
				\bigl[f'(\widetilde A_n^\varphi(s))
				-|\theta|^\alpha f(\widetilde A_n^\varphi(s))\bigr]
				\dd\widetilde A_n^\varphi(s).
			\end{aligned}
		\end{equation}
		On the original probability space, let $Y_n^{f,\varphi,\theta}$ be
		this expression with $M_n^\varphi,A_n^\varphi$ in place of
		$\widetilde M_n^\varphi,\widetilde A_n^\varphi$.
		The product rule and \eqref{eq:ito-complex-martingale} show that
		\[
		\dd Y_n^{f,\varphi,\theta}(t)
		=f(A_n^\varphi(t))\dd\calM_n^\varphi(t),
		\qquad Y_n^{f,\varphi,\theta}(0)=0.
		\]
		Since $f(A_n^\varphi)$ is bounded and predictable,
		$Y_n^{f,\varphi,\theta}$ is a complex-valued local martingale.
		Both expressions have the deterministic bound
		\begin{equation}\label{eq:cutoff-martingale-bound}
			\sup_{t\le T}|\widetilde Y_n^{f,\varphi,\theta}(t)|
			\le 2\|f\|_{\infty}+
			\int_0^\infty|f'(a)-|\theta|^\alpha f(a)|\dd a
			=:\calC_{4,1}(f,\alpha,\theta).
		\end{equation}
		Indeed, $|\widetilde E_n^{\varphi,\theta}|=1$, and continuity and
		monotonicity of $\widetilde A_n^\varphi$ give
		\[
		\begin{aligned}
			&\int_0^t
			|f'(\widetilde A_n^\varphi(s))
			-|\theta|^\alpha f(\widetilde A_n^\varphi(s))|
			\dd\widetilde A_n^\varphi(s)\\
			&\qquad=\int_0^{\widetilde A_n^\varphi(t)}
			|f'(a)-|\theta|^\alpha f(a)|\dd a.
		\end{aligned}
		\]
		The same bound holds on the original space. Thus,
		$Y_n^{f,\varphi,\theta}$ is a bounded martingale for the natural
		filtration of $u_n$. The representative chosen above makes it a
		measurable function of the state path up to the current time. Equality
		in law in \eqref{eq:skor} therefore transfers the martingale identities:
		$\widetilde Y_n^{f,\varphi,\theta}$ is a martingale for the natural
		filtration of $\widetilde u_n$.
		The constant in \eqref{eq:cutoff-martingale-bound} depends only on
		$f,\alpha,\theta$, not on $n,T,\varphi$ or $\omega$. It is also listed in Appendix~\ref{app:constants}.
		
		\smallskip
		\noindent\textit{2. Convergence of these martingales.}
		The functions $\widetilde A_n^\varphi$ and $\widetilde A^\varphi$
		are absolutely continuous. Their densities are
		\[
		\begin{aligned}
			\dot{\widetilde A}_n^\varphi(s)
			&:=\int_D|\sigma_n(\widetilde u_n(s,y))\varphi(y)|^\alpha\dd y,\\
			\dot{\widetilde A}^\varphi(s)
			&:=\int_D|\sigma(\widetilde u(s,y))\varphi(y)|^\alpha\dd y.
		\end{aligned}
		\]
		Lemma~\ref{lem:nonlinear-convergence}, with $q'=2$, gives
		\begin{equation}\label{eq:clock-density-convergence}
			\begin{aligned}
				&\int_0^T|\dot{\widetilde A}_n^\varphi(s)
				-\dot{\widetilde A}^\varphi(s)|\dd s\longrightarrow0,\\
				&\int_{\calZ_T}
				|b_n(\widetilde u_n(s,y))-b(\widetilde u(s,y))|
				|\varphi(y)|\dd s\dd y\longrightarrow0
			\end{aligned}
		\end{equation}
		in probability. Choose a further deterministic subsequence along which
		both convergences hold almost surely. In particular,
		$\widetilde A_n^\varphi\to\widetilde A^\varphi$ uniformly on $[0,T]$.
		
		Let $\calJ$ be the set of deterministic times $t\in(0,T)$ at which
		$\overline\P(\Delta\widetilde u(t)\ne0)>0$. This set is at most
		countable for a c\`adl\`ag process. At each $t\in[0,T]\setminus \calJ$,
		$\mathrm{J}_1$ convergence gives $\widetilde u_n(t)\to\widetilde u(t)$ almost
		surely; at $0$ and $T$ this follows because the time changes fix both
		endpoints. Also,
		\[
		\int_0^T|\langle\widetilde u_n(s)-\widetilde u(s),\varphi''\rangle|
		\dd s\longrightarrow0
		\qquad\overline\P\text{-a.s.}
		\]
		To see this, the integrand converges at every continuity time of the
		limit path, hence for almost every $s$. A $\mathrm{J}_1$-convergent sequence
		satisfies $\sup_n\sup_{s\le T}\|\widetilde u_n(s)\|_{\H_{-r}}<\infty$
		pathwise, so dominated convergence applies. Together with
		\eqref{eq:clock-density-convergence}, this proves
		$\widetilde M_n^\varphi(t)\to\widetilde M^\varphi(t)$ and
		$\widetilde E_n^{\varphi,\theta}(t)\to
		\widetilde E^{\varphi,\theta}(t)$ almost surely at each $t\notin \calJ$.
		The same argument gives convergence for almost every $t$, almost surely.
		
		Define $\widetilde Y^{f,\varphi,\theta}$ by
		\eqref{eq:smooth-cutoff-martingale} with
		$\widetilde A^\varphi,\widetilde E^{\varphi,\theta}$ in place of
		$\widetilde A_n^\varphi,\widetilde E_n^{\varphi,\theta}$, and set
		$\psi^{f,\theta}(a):=f'(a)-|\theta|^\alpha f(a)$.
		For the integral term,
		\[
		\begin{aligned}
			&\sup_{t\le T}\Bigg|
			\int_0^t\widetilde E_n^{\varphi,\theta}(s-)
			\psi^{f,\theta}(\widetilde A_n^\varphi(s))
			\dot{\widetilde A}_n^\varphi(s)\dd s\\
			&\hspace{30mm}-\int_0^t\widetilde E^{\varphi,\theta}(s-)
			\psi^{f,\theta}(\widetilde A^\varphi(s))
			\dot{\widetilde A}^\varphi(s)\dd s\Bigg|\\
			&\quad\le\|\psi^{f,\theta}\|_{\infty} \int_0^T
			|\dot{\widetilde A}_n^\varphi(s)
			-\dot{\widetilde A}^\varphi(s)|\dd s\\
			&\qquad+\int_0^T\Big|
			\widetilde E_n^{\varphi,\theta}(s-)
			\psi^{f,\theta}(\widetilde A_n^\varphi(s))
			-\widetilde E^{\varphi,\theta}(s-)
			\psi^{f,\theta}(\widetilde A^\varphi(s))\Big|
			\dot{\widetilde A}^\varphi(s)\dd s
			\longrightarrow0
		\end{aligned}
		\]
		almost surely. The first term tends to zero by
		\eqref{eq:clock-density-convergence}. For the second, left and right
		values differ only at countably many times for each path; its integrand
		converges almost everywhere and is bounded by
		$2\|\psi^{f,\theta}\|_{\infty}\dot{\widetilde A}^\varphi$, which is
		integrable. Hence
		$\widetilde Y_n^{f,\varphi,\theta}(t)\to
		\widetilde Y^{f,\varphi,\theta}(t)$ almost surely at each $t\notin \calJ$.
		The process $\widetilde Y^{f,\varphi,\theta}$ is c\`adl\`ag and has
		the same bound \eqref{eq:cutoff-martingale-bound}.
		
		\smallskip
		\noindent\textit{3. Passing the martingale identity to the limit.}
		Take $s<t$ outside $\calJ$, times $s_1,\ldots,s_j\le s$ outside $\calJ$,
		indices $k_1,\ldots,k_j$, and a bounded continuous function
		$h:\R^j\to\R$. Set
		\[
		\Phi_n:=h\bigl(\langle\widetilde u_n(s_1),e_{k_1}\rangle,
		\ldots,\langle\widetilde u_n(s_j),e_{k_j}\rangle\bigr),
		\]
		and define $\Phi$ by replacing $\widetilde u_n$ by $\widetilde u$.
		The martingale property gives
		\[
		\overline{\E}\bigl[\Phi_n
		(\widetilde Y_n^{f,\varphi,\theta}(t)
		-\widetilde Y_n^{f,\varphi,\theta}(s))\bigr]=0.
		\]
		All the variables converge almost surely and are uniformly bounded,
		so
		\[
		\overline{\E}\bigl[\Phi
		(\widetilde Y^{f,\varphi,\theta}(t)
		-\widetilde Y^{f,\varphi,\theta}(s))\bigr]=0.
		\]
		For $s\notin\calJ$, the variables $\Phi$ generate $\overline\Fil_s^0$.
		Indeed, each state $\widetilde u(v)$ is determined by its coefficients
		$\langle\widetilde u(v),e_k\rangle$, $k\ge1$. The time $s$ itself is
		allowed among $s_1,\ldots,s_j$. For $v<s$, we can choose times
		$v_\ell\downarrow v$ outside the countable set $\calJ$, and
		$\widetilde u(v_\ell)\to\widetilde u(v)$ by right-continuity. A monotone class argument therefore gives the same
		identity for every bounded $\overline\Fil_s^0$-measurable variable.
		For arbitrary $s<t$, approach them from the right by times outside
		$\calJ$, keeping $t=T$ if needed. Right-continuity of
		$\widetilde Y^{f,\varphi,\theta}$ and bounded convergence give the
		martingale identity for all $s<t$.
		
		This property also holds for $(\overline\Fil_t)$.
		Indeed, completion does not change conditional expectations. If a
		bounded variable is measurable in the right-continuous augmentation
		at time $s$, it is measurable in every completed $\sigma$-field at times
		$s_\ell\downarrow s$. Apply the identity at $s_\ell<t$ and let
		$\ell\to\infty$. Thus, $\widetilde Y^{f,\varphi,\theta}$ is a bounded
		$(\overline\Fil_t)$-martingale.
		
		\smallskip
		\noindent\textit{4. Removing the smooth cutoff.}
		Restrict to positive integer levels $j$. Choose
		$f_j\in C_c^1([0,\infty);\R)$ equal to $1$ on $[0,j]$, and put
		\[
		\widetilde\tau_j^\varphi
		:=\inf\{t\in[0,T]:\widetilde A^\varphi(t)\ge j\}\wedge T,
		\qquad \inf\varnothing=\infty.
		\]
		These stopping times increase with $j$. By continuity and monotonicity,
		$\widetilde A^\varphi(s)\le j$ for $s\le\widetilde\tau_j^\varphi$.
		Thus $f_j(\widetilde A^\varphi)=1$ and
		$f_j'(\widetilde A^\varphi)=0$ up to and including this time, giving
		\[
		\begin{aligned}
			\widetilde Y^{f_j,\varphi,\theta}
			(t\wedge\widetilde\tau_j^\varphi)
			={}&\widetilde E^{\varphi,\theta}
			(t\wedge\widetilde\tau_j^\varphi)-1\\
			&+|\theta|^\alpha\int_0^{t\wedge\widetilde\tau_j^\varphi}
			\widetilde E^{\varphi,\theta}(s-)\dd\widetilde A^\varphi(s).
		\end{aligned}
		\]
		For integers $i<j$, both cutoff functions equal $1$, with derivative
		zero, throughout the range of $\widetilde A^\varphi$ up to
		$\widetilde\tau_i^\varphi$. The preceding integral identity therefore
		gives
		\[
		\widetilde Y^{f_i,\varphi,\theta}
		(t\wedge\widetilde\tau_i^\varphi)
		=\widetilde Y^{f_j,\varphi,\theta}
		(t\wedge\widetilde\tau_i^\varphi),
		\qquad\forall t\in[0,T],\quad i<j.
		\]
		Since the levels are countable, these consistent identities hold on one
		event of probability one. On this event, define the pasted process by
		\[
		\widetilde W^{\varphi,\theta}(t)
		:=\widetilde Y^{f_j,\varphi,\theta}(t)
		\quad\text{on }\{t\le\widetilde\tau_j^\varphi\}.
		\]
		The identity $\widetilde Y^{f_i,\varphi,\theta}(t\wedge\widetilde\tau_i^\varphi) =\widetilde Y^{f_j,\varphi,\theta}(t\wedge\widetilde\tau_i^\varphi)$ also holds at $t=\widetilde\tau_i^\varphi$. Therefore, the definition does not depend on the choice of $j$ with $t\le\widetilde\tau_j^\varphi$.
		
		Since $\widetilde A^\varphi(T)<\infty$ almost surely,
		$\widetilde\tau_j^\varphi=T$ for all sufficiently large $j$, pathwise;
		hence the definition covers $[0,T]$ almost surely.
		For each $t$, the pasted value is the eventual value of the adapted
		sequence $\widetilde Y^{f_j,\varphi,\theta}(t)$, so it is adapted.
		Its paths are c\`adl\`ag because this sequence of processes eventually
		agrees on the whole interval, for each such $\omega$.
		We set the process to zero on the remaining null event.
		
		Moreover,
		\[
		\widetilde W^{\varphi,\theta}(t\wedge\widetilde\tau_j^\varphi)
		=\widetilde Y^{f_j,\varphi,\theta}
		(t\wedge\widetilde\tau_j^\varphi).
		\]
		The right side is a martingale by optional stopping. Thus
		$\widetilde W^{\varphi,\theta}$ is a complex-valued local martingale.
		Since $\widetilde W^{\varphi,\theta}=\widetilde Y^{f_j,\varphi,\theta}$
		on $[0,\widetilde\tau_j^\varphi]$ and $\widetilde\tau_j^\varphi=T$ for
		large $j$, the formula for
		$\widetilde Y^{f_j,\varphi,\theta}(t\wedge\widetilde\tau_j^\varphi)$
		obtained above gives
		\begin{equation}\label{eq:limit-additive-identity}
			\widetilde W^{\varphi,\theta}(t)
			=\widetilde E^{\varphi,\theta}(t)-1
			+|\theta|^\alpha\int_0^t
			\widetilde E^{\varphi,\theta}(s-)\dd\widetilde A^\varphi(s).
		\end{equation}
		The same product rule as after \eqref{eq:ito-complex-martingale}
		therefore yields
		\[
		\widetilde Z^{\varphi,\theta}(t)
		=1+\int_0^t e^{|\theta|^\alpha\widetilde A^\varphi(s)}
		\dd\widetilde W^{\varphi,\theta}(s).
		\]
		Its integrand is predictable and locally bounded. This proves the
		claimed local martingale property.
	\end{proof}
	
	\begin{corollary}\label{cor:limit-coordinate-martingale}
		For every $\varphi\in\calE$, the process
		$\widetilde M^\varphi$ defined in
		\eqref{eq:limit-martingale-coordinates} is a real-valued
		c\`adl\`ag $(\overline\Fil_t)$-local martingale starting at zero.
	\end{corollary}
	\begin{proof}
		Fix $\varphi\in\calE$. For $R>0$, stop at
		\[
		\widetilde\tau_R^\varphi
		:=\inf\{t\in[0,T]:\widetilde A^\varphi(t)\ge R\}\wedge T,
		\qquad R>0,\quad\inf\varnothing:=\infty.
		\]
		By the proof of Proposition~\ref{prop:limit-martingale}, the process
		$\widetilde W^{\varphi,\theta}$ in \eqref{eq:limit-additive-identity},
		stopped at $\widetilde\tau_R^\varphi$, is a bounded martingale. Taking
		expectations gives, for every $t\in[0,T]$ and $\theta\in\R$,
		\[
		\left|1-\overline\E
		e^{i\theta\widetilde M^\varphi(t\wedge\widetilde\tau_R^\varphi)}\right|
		\le\min\{2,R|\theta|^\alpha\}.
		\]
		The identity
		$|x|=(2/\pi)\int_0^\infty(1-\cos(\theta x))\theta^{-2}\dd\theta$
		and Tonelli's theorem imply
		$\overline\E|\widetilde M^\varphi(t\wedge
		\widetilde\tau_R^\varphi)|<\infty$, since $\alpha>1$.
		Take imaginary parts of the stopped additive identity, divide by
		$\theta>0$, and let $\theta\downarrow0$. The sine term converges
		in $\L^1(\overline\Omega)$ by
		$|\sin(\theta x)|/\theta\le|x|$, and the time-integral
		term is bounded by $R\theta^{\alpha-1}\to0$. Passing to the limit $\theta\downarrow0$ in the martingale identity at two deterministic times shows that
		$\widetilde M^\varphi(\cdot\wedge\widetilde\tau_R^\varphi)$ is a
		martingale. Since $\widetilde A^\varphi(T)<\infty$ almost surely,
		this proves the local martingale claim.
	\end{proof}

	\subsection{Weighted stochastic integration}
	\label{subsec:weighted-stochastic-integration}
	
	We follow a strategy similar to that of Mytnik
	\cite{Mytnik02}*{Section~5, pp.~190--194, Lemmas~5.4--5.6}:
	starting from the linear family $\widetilde M^\varphi$ in
	\eqref{eq:limit-martingale-coordinates}, we construct a stochastic
	integral for predictable space--time fields. We then recover the noise as in \cite{Mytnik02}*{Lemma~5.7, pp.~194--195}, which adapts \cite{KonnoShiga88}*{Lemma~2.4}: we divide by $\sigma(\widetilde u)$ where it is nonzero, and use an independent noise where it vanishes. In our setting, the integrands are taken in a weighted $\L^\alpha$
	space with stochastic weight $|\sigma(\widetilde u)|^\alpha$ (see
	\eqref{eq:mytnik-weighted-integrability} below), and the noise to be
	recovered is symmetric $\alpha$-stable.
	
	Our goal is to construct an extension of
	$\overline{\bbB}=(\overline\Omega,\overline\Fil,(\overline\Fil_t),\overline\P)$,
	satisfying the usual conditions and carrying a Poisson random measure
	$\calN$ with intensity $\m$. The associated symmetric stable
	noise $\widetilde L$, defined by \eqref{eq:Levy-Ito}, should satisfy,
	for every $\varphi\in\calE$,
	\begin{equation}\label{eq:mytnik-reconstruction-goal}
		\widetilde M^\varphi(t)
		=\int_{\calZ_t}\sigma(\widetilde u(s,y))\varphi(y)
		\widetilde L(\dd s,\dd y),
		\qquad\forall t\in[0,T],\quad\widetilde\P\text{-a.s.}
	\end{equation}
	The extension will preserve the joint law of the state and its
	predictable field. This identity gives the equation in
	Definition~\ref{def:weak-solution}(ii); the fixed-time identification
	in condition~(i) is a separate step.
	
	Following \cite{Mytnik02}*{Lemma~5.7}, first take a product extension
	$\widetilde{\bbB}=(\widetilde\Omega,\widetilde\Fil,(\widetilde\Fil_t),\widetilde\P)$
	carrying a copy $\widehat L$ of the original symmetric stable noise.
	Its values on deterministic space--time sets are independent of
	$\overline\Fil$. Let $(\widehat\Fil_t)$ be its usual natural
	filtration, and let $(\widetilde\Fil_t)$ be the usual augmentation
	of $(\overline\Fil_t\vee\widehat\Fil_t)$.
	The noise $\widehat L$ has the same law as $L$ and is a L\'evy basis
	with respect to $(\widehat\Fil_t)$, and also with respect to
	$(\widetilde\Fil_t)$ by independence.
	
	We keep the notation $\widetilde u$, $\widetilde M^\varphi$ and
	$\widetilde A^\varphi$ for these processes on the extended space.
	
	We show that enlarging the filtration preserves the local
	martingale properties of $\widetilde M^\varphi$,
	$\widetilde W^{\varphi,\theta}$ and
	$\widetilde Z^{\varphi,\theta}$.
	Fix $\varphi\in\calE$, $\theta\in\R$ and $j\ge1$,
	and let $X$ denote any one of these processes stopped at
	$\widetilde\tau_j^\varphi$. For $0\le s<t\le T$,
	$B\in\overline\Fil_s$, and $C\in\widehat\Fil_s$, independence gives
	\[
	\widetilde\E\bigl[(X(t)-X(s))\mathbf1_B\mathbf1_C\bigr]
	=
	\widetilde\E\bigl[(X(t)-X(s))\mathbf1_B\bigr]\,
	\widetilde\P(C)=0.
	\]
	A monotone class argument extends this identity to
	$\overline\Fil_s\vee\widehat\Fil_s$, and adding null sets does not
	change it.
	
	For $A\in\widetilde\Fil_s$, apply the identity at $s_\ell\downarrow s$,
	with $s<s_\ell<t$. Right-continuity and uniform integrability give
	$X(s_\ell)\to X(s)$ in $\L^1(\widetilde\Omega)$, so the identity
	also holds at $s$. The original stopping times therefore show that
	the processes remain local martingales in $(\widetilde\Fil_t)$,
	which we use below.
	
	We work with the predictable representative from
	Lemma~\ref{lem:predictable-representative}, so
	$\widetilde u\in\calP(\widetilde{\bbB})$ on the extended basis.
	We construct an $\L^0(\widetilde\Omega)$-valued random measure $\scrM$ from
	$\widetilde M^\varphi$, beginning with elementary integrands.
	
	\Needspace{9\baselineskip}
	Let $\calS_{\widetilde M}$ consist of the fields
	\begin{equation}\label{eq:mytnik-elementary-fields}
		F(s,y)=\sum_{j=1}^{m}
		\xi_j\mathbf1_{(a_j,b_j]}(s)\varphi_j(y),
	\end{equation}
	where $m\ge1$, $0\le a_j<b_j\le T$,
	$\varphi_j\in\calE$, and $\xi_j$ is a bounded real
	$\widetilde\Fil_{a_j}$-measurable random variable.
	For such a field, define
	\begin{equation}\label{eq:mytnik-elementary-integral}
		\calI_{\widetilde M}(F)(t)
		:=\sum_{j=1}^{m}\xi_j
		\bigl(\widetilde M^{\varphi_j}(t\wedge b_j)
		-\widetilde M^{\varphi_j}(t\wedge a_j)\bigr),
		\qquad t\in[0,T].
	\end{equation}
	The definition of $\widetilde M^\varphi$ is linear in $\varphi$.
	Consequently, refining the time intervals and expanding the
	$\varphi_j$ in the common basis $e_k$ shows that
	\eqref{eq:mytnik-elementary-integral} does not depend on the
	representation of the elementary field.
	
	By Corollary~\ref{cor:limit-coordinate-martingale},
	$\calI_{\widetilde M}(F)$ is a c\`adl\`ag local martingale for every $F\in\calS_{\widetilde M}$: each summand
	in \eqref{eq:mytnik-elementary-integral} is a bounded predictable
	multiple of an increment of $\widetilde M^{\varphi_j}$.
	
	Define the nonnegative random measure
	\[
	\mu_{\widetilde u}(\dd s,\dd y)
	:=|\sigma(\widetilde u(s,y))|^\alpha\dd s\dd y
	\quad\text{on }\calZ_T.
	\]
	By linear growth and \eqref{eq:limit-L2},
	$\mu_{\widetilde u}(\calZ_T)<\infty$,
	$\widetilde\P$-a.s. By definition,
	$\widetilde A^\varphi(t)=\int_{\calZ_t}|\varphi(y)|^\alpha
	\mu_{\widetilde u}(\dd s,\dd y)$.
	We extend the elementary integral to predictable fields $F$ satisfying
	\begin{equation}\label{eq:mytnik-weighted-integrability}
		\|F\|_{\L^\alpha(\calZ_T,\mu_{\widetilde u})}^\alpha
		=\int_{\calZ_T}|F(s,y)|^\alpha
		\mu_{\widetilde u}(\dd s,\dd y)<\infty,
		\qquad\widetilde\P\text{-a.s.}
	\end{equation}
	Here the $\L^\alpha$ space is understood for each $\omega$.
	In particular, every bounded predictable field satisfies this condition.
	
	\begin{lemma}[Weighted elementary approximation]
		\label{lem:weighted-density}
		Let $F\in\calP(\widetilde{\bbB})$ satisfy
		\eqref{eq:mytnik-weighted-integrability}. There exist
		$F_\ell\in\calS_{\widetilde M}$ such that
		\[
		\|F_\ell-F\|_{\L^\alpha(\calZ_T,\mu_{\widetilde u})}
		\longrightarrow0,
		\qquad \widetilde\P\text{-a.s.}
		\]
		In particular, convergence holds in probability.
	\end{lemma}
	
	\begin{proof}
		Suppose first that $|F|\le M$. Set $h_n=T2^{-n}$ and define
		\[
		F_n(t,y)
		:=
		\sum_{j=1}^{2^n-1}\sum_{k=1}^{n}
		\xi_{j,k}^{(n)}
		\mathbf1_{(jh_n,(j+1)h_n]}(t)e_k(y),
		\]
		where
		\[
		\xi_{j,k}^{(n)}
		:=
		\left(1-\frac{k}{n+1}\right)\frac1{h_n}
		\int_{(j-1)h_n}^{jh_n}\int_D
		F(s,z)e_k(z)\,\dd z\,\dd s.
		\]
		Each $\xi_{j,k}^{(n)}$ is $\widetilde\Fil_{jh_n}$-measurable
		and satisfies $|\xi_{j,k}^{(n)}|\le\sqrt2 M$.
		Thus $F_n\in\calS_{\widetilde M}$.
		
		To prove convergence, we define the spatial Fej\'er operator as in
		Appendix~\ref{app:predictable-representatives},
		\[
		(S_nH)(t,y)
		:= \fraks_n (H(t, \cdot ) )(y) =
		\sum_{k=1}^n\left(1-\frac{k}{n+1}\right)
		\left(\int_D H(t,z)e_k(z)\,\dd z\right)e_k(y)
		\]
		and the time-averaging operator
		\[
		(Q_nH)(t,y)
		:=
		\sum_{j=1}^{2^n-1}
		\mathbf1_{(jh_n,(j+1)h_n]}(t)
		\frac1{h_n}\int_{(j-1)h_n}^{jh_n}H(s,y)\,\dd s.
		\]
		Then $F_n=Q_nS_nF$. Both operators are contractions in
		$\L^2(\widetilde\Omega\times\calZ_T,
		\widetilde\P\otimes\operatorname{Leb})$, by Parseval's identity
		and Jensen's inequality, respectively.
		
		For $f\in\L^2(D)$, put
		$c_k:=\int_D f(y)e_k(y)\,\dd y$. Parseval's identity gives
		\[
		\|S_nf-f\|_{\L^2(D)}^2
		=
		\sum_{k=1}^n\left(\frac{k}{n+1}\right)^2|c_k|^2
		+
		\sum_{k>n}|c_k|^2
		\longrightarrow0.
		\]
		The expression is bounded by $\|f\|_{\L^2(D)}^2$.
		Applying this to $F(\omega,t,\cdot)$ and using dominated convergence
		therefore gives
		\[
		S_nF\longrightarrow F
		\quad\text{in }
		\L^2(\widetilde\P\otimes\operatorname{Leb}).
		\]
		Likewise, the averages of a continuous function over the preceding
		time interval converge to that function in $\L^2(0,T)$.
		Applying this to finite sums of products $X(\omega)g(t)h(y)$,
		with $X\in\L^2(\widetilde\Omega)$, $g$ continuous, and
		$h\in\L^2(D)$, and then using density and the contraction
		bound gives
		\[
		Q_nF\longrightarrow F
		\quad\text{in }
		\L^2(\widetilde\P\otimes\operatorname{Leb}).
		\]
		Hence
		\[
		\|F_n-F\|_{\L^2(\widetilde\P\otimes\operatorname{Leb})}
		\le
		\|S_nF-F\|_{\L^2(\widetilde\P\otimes\operatorname{Leb})}
		+
		\|Q_nF-F\|_{\L^2(\widetilde\P\otimes\operatorname{Leb})}
		\longrightarrow0.
		\]
		
		The Fej\'er operator preserves the bound $M$: it is convolution
		of the odd periodic extension with a nonnegative kernel of integral
		one. Time averaging also preserves this bound, so $|F_n|\le M$.
		By the $\L^2$ convergence above, we can choose a deterministic
		subsequence such that
		\[
		F_{n_\ell}\longrightarrow F
		\quad \widetilde\P\otimes\operatorname{Leb}\text{-a.e.}
		\]
		By Fubini's theorem and the absolute continuity of
		$\mu_{\widetilde u}(\omega)$ with respect to Lebesgue measure,
		this convergence holds $\mu_{\widetilde u}(\omega)$-a.e.
		for almost every $\omega$. Since
		\[
		|F_{n_\ell}-F|^\alpha\le(2M)^\alpha,
		\qquad
		\mu_{\widetilde u}(\calZ_T)<\infty,
		\quad \widetilde\P\text{-a.s.},
		\]
		dominated convergence applied for each such $\omega$ gives
		\[
		\int_{\calZ_T}|F_{n_\ell}-F|^\alpha
		\,\mu_{\widetilde u}(\dd s,\dd y)
		\longrightarrow0,
		\qquad \widetilde\P\text{-a.s.}
		\]
		
		For general $F$, set $F^{(m)}:=(-m)\vee(F\wedge m)$.
		By \eqref{eq:mytnik-weighted-integrability} and dominated convergence,
		\[
		\int_{\calZ_T}|F^{(m)}-F|^\alpha
		\,\mu_{\widetilde u}(\dd s,\dd y)
		\longrightarrow0,
		\qquad \widetilde\P\text{-a.s.}
		\]
		For each $m$, the bounded case allows us to choose
		$G_m\in\calS_{\widetilde M}$ such that
		\[
		\widetilde\P\left(
		\int_{\calZ_T}|G_m-F^{(m)}|^\alpha
		\,\mu_{\widetilde u}(\dd s,\dd y)>2^{-m}
		\right)<2^{-m}.
		\]
		By Borel--Cantelli, the weighted errors between $G_m$ and $F^{(m)}$
		tend to zero almost surely. The inequality
		\[
		|G_m-F|^\alpha
		\le 2^{\alpha-1}
		\bigl(|G_m-F^{(m)}|^\alpha+|F^{(m)}-F|^\alpha\bigr)
		\]
		then proves the result with $F_\ell=G_\ell$.
	\end{proof}

	Fix a predictable field $F$ satisfying
	\eqref{eq:mytnik-weighted-integrability}. Choose elementary fields
	$F_\ell\in\calS_{\widetilde M}$ as in
	Lemma~\ref{lem:weighted-density}, so that
	\begin{equation}\label{eq:mytnik-weighted-approximation}
		\|F_\ell-F\|_{\L^\alpha(\calZ_T,\mu_{\widetilde u})}
		\longrightarrow0
		\quad\text{in }\widetilde\P\text{-probability}.
	\end{equation}
	Take a deterministic subsequence and relabel it so that
	\eqref{eq:mytnik-weighted-approximation} holds
	$\widetilde\P$-a.s. Then, for almost every $\omega$, $(F_\ell)$
	is Cauchy in $\L^\alpha(\calZ_T,\mu_{\widetilde u})$.
	We show next that the elementary integrals are Cauchy in probability.
	
	For any predictable $H$ satisfying
	\eqref{eq:mytnik-weighted-integrability}, write
	\[
	\widetilde A^H(t):=\int_{\calZ_t}|H(s,y)|^\alpha
	\mu_{\widetilde u}(\dd s,\dd y).
	\]
	This process is continuous, adapted and nondecreasing.
	To prove the Cauchy property, we need to show that the elementary
	integral tends to zero in probability when its integrand tends to
	zero in $\L^\alpha(\calZ_T,\mu_{\widetilde u})$.
	The following exponential identity relates the integral to
	$\widetilde A^H$ and yields the estimate
	\eqref{eq:elementary-integral-continuity} below.
	For an elementary $H$ and $\theta\in\R$, the process
	\begin{equation}\label{eq:elementary-integral-exponential}
		t\longmapsto e^{i\theta\calI_{\widetilde M}(H)(t)}
		+|\theta|^\alpha\int_0^t
		e^{i\theta\calI_{\widetilde M}(H)(s-)}
		\dd\widetilde A^H(s)
	\end{equation}
	is a local martingale starting at $1$.
	First suppose the $\xi_j$ in the elementary representation of $H$
	are finite-valued. Refining the time intervals gives a partition on
	each interval of which $H$ is a finite-valued choice of a test
	function in $\calE$, measurable at the left endpoint.
	
	Fix one of these intervals $(a,b]$ and let $G_\varphi\in\widetilde\Fil_a$
	be the event where $H=\varphi$ on $(a,b]$. On $G_\varphi$ we have
	$\calI_{\widetilde M}(H)(t)-\calI_{\widetilde M}(H)(a)
	=\widetilde M^\varphi(t)-\widetilde M^\varphi(a)$ for $a\le t\le b$.
	Hence the identity for $H$ on $(a,b]$ follows from
	\eqref{eq:limit-additive-identity} after multiplying by
	$\mathbf1_{G_\varphi}\exp\bigl(i\theta[\calI_{\widetilde M}(H)(a)
	-\widetilde M^\varphi(a)]\bigr)$, and since this factor is bounded and
	$\widetilde\Fil_a$-measurable, the martingale property is preserved.
	
	Summing over the finitely many choices and intervals proves
	\eqref{eq:elementary-integral-exponential} in this case.
	For general bounded $\xi_j$, choose finite-valued
	$\xi_{j,\ell}$ with $|\xi_{j,\ell}-\xi_j|\le\ell^{-1}$, retaining
	the same measurability, and replace $\xi_j$ by $\xi_{j,\ell}$ to
	obtain $H_\ell$. By \eqref{eq:mytnik-elementary-integral},
	\[
	\sup_{t\le T}|\calI_{\widetilde M}(H_\ell)(t)
	-\calI_{\widetilde M}(H)(t)|
	\le\frac2\ell\sum_{j=1}^{m}
	\sup_{t\le T}|\widetilde M^{\varphi_j}(t)|
	\longrightarrow0\quad\widetilde\P\text{-a.s.}
	\]
	Also $H_\ell\to H$ uniformly with a common deterministic bound.
	Thus, $\int_{\calZ_T}\bigl||H_\ell|^\alpha-|H|^\alpha\bigr|
	\dd \mu_{\widetilde u}\to0$ almost surely.
	Stop where $\mu_{\widetilde u}(\calZ_t)$ reaches a level $\kappa \in\N$.
	On this interval, the expressions in
	\eqref{eq:elementary-integral-exponential} for $H_\ell$ have a
	common deterministic bound and converge uniformly almost surely.
	For $s<t$ and bounded $\widetilde\Fil_s$-measurable $Z$, let $X_\ell$
	be the process in \eqref{eq:elementary-integral-exponential} for
	$H_\ell$, stopped at this time. The identities
	$\widetilde\E[Z(X_\ell(t)-X_\ell(s))]=0$ pass to the limit by bounded
	convergence. This proves the claim for $H$.
	
	For $R>0$, put
	\[
	\widetilde\tau_R^H
	:=\inf\{t\in[0,T]:\widetilde A^H(t)\ge R\}\wedge T,
	\qquad\inf\varnothing:=\infty.
	\]
	Since $\widetilde A^H$ is continuous and nondecreasing,
	$\widetilde A^H(t)\le R$ for all $t\le\widetilde\tau_R^H$. As in
	Corollary~\ref{cor:limit-coordinate-martingale}, the stopped
	exponential identity shows that
	$\calI_{\widetilde M}(H)(\cdot\wedge\widetilde\tau_R^H)$
	is a martingale and
	\[
	\widetilde\E\left|
	\calI_{\widetilde M}(H)(T\wedge\widetilde\tau_R^H)\right|
	\le\frac2\pi\int_0^\infty
	\frac{\min\{2,R\theta^\alpha\}}{\theta^2}\dd\theta
	=\calC_{5,1}(\alpha)R^{1/\alpha},
	\]
	where $\calC_{5,1}(\alpha)
	=2^{2-1/\alpha}\alpha/[\pi(\alpha-1)]$ is given in
	Appendix~\ref{app:constants}. Doob's maximal inequality gives
	\begin{equation}\label{eq:elementary-integral-continuity}
		\widetilde\P\!\left(
		\sup_{t\le T}|\calI_{\widetilde M}(H)(t)|>\varepsilon\right)
		\le\widetilde\P\bigl(\widetilde A^H(T)\ge R\bigr)
		+\frac{\calC_{5,1}(\alpha)R^{1/\alpha}}{\varepsilon},
		\qquad\varepsilon,R>0.
	\end{equation}
	Here we used that $\widetilde\tau_R^H=T$ on
	$\{\widetilde A^H(T)<R\}$.
	Consequently, if elementary $H_\ell$ satisfy
	$\widetilde A^{H_\ell}(T)\to0$ almost surely, then
	\[
	\sup_{t\le T}|\calI_{\widetilde M}(H_\ell)(t)|
	\longrightarrow0\quad\text{in probability}.
	\]
	Indeed, for fixed $R>0$, the first term on the right-hand side of
	\eqref{eq:elementary-integral-continuity} tends to zero as
	$\ell\to\infty$, and the second term tends to zero as $R\downarrow0$.
	This argument only uses that $\widetilde A^{H_\ell}(T)\to0$ in
	probability.
	
	Apply this estimate to $H=F_\ell-F_j$. Since
	\[
	\widetilde A^{F_\ell-F_j}(T)
	\le2^{\alpha-1}\bigl(
	\widetilde A^{F_\ell-F}(T)+\widetilde A^{F_j-F}(T)\bigr)
	\longrightarrow0\quad\widetilde\P\text{-a.s.},
	\qquad\ell,j\to\infty,
	\]
	linearity gives, for every $\varepsilon>0$,
	\[
	\lim_{\ell,j\to\infty}\widetilde\P\!\left(
	\sup_{t\le T}\left|
	\calI_{\widetilde M}(F_\ell)(t)
	-\calI_{\widetilde M}(F_j)(t)\right|>\varepsilon\right)=0.
	\]
	Thus, the elementary integrals are Cauchy uniformly in time in
	probability. A subsequence converges uniformly in time almost
	surely, and its limit has c\`adl\`ag paths. The Cauchy property then
	implies convergence of the whole sequence uniformly in time in
	probability. We define
	\begin{equation}\label{eq:mytnik-integral-extension}
		\calI_{\widetilde M}(F)(t)
		\overset{\widetilde\P}{:=}
		\lim_{\ell\to\infty}\calI_{\widetilde M}(F_\ell)(t),
		\qquad t\in[0,T],
	\end{equation}
	and use this adapted c\`adl\`ag version.
	If another elementary sequence approximates $F$, its difference
	from $F_\ell$ tends to zero in
	$\L^\alpha(\calZ_T,\mu_{\widetilde u})$ in probability. The same estimate shows that the integral limits are
	indistinguishable. Hence, the definition is independent of the
	approximation, is linear, and agrees with
	\eqref{eq:mytnik-elementary-integral} for elementary fields.
	Choosing elementary approximations also extends the continuity
	property: whenever $\widetilde A^{H_\ell-H}(T)\to0$ in probability,
	$\calI_{\widetilde M}(H_\ell)\to\calI_{\widetilde M}(H)$
	uniformly in time in probability.
	
	We can now define
	\[
	\scrM(A):=\calI_{\widetilde M}(\mathbf1_A)(T),
	\qquad A\in\calP(\widetilde{\bbB}).
	\]
	Indicators satisfy \eqref{eq:mytnik-weighted-integrability} because
	$\mu_{\widetilde u}(\calZ_T)<\infty$ almost surely.
	For disjoint predictable
	sets $A_j$ and $A=\bigcup_{j\ge1}A_j$, dominated convergence gives
	\[
	\widetilde A^{\mathbf1_A-\sum_{j=1}^{\ell}\mathbf1_{A_j}}(T)
	\longrightarrow0,\qquad\widetilde\P\text{-a.s.}
	\]
	Linearity and continuity of the extension imply
	$\scrM(A)=\sum_{j\ge1}\scrM(A_j)$ in probability.
	Also $\scrM(\varnothing)=0$. Restricting in time commutes with the limit:
	\[
	\scrM\bigl(A\cap(\widetilde\Omega\times\calZ_t)\bigr)
	=\calI_{\widetilde M}(\mathbf1_A)(t)
	\quad\text{is }\widetilde\Fil_t\text{-measurable}.
	\]
	For $B\in\widetilde\Fil_t$, \eqref{eq:mytnik-elementary-integral}
	similarly gives, after passage to the limit,
	\[
	\scrM\bigl(A\cap(B\times(t,T]\times D)\bigr)
	=\mathbf1_B\,
	\scrM\bigl(A\cap(\widetilde\Omega\times(t,T]\times D)\bigr).
	\]
	These are the required properties of an $\L^0(\widetilde\Omega)$-valued random measure
	\cite{ChongKlueppelberg}*{Definition~2.1(1), with $p=0$}.
	We henceforth write
	\[
	\int_{\calZ_t}F(s,y)\,\scrM(\dd s,\dd y)
	:=\calI_{\widetilde M}(F)(t).
	\]
	
	For an integrand predictable with respect to the original
	filtration, the construction above can use elementary
	approximations from that filtration. Uniqueness of the limit shows
	that its integral is unchanged by the enlargement and is
	$\overline\Fil$-measurable. We shall use this observation when
	conditioning on $\overline\Fil$ below.

	\subsection{Reconstruction of the stable noise}
	\label{subsec:stable-noise-reconstruction}
	
	For $A\in\calP(\widetilde{\bbB})$, define
	\begin{equation}\label{eq:reconstructed-noise-definition}
		\begin{aligned}
			\widetilde L(A)
			:= {}&\int_{\calZ_T}
			\frac{\mathbf1_A(s,y)\mathbf1_{\{\sigma(\widetilde u(s,y))\ne0\}}}
			{\sigma(\widetilde u(s,y))}\,\scrM(\dd s,\dd y)\\
			&+\int_{\calZ_T}\mathbf1_A(s,y)
			\mathbf1_{\{\sigma(\widetilde u(s,y))=0\}}
			\,\widehat L(\dd s,\dd y).
		\end{aligned}
	\end{equation}
	The dependence of $\mathbf1_A$ on $\omega$ is suppressed, and
	the quotient is zero when $\sigma(\widetilde u)=0$.
	Both integrals are well defined, since
	\[
	\int_{\calZ_T}
	\left|\frac{\mathbf1_A\mathbf1_{\{\sigma(\widetilde u)\ne0\}}}
	{\sigma(\widetilde u)}\right|^\alpha
	\dd \mu_{\widetilde u}
	=\int_{\calZ_T}\mathbf1_A
	\mathbf1_{\{\sigma(\widetilde u)\ne0\}}\dd s\dd y
	\le T,
	\]
	and the integrand against $\widehat L$ satisfies the same bound
	with the complementary indicator.
	For integrals against $\widehat L$, the same exponential-martingale
	and Doob argument gives \eqref{eq:elementary-integral-continuity} with
	$\widetilde A^H(T)$ replaced by $\int_{\calZ_T}|H(s,y)|^\alpha\dd s\dd y$.
	
	This defines an $\L^0(\widetilde\Omega)$-valued random measure in the sense of
	\cite{ChongKlueppelberg}*{Definition~2.1(1), with $p=0$} on the enlarged basis.
	To check countable additivity, let the predictable sets $A_j$ be
	disjoint and set $A=\bigcup_j A_j$. The indicators of
	$A\setminus\bigcup_{j\le m}A_j$ decrease to zero. The two
	corresponding integrability bounds above tend to zero by dominated
	convergence. Continuity of the two stochastic integrals therefore
	gives $\widetilde L(A)=\sum_j\widetilde L(A_j)$ in probability.
	Also $\widetilde L(\varnothing)=0$. The other two properties in that definition also hold for
	$\widetilde L$, because they hold for both integrals: $\widetilde L(A)$
	is $\widetilde\Fil_t$-measurable when $A\subset\widetilde\Omega\times
	[0,t]\times D$, and the indicator of an event in $\widetilde\Fil_t$ can
	be taken out of $\widetilde L(A)$ when $A$ lies after time $t$.
	
	For a deterministic $B\in\calB(\calZ_T)$, write
	$\widetilde L(B):=\widetilde L(\widetilde\Omega\times B)$.
	For predictable $f$ with
	$\int_{\calZ_T}|f|^\alpha\dd s\dd y<\infty$ almost surely, set
	\[
	\widetilde L_t(f):=\int_{\calZ_t}f(s,y)\,
	\widetilde L(\dd s,\dd y).
	\]
	These integrals are well defined and satisfy
	\begin{equation}\label{eq:reconstructed-noise-substitution}
		\begin{aligned}
			\widetilde L_t(f)
			={}&\int_{\calZ_t}
			\frac{f(s,y)\mathbf1_{\{\sigma(\widetilde u(s,y))\ne0\}}}
			{\sigma(\widetilde u(s,y))}\,\scrM(\dd s,\dd y)\\
			&+\int_{\calZ_t}f(s,y)
			\mathbf1_{\{\sigma(\widetilde u(s,y))=0\}}
			\,\widehat L(\dd s,\dd y).
		\end{aligned}
	\end{equation}
	For simple predictable $f$, this follows by linearity from
	\eqref{eq:reconstructed-noise-definition}. For general $f$, choose
	simple predictable $f_\ell$ such that $\int_{\calZ_T}|f_\ell-f|^\alpha
	\dd s\dd y\to0$ in probability. Both integrals in
	\eqref{eq:reconstructed-noise-substitution} are controlled by this
	quantity; for the $\scrM$-integral, the weight
	$|\sigma(\widetilde u)|^\alpha$ cancels the factor
	$1/\sigma(\widetilde u)$. So the estimate
	\eqref{eq:elementary-integral-continuity}, and its analogue for
	$\widehat L$, show that both integrals with $f_\ell$ converge to those
	with $f$, uniformly in time in probability. The same bound holds for
	every simple integrand dominated by $|f_\ell-f|$, so $f_\ell\to f$ in
	the Daniell mean of $\widetilde L$. This proves both integrability and
	\eqref{eq:reconstructed-noise-substitution}. We use the resulting
	adapted c\`adl\`ag version.
	
	We next identify the law of $\widetilde L$ on deterministic
	space--time sets and verify its compatibility with the filtration.
	The representation \eqref{eq:mytnik-reconstruction-goal} will be
	addressed afterwards.
	
	\smallskip
	\noindent\textit{The exponential identity for general integrands.}
	We first extend \eqref{eq:elementary-integral-exponential} to every
	predictable $H$ satisfying \eqref{eq:mytnik-weighted-integrability};
	this is the counterpart of \cite{Mytnik02}*{Lemma~5.5 and equation~(5.6)}.
	Choose elementary $H_\ell$ converging to $H$ in the
	$\L^\alpha(\calZ_T,\mu_{\widetilde u})$ norm in probability. By the construction above, after
	a deterministic subsequence, both
	\[
	\sup_{t\le T}|\calI_{\widetilde M}(H_\ell)(t)
	-\calI_{\widetilde M}(H)(t)|\longrightarrow0,
	\qquad
	\int_{\calZ_T}\bigl||H_\ell|^\alpha-|H|^\alpha\bigr|
	\dd \mu_{\widetilde u}\longrightarrow0
	\]
	hold almost surely. The second convergence follows from
	$\bigl||x|^\alpha-|y|^\alpha\bigr| \le\alpha|x-y|(|x|+|y|)^{\alpha-1}$ and H\"older's inequality. In particular, $\widetilde A^{H_\ell}\to\widetilde A^H$ uniformly on
	$[0,T]$, and the measures $\dd\widetilde A^{H_\ell}$ converge to $\dd\widetilde A^H$ in total variation.
	
	We use the same smooth cutoff as in
	Proposition~\ref{prop:limit-martingale}. For
	$\chi\in C_c^1([0,\infty);\R)$, the product rule and
	\eqref{eq:elementary-integral-exponential} show that
	\[
	\begin{aligned}
		&\chi(\widetilde A^{H_\ell}(t))
		e^{i\theta\calI_{\widetilde M}(H_\ell)(t)}-\chi(0)\\
		&\quad-\int_0^t e^{i\theta\calI_{\widetilde M}(H_\ell)(s-)}
		\bigl[\chi'(\widetilde A^{H_\ell}(s))
		-|\theta|^\alpha\chi(\widetilde A^{H_\ell}(s))\bigr]
		\dd\widetilde A^{H_\ell}(s)
	\end{aligned}
	\]
	is a martingale bounded by
	$\calC_{4,1}(\chi,\alpha,\theta)$ from Appendix~\ref{app:constants},
	independently of $\ell$ and $H_\ell$.
	The two convergences above give uniform convergence of this
	expression to the expression with $H$ in place of $H_\ell$.
	For the integral term, use uniform convergence of the exponential
	factors and cutoff functions, together with convergence in total
	variation of the measures $\dd\widetilde A^{H_\ell}$. Bounded convergence against every bounded $\widetilde\Fil_s$-measurable variable, for $s\le t$, therefore
	passes the martingale identity to the limit. Choose $\chi=1$ on $[0,R]$ and stop at
	$\widetilde\tau_R^H$. The limiting expression is then the stopped
	version of \eqref{eq:elementary-integral-exponential}, minus $1$.
	Letting $R$ increase proves that identity for $H$.
	The product rule now shows that the process
	\begin{equation}\label{eq:extended-exponential-martingale}
		t\longmapsto
		\exp\!\left(i\theta\int_{\calZ_t}H\dd\scrM
		+|\theta|^\alpha\widetilde A^H(t)\right)
	\end{equation}
	is a local martingale starting at $1$.
	If $\widetilde A^H(T)\le B$ for a deterministic finite $B$, its
	modulus is at most $e^{|\theta|^\alpha B}$, so it is a true martingale.
	
	\smallskip
	\noindent\textit{The stable characteristic function.}
	For deterministic $f\in \L^\alpha(\calZ_T)$, put
	\[
	H_f(s,y):=\frac{f(s,y)\mathbf1_{\{\sigma(\widetilde u(s,y))\ne0\}}}
	{\sigma(\widetilde u(s,y))},
	\]
	with the same convention at zero, and define
	\[
	a_f(t):=\int_{\calZ_t}|f(s,y)|^\alpha
	\mathbf1_{\{\sigma(\widetilde u(s,y))\ne0\}}\dd s\dd y,
	\qquad
	b_f(t):=\int_{\calZ_t}|f(s,y)|^\alpha\dd s\dd y.
	\]
	Then $\widetilde A^{H_f}(t)=a_f(t)\le b_f(T)$, and $b_f$ is
	deterministic. By \eqref{eq:extended-exponential-martingale},
	\begin{equation}\label{eq:inverse-exponential-mean}
		\widetilde\E\exp\!\left(
		i\theta\int_{\calZ_t}H_f\dd\scrM
		+|\theta|^\alpha a_f(t)\right)=1.
	\end{equation}
	Conditioning on $\overline\Fil$ fixes the integrand against
	$\widehat L$. Independence of $\widehat L$ gives
	\begin{equation}\label{eq:completion-conditional-characteristic}
		\widetilde\E\!\left[
		\exp\!\left(i\theta\int_{\calZ_t}
		f\mathbf1_{\{\sigma(\widetilde u)=0\}}\dd\widehat L\right)
		\,\middle|\,\overline\Fil\right]
		=e^{-|\theta|^\alpha[b_f(t)-a_f(t)]}.
	\end{equation}
	To justify conditioning on the whole original $\sigma$-field, work
	with $(\overline\Fil\vee\widehat\Fil_v)_{v\le T}$: independence
	ensures that $\widehat L$ is still a stable noise in this filtration.
	The integral against $\widehat L$ in
	\eqref{eq:reconstructed-noise-substitution} is unchanged,
	since its elementary approximations and probability limits agree in
	both filtrations. 
	
	Let $\calK_t$ be the class of functions
	$g:\widetilde\Omega\times\calZ_t\to\R$ that are
	$\overline\Fil\otimes\calB(\calZ_t)$-measurable and satisfy
	$\widetilde\E\int_{\calZ_t}|g(s,y)|^\alpha\dd s\dd y<\infty$.
	Every $g\in\calK_t$ is predictable with respect to
	$(\overline\Fil\vee\widehat\Fil_v)_{v\le T}$, so its integral against
	$\widehat L$ is defined. We show that, for every $g\in\calK_t$,
	\begin{equation}\label{eq:conditional-chf-general}
		\widetilde\E\!\left[\exp\!\left(i\theta\int_{\calZ_t}g\dd\widehat L\right)
		\,\middle|\,\overline\Fil\right]
		=\exp\!\left(-|\theta|^\alpha\int_{\calZ_t}|g(s,y)|^\alpha\dd s\dd y\right).
	\end{equation}
	First let $g=\sum_{i=1}^m\mathbf1_{B_i}h_i$, where
	$B_1,\dots,B_m\in\overline\Fil$ are disjoint and each $h_i$ is a finite
	linear combination of indicators of Borel subsets of $\calZ_t$.
	On $B_i$, the integral of $g$ equals the integral of the deterministic
	function $h_i$, and $\widehat L$ is independent of $\overline\Fil$; so
	\eqref{eq:conditional-chf-general} holds for this $g$. A general
	$g\in\calK_t$ is the limit in
	$\L^\alpha(\widetilde\Omega\times\calZ_t,
	\widetilde\P\otimes\operatorname{Leb})$ of a sequence $(g_m)$ of
	functions of this form. Then $\int_{\calZ_t}g_m\dd\widehat L\to
	\int_{\calZ_t}g\dd\widehat L$ in probability, so the exponentials
	on the left of \eqref{eq:conditional-chf-general}, being bounded,
	converge in $\L^1(\widetilde\Omega)$. Also
	$\int_{\calZ_t}|g_m|^\alpha\to\int_{\calZ_t}|g|^\alpha$ in
	probability, so the right-hand sides converge in
	$\L^1(\widetilde\Omega)$ as well. This proves
	\eqref{eq:conditional-chf-general} for every $g\in\calK_t$.
	
	Finally, $g=f\mathbf1_{\{\sigma(\widetilde u)=0\}}$ belongs to
	$\calK_t$, since it is
	$\overline\Fil\otimes\calB(\calZ_t)$-measurable and
	$|g|\le|f|$ with $f\in\L^\alpha(\calZ_t)$. For this $g$,
	\eqref{eq:conditional-chf-general} is
	\eqref{eq:completion-conditional-characteristic}.

	The two integrals in \eqref{eq:reconstructed-noise-substitution} need
	not be independent. We only use that the first one is
	$\overline\Fil$-measurable. Hence
	\eqref{eq:completion-conditional-characteristic} and
	\eqref{eq:inverse-exponential-mean} yield
	\begin{equation}\label{eq:reconstructed-stable-characteristic}
		\begin{aligned}
			\widetilde\E e^{i\theta\widetilde L_t(f)}
			&=\widetilde\E\left[
			e^{i\theta\int_{\calZ_t}H_f\dd\scrM}
			e^{-|\theta|^\alpha[b_f(t)-a_f(t)]}\right]\\
			&=e^{-|\theta|^\alpha b_f(t)}
			\widetilde\E\exp\!\left(
			i\theta\int_{\calZ_t}H_f\dd\scrM
			+|\theta|^\alpha a_f(t)\right)\\
			&=\exp\!\left(-|\theta|^\alpha
			\int_{\calZ_t}|f(s,y)|^\alpha\dd s\dd y\right).
		\end{aligned}
	\end{equation}
	In particular, for every $A\in\calB(\calZ_T)$,
	\[
	\widetilde\E e^{i\theta\widetilde L(A)}
	=\exp\bigl(-|\theta|^\alpha\operatorname{Leb}(A)\bigr).
	\]
	For disjoint $A_1,\ldots,A_m$, take
	$f=\sum_{j=1}^m\theta_j\mathbf1_{A_j}$ and $\theta=1$ in
	\eqref{eq:reconstructed-stable-characteristic}. Linearity gives
	\[
	\widetilde\E\exp\!\left(i\sum_{j=1}^m\theta_j\widetilde L(A_j)\right)
	=\prod_{j=1}^m
	\exp\bigl(-|\theta_j|^\alpha\operatorname{Leb}(A_j)\bigr).
	\]
	Thus, the values on disjoint sets are independent.
	Together with countable additivity already proved above, this
	identifies the restriction to deterministic sets as a symmetric
	$\alpha$-stable independently scattered random measure with Lebesgue
	control measure. For each disjoint sequence, its independent partial
	sums converge in probability, hence also almost surely, to the value
	on the union.

	\noindent\textit{Compatibility with the filtration.}
	We also need the future values of $\widetilde L$ to be independent of
	the past of the extended filtration. Let $0\le s<t\le T$. Conditionally
	on $\overline\Fil\vee\widehat\Fil_s$, the integrand of the
	$\widehat L$-integral over $(s,t]\times D$ in
	\eqref{eq:reconstructed-noise-substitution} is fixed, and the
	increments of $\widehat L$ after time $s$ are independent of
	$\overline\Fil\vee\widehat\Fil_s$. As in
	\eqref{eq:completion-conditional-characteristic}, the conditional
	characteristic function of this integral is therefore
	\[
	\exp\!\left(-|\theta|^\alpha
	[b_f(t)-b_f(s)-a_f(t)+a_f(s)]\right).
	\]
	For the $\scrM$-integral, the bounded martingale
	\eqref{eq:extended-exponential-martingale} with $H=H_f$ is adapted to
	$(\overline\Fil_t)$ and is a martingale with respect to
	$(\widetilde\Fil_t)$, hence also with respect to the smaller
	filtration $(\overline\Fil_t\vee\widehat\Fil_t)$. Since its modulus is
	at least $1$, we may divide its martingale identity between $s$ and
	$t$ by its value at $s$, which gives
	\[
	\widetilde\E\!\left[
	\exp\!\left(i\theta\int_{(s,t]\times D}H_f\dd\scrM
	+|\theta|^\alpha[a_f(t)-a_f(s)]\right)
	\,\middle|\,\overline\Fil_s\vee\widehat\Fil_s\right]=1.
	\]
	Since the $\scrM$-integral is $\overline\Fil$-measurable,
	combining these two conditional identities gives
	\begin{equation}\label{eq:reconstructed-noise-past-independence}
		\widetilde\E\!\left[
		e^{i\theta(\widetilde L_t(f)-\widetilde L_s(f))}
		\,\middle|\,\widetilde\Fil_s\right]
		=\exp\!\left(-|\theta|^\alpha
		\int_{(s,t]\times D}|f(v,y)|^\alpha\dd v\dd y\right),
	\end{equation}
	first with $\overline\Fil_s\vee\widehat\Fil_s$ in place of
	$\widetilde\Fil_s$. Completion does not change this identity. To pass
	to the right-continuous augmentation, we apply the identity at times
	$s_j\downarrow s$ with $s_j<t$ and use bounded convergence, the
	c\`adl\`ag paths of $\widetilde L_\cdot(f)$ and the continuity of
	$b_f$. Taking $f$ to be a finite linear combination of indicators of
	deterministic sets in $(s,T]\times D$ then shows that the values of
	$\widetilde L$ on such sets are independent of $\widetilde\Fil_s$.
	Adaptedness of $\widetilde L$ follows directly from
	\eqref{eq:reconstructed-noise-definition}.

	\emph{L\'evy basis verification.}
	We will show that $\widetilde L$ is a symmetric $\alpha$-stable
	L\'evy basis relative to $(\widetilde\Fil_t)$ and identify its
	canonical decomposition using Chong--Kl\"uppelberg
	\cite{ChongKlueppelberg}*{Definitions~2.1 and~3.1 and Theorem~3.2}.
	
	First, the construction following
	\eqref{eq:reconstructed-noise-definition} proves that $\widetilde L$
	vanishes at the empty set, is countably additive in probability on $\calP(\widetilde{\bbB})$, is
	$\widetilde\Fil_t$-measurable on predictable sets supported before $t$,
	and commutes with multiplication by events in $\widetilde\Fil_t$ on
	future predictable sets. Hence, $\widetilde L$ is an $\L^0(\widetilde\Omega)$-valued random measure on $\calP(\widetilde{\bbB})$
	in the sense of \cite{ChongKlueppelberg}*{Definition~2.1(1), with $p=0$}. We apply \cite{ChongKlueppelberg}*{Theorem~3.2} to the random measure
	on $[0,T]\times D$ by extending it by zero outside $[0,T]$
	and keeping the filtration fixed at $\widetilde\Fil_0$
	before $0$ and at $\widetilde\Fil_T$ after $T$. This extension by zero is used only to apply
	\cite{ChongKlueppelberg}*{Theorem~3.2}; it is not a stable noise after
	time $T$. At the end of this subsection we extend $\widetilde L$ in a
	different way, by using the jumps of the auxiliary noise $\widehat L$
	after time $T$, so that it becomes a symmetric $\alpha$-stable noise on
	$(0,\infty)$, as required in \eqref{eq:Levy-Ito}.
	
	For deterministic sets, \eqref{eq:reconstructed-stable-characteristic} gives infinitely divisible stable laws and factorization on disjoint sets, while
	\eqref{eq:reconstructed-noise-past-independence} gives independence of
	future increments from $\widetilde\Fil_t$. The Lebesgue control measure
	has no fixed-time atoms. These are precisely the remaining L\'evy-basis
	requirements (5)--(7) in Chong \cite{ChongVolterra}*{Section~2}.
	
	Next, let $U_1,U_2\in\calB(D)$ be disjoint. By
	\eqref{eq:reconstructed-stable-characteristic}, the processes
	$t\mapsto\widetilde L((0,t]\times U_1)$ and
	$t\mapsto\widetilde L((0,t]\times U_2)$ are independent. Since neither
	of them has fixed jump times, they almost surely never jump at the
	same time. Hence $\widetilde L$ has different times of discontinuity
	in the sense of \cite{ChongKlueppelberg}*{Definition~3.1}, and
	\cite{ChongKlueppelberg}*{Theorem~3.2} gives its predictable
	characteristics $(\frakB,\frakC,\nu^{\widetilde L})$ and its jump measure $\calN$.
	
	Here $\calN$ is
	the random measure on $(0,T]\times D\times(\R\setminus\{0\})$ that puts
	a unit mass at $(s,y,z)$ whenever $\widetilde L$ has a jump of size $z$
	at the space--time point $(s,y)$. Thus $\calN((s,t]\times U\times V)$ is
	the number of jumps of $\widetilde L$ in $(s,t]\times U$ with size in
	$V$.

	It remains to identify these characteristics in the enlarged
	filtration. Let $U\in\calB(D)$. By
	\eqref{eq:reconstructed-noise-past-independence} with
	$f=\mathbf1_{U}$, the process $t\mapsto\widetilde L((0,t]\times U)$
	is a symmetric $\alpha$-stable L\'evy process with respect to
	$(\widetilde\Fil_t)$, with L\'evy measure
	$\operatorname{Leb}(U)\nu_\alpha$. In particular, it has no drift and
	no Gaussian part, so $\frakB=\frakC=0$. Moreover, if $F\in\widetilde\Fil_s$ and
	$V\subset\R\setminus\{0\}$ is bounded away from zero, then
	\[
	\widetilde\E\left[\mathbf1_F
	\calN((s,t]\times U\times V)\right]
	=\widetilde\P(F)(t-s)\operatorname{Leb}(U)\nu_\alpha(V).
	\]
	Since the compensator of $\calN$ is determined by its values on such
	predictable rectangles, we obtain
	\[
	\nu^{\widetilde L}(\dd s,\dd y,\dd z)
	=\dd s\,\dd y\,\nu_\alpha(\dd z)=\m(\dd s,\dd y,\dd z).
	\]
	
	The jump measure $\calN$ has at most one point at each time.
	Its compensator with respect to $(\widetilde\Fil_t)$ is the
	deterministic measure $\m$, which gives zero mass to every time
	slice. The characterization of Poisson random measures
	\cite{JS}*{Theorem~II.4.8} therefore shows that $\calN$ is an
	$(\widetilde\Fil_t)$-Poisson random measure with intensity $\m$.
	
	We conclude that $\widetilde L$ is a L\'evy basis with respect to
	$(\widetilde\Fil_t)$ and satisfies properties (1)--(7) of
	\cite{ChongVolterra}*{Section~2}. In particular, for each
	$A\in\calP(\widetilde{\bbB})$,
	$\widetilde L$ has the L\'evy--It\^o decomposition
	\cite{ChongVolterra}*{(2.2)--(2.3)}
	\begin{equation}\label{eq:reconstructed-Levy-Ito}
		\begin{aligned}
			\widetilde L(A)
			={}&\int_{\calZ_T\times\{0<|z|\le1\}}
			\mathbf1_A(s,y)z\,(\calN-\m)(\dd s,\dd y,\dd z)\\
			&+\int_{\calZ_T\times\{|z|>1\}}
			\mathbf1_A(s,y)z\,\calN(\dd s,\dd y,\dd z),
			\qquad\widetilde\P\text{-a.s.},
		\end{aligned}
	\end{equation}
	where $\calN$ is the $(\widetilde\Fil_t)$-Poisson random measure
	with intensity $\m$ obtained above.
	
	So far, $\widetilde L$ is defined only on $[0,T]$, while
	\eqref{eq:Levy-Ito} uses the time interval $(0,\infty)$. To extend it,
	let $\calN^\star$ be the Poisson random measure of the auxiliary
	noise $\widehat L$, and define, for Borel sets
	$A\subset(0,\infty)\times D\times(\R\setminus\{0\})$,
	\[
	\calN^\infty(A)
	:=\calN\bigl(A\cap((0,T]\times D\times(\R\setminus\{0\}))\bigr)
	+\calN^\star\bigl(A\cap((T,\infty)\times D\times(\R\setminus\{0\}))\bigr).
	\]
	For $t>T$, let $\widetilde\Fil_t$ be generated by $\widetilde\Fil_T$
	and the restriction of $\calN^\star$ to
	$(T,t]\times D\times(\R\setminus\{0\})$, completed as before. Since
	the jumps of $\calN^\star$ after time $T$ are independent of
	$\widetilde\Fil_T$, the measure $\calN^\infty$ is a Poisson random
	measure with intensity $\m$ with respect to this filtration. The noise
	defined from $\calN^\infty$ by \eqref{eq:Levy-Ito} agrees with
	$\widetilde L$ on $[0,T]$. We keep the notation $\calN$ and
	$\widetilde L$ for these extensions; this does not change the state
	process or any identity on $[0,T]$.
	
	
	Fix $\varphi\in\calE$. By linear growth and
	\eqref{eq:limit-L2}, the field $f=\sigma(\widetilde u)\varphi$
	satisfies $\int_{\calZ_T}|f|^\alpha\dd s\dd y<\infty$,
	$\widetilde\P$-a.s., so \eqref{eq:reconstructed-noise-substitution}
	applies to it. Since $f\mathbf1_{\{\sigma(\widetilde u)=0\}}=0$, we
	obtain, with probability one, for all $t\in[0,T]$,
	\[
	\int_{\calZ_t}\sigma(\widetilde u)\varphi\dd\widetilde L
	=\int_{\calZ_t}\varphi
	\mathbf1_{\{\sigma(\widetilde u)\ne0\}}\dd\scrM
	=\int_{\calZ_t}\varphi\dd\scrM
	=\widetilde M^\varphi(t).
	\]
	The second equality holds because
	$H=\varphi\mathbf1_{\{\sigma(\widetilde u)=0\}}$ satisfies
	$\widetilde A^H(T)=0$, since
	$\mu_{\widetilde u}(\{\sigma(\widetilde u)=0\})=0$, and
	\eqref{eq:elementary-integral-continuity} then shows that its integral
	is zero. The last equality is the definition of the elementary
	integral. All these processes are c\`adl\`ag, so the equalities hold
	for all $t$ on one event of probability one. By the definition
	\eqref{eq:limit-martingale-coordinates} of $\widetilde M^\varphi$, this
	is the weak equation in Definition~\ref{def:weak-solution}(ii).
	
	It remains to verify condition~(i). Fix $t\in(0,T]$. Since the
	intensity $\m$ gives no mass to $\{t\}\times D\times(\R\setminus\{0\})$,
	the Poisson random measure $\calN$ has almost surely no points at
	time $t$. Thus $\widetilde M^{e_k}$ does not jump at $t$, almost
	surely, for every $k$. The other terms in the weak equation are
	continuous in time, so
	$\langle\widetilde u(t)-\widetilde u(t-),e_k\rangle=0$ for all $k$,
	almost surely. This gives $\widetilde u(t-)=\widetilde u(t)$ in
	$\H_{-r}$, $\widetilde\P$-a.s., where the left limit is taken in
	$\H_{-r}$. In particular, $\widetilde u$ is almost surely continuous at
	$t$, and \eqref{eq:skor} gives $\widetilde u_n(t)\to\widetilde u(t)$ in
	$\H_{-r}$, $\widetilde\P$-a.s.
	
	We now show that $\widetilde u(t)\in\L^2(D)$ almost surely. For
	$R>0$, let
	$B_R:=\{h\in\L^2(D):\|h\|_{\L^2(D)}^2\le R\}$, viewed as
	a subset of $\H_{-r}$. This set is closed in $\H_{-r}$: if $h_j\in B_R$
	and $h_j\to h$ in $\H_{-r}$, then a subsequence converges weakly in
	$\L^2(D)$ to some $g$ with $\|g\|_{\L^2(D)}^2\le R$, and
	testing against each $e_k$ shows that $g=h$. On $\{\tauK>T\}$ the
	states $u_n(t)$ and $u_n^K(t)$ agree, so Lemma~\ref{lem1} with $q=2$
	and Markov's inequality give
	\[
	\P\bigl(u_n(t)\notin B_R\bigr)
	\le\P(\tauK\le T)+\frac{\Cbound{2}{K}}{R},
	\qquad n\ge1.
	\]
	By equality in law, the same bound holds for $\widetilde u_n(t)$.
	Since $\widetilde u_n(t)\to\widetilde u(t)$ almost surely and the
	complement of $B_R$ is open, the Portmanteau theorem
	\cite{EK}*{Chapter~3, Theorem~3.1(e)} gives the same
	bound for $\widetilde\P(\widetilde u(t)\notin B_R)$. Letting first
	$R\to\infty$ with $K$ fixed, and then $K\to\infty$, we obtain
	$\widetilde u(t)\in\L^2(D)$, $\widetilde\P$-a.s.
	
	Finally, the field $\widetilde u(t,\cdot)$ is the representative of
	Lemma~\ref{lem:predictable-representative}, which is built from the
	Fej\'er means of $\widetilde u(t-)$. When
	$\widetilde u(t-)\in\L^2(D)$, these means converge almost
	everywhere to its density. Since
	$\widetilde u(t-)=\widetilde u(t)\in\L^2(D)$ almost surely, we
	obtain
	\[
	\widetilde\P\!\left(
	\langle\widetilde u(t),\varphi\rangle
	=\int_D\widetilde u(t,y)\varphi(y)\dd y
	\quad\text{for all }\varphi\in\calE
	\right)=1.
	\]
	This is condition~(i); the exceptional event may depend on $t$.
	Condition~(iii) follows from \eqref{eq:limit-L2}, linear growth and
	$\alpha<2$. Together with $\widetilde u(0)=u_0$ and the c\`adl\`ag
	paths of the state, this completes the proof of
	Theorem~\ref{thm:weak-existence} on the extended stochastic basis
	$\widetilde{\bbB}$.

	\section{Mild representation}\label{sec:mild-representation}
	\label{subsec:proof-mild-existence}
	
	This section proves Theorem~\ref{thm:mild-existence}. We first obtain
	the pointwise kernel estimates needed to define the mild map. We then
	derive a regularized identity, remove the regularization, and identify
	the resulting predictable field with the distribution-valued state.
	
	\subsection{Kernel integrability and the mild map}
	\label{subsec:mild-kernel-map}
	
	We first show that both integrals in the mild formula are well defined
	at each fixed $(t,x)$. We use the field $\widetilde u\in\calP(\widetilde{\bbB})$
	and the noise $\widetilde L$ constructed in the preceding subsection.
	In integrals containing $\rho_{t-s}$, the kernel integrands are set
	to zero at $s=t$.
	
	For each fixed $(t,x)\in(0,T]\times D$, define the functional
	$\calV(\cdot)(t,x):\calL\to[0,\infty]$ by
	\begin{equation}\label{eq:mild-kernel-integrability}
		\calV(v)(t,x)
		:=\int_{\calZ_t}\rho_{t-s}(x,y)|b(v(s,y))|\dd s\dd y
		+\int_{\calZ_t}\rho_{t-s}(x,y)^\alpha
		|\sigma(v(s,y))|^\alpha\dd s\dd y,
	\end{equation}
	where $v\in\calL=\L^q((0,T)\times D)$.
	The value does not depend on the chosen representative of $v$.
	We define $\calV_n(\cdot)(t,x):\calL\to[0,\infty]$
	by replacing $b,\sigma$ with $b_n,\sigma_n$.
	
	\begin{lemma}\label{lem:mild-functional-measurability}
		For each fixed deterministic $(t,x)\in(0,T]\times D$ and every $n\ge1$,
		the functionals
		\[
		\calV(\cdot)(t,x),\ \calV_n(\cdot)(t,x):
		(\calL,\calB(\calL))
		\longrightarrow([0,\infty],\calB([0,\infty]))
		\]
		are measurable. In fact, both functionals are
		lower semicontinuous in the norm topology of $\calL$.
	\end{lemma}
	\begin{proof}
		Fix $n$ and $(t,x)$, and let $v_j\to v$ in $\calL$. We show that
		\[
		\calV_n(v)(t,x)
		\le\liminf_{j\to\infty}\calV_n(v_j)(t,x).
		\]
		Choose a subsequence along which $\calV_n(v_j)(t,x)$ converges
		to the lower limit on the right, which may be infinite. Since
		convergence in $\L^q(\calZ_T)$ implies almost-everywhere
		convergence along a subsequence, we may pass to a further
		subsequence, still denoted $(v_j)$, such that $v_j\to v$ Leb-a.e.\ on
		$\calZ_T$. Because $b_n$ and $\sigma_n$ are continuous, the
		nonnegative integrands
		$\rho_{t-s}(x,y)|b_n(v_j(s,y))|$ and
		$\rho_{t-s}(x,y)^\alpha|\sigma_n(v_j(s,y))|^\alpha$
		converge almost everywhere to the corresponding integrands for $v$.
		Fatou's lemma then gives the inequality. The same argument applies
		to $\calV$, since $b$ and $\sigma$ are also continuous.
		
		A lower semicontinuous function is Borel measurable, because the sets
		$\{v:\calV_n(v)(t,x)>a\}$, $a\in\R$, are open. This proves the
		lemma.
	\end{proof}
	
	\begin{lemma}\label{lem:mild-kernel-integrability}
		For every fixed $(t,x)\in(0,T]\times D$,
		$\calV(\widetilde u)(t,x)<\infty$, $\widetilde\P$-a.s. More
		precisely, for every $R>0$ and every integer $K\ge1$,
		\[
		\widetilde\P\bigl(\calV(\widetilde u)(t,x)>R\bigr)
		\le\P(\tauK\le T)+\frac{\calC_{6,1}(K)}{R},
		\]
		where $\calC_{6,1}(K)$ is the constant in
		\eqref{eq:mild-convolution-constant}. It depends only on
		$\alpha,T,\Cgrow,\|u_0\|_{\L^\infty(D)}$ and $m_2(K)$, and not
		on $n$, $t$, $x$ or $R$.
	\end{lemma}
	\begin{proof}
		Fix $(t,x)\in(0,T]\times D$, $R>0$ and $K\ge1$.
		
		\emph{Step 1: the approximations.} By \eqref{eq:approx-growth} and
		Lemma~\ref{lem1} with $q=2$, the second moments of
		$b_n(u_n^K(s,y))$ and $\sigma_n(u_n^K(s,y))$ are bounded by $J_2^K$,
		uniformly in $n,s,y$. Jensen's inequality, Tonelli's theorem and the
		heat-kernel bounds $\int_D\rho_{t-s}(x,y)\dd y\le1$ and
		$\int_D\rho_{t-s}(x,y)^\alpha\dd y\le(2\pi(t-s))^{-(\alpha-1)/2}$
		then give
		\[
		\E\,\calV_n(u_n^K)(t,x)\le\calC_{6,1}(K),\qquad n\ge1.
		\]
		On $\{\tauK>T\}$, the fields $u_n$ and $u_n^K$ agree Leb-a.e.\ on
		$\calZ_T$, so $\calV_n(u_n)(t,x)=\calV_n(u_n^K)(t,x)$
		there. By Markov's inequality,
		\begin{equation}\label{eq:mild-kernel-tail}
			\sup_{n\ge1}\P\bigl(\calV_n(u_n)(t,x)>R\bigr)
			\le\P(\tauK\le T)+\frac{\calC_{6,1}(K)}{R}.
		\end{equation}
		
		\emph{Step 2: transfer to $\widetilde u_n$.} By
		Lemma~\ref{lem:mild-functional-measurability}, the map
		$v\mapsto\calV_n(v)(t,x)$ is Borel measurable on $\calL$.
		Since $\widetilde u_n$ and $u_n$ have the same law on $\calL$,
		the bound \eqref{eq:mild-kernel-tail} also holds with
		$\widetilde\P$ and $\widetilde u_n$ in place of $\P$ and $u_n$.
		
		\emph{Step 3: the limit.} Almost surely, $\widetilde u_n\to\widetilde u$
		Leb-a.e.\ on $\calZ_T$, as shown in
		Section~\ref{subsec:joint-convergence}. Since $b_n\to b$ and
		$\sigma_n\to\sigma$ locally uniformly and $b,\sigma$ are continuous,
		we also have $b_n(\widetilde u_n)\to b(\widetilde u)$ and
		$\sigma_n(\widetilde u_n)\to\sigma(\widetilde u)$ Leb-a.e. The
		predictable representative of $\widetilde u$ agrees with this limit
		field Leb-a.e., so it gives the same value of $\calV$. By
		Fatou's lemma,
		\[
		\calV(\widetilde u)(t,x)
		\le\liminf_{n\to\infty}\calV_n(\widetilde u_n)(t,x),
		\qquad\widetilde\P\text{-a.s.}
		\]
		So, if $\calV(\widetilde u)(t,x)>R$, then
		$\calV_n(\widetilde u_n)(t,x)>R$ for all large $n$, and
		\[
		\widetilde\P\bigl(\calV(\widetilde u)(t,x)>R\bigr)
		\le\liminf_{n\to\infty}
		\widetilde\P\bigl(\calV_n(\widetilde u_n)(t,x)>R\bigr)
		\le\P(\tauK\le T)+\frac{\calC_{6,1}(K)}{R}.
		\]
		This is the stated bound. Letting $R\to\infty$ gives
		$\widetilde\P(\calV(\widetilde u)(t,x)=\infty)\le
		\P(\tauK\le T)$ for every $K$, and the right-hand side tends to
		zero as $K\to\infty$.
	\end{proof}
	
	\noindent\textit{The mild-solution map.}
	Let $v\in\calP(\widetilde{\bbB})$ be such that
	$v\in\calL$, $\widetilde\P$-a.s., and
	$\calV(v)(t,x)<\infty$, $\widetilde\P$-a.s., for each fixed
	$(t,x)\in(0,T]\times D$. For such $(t,x)$, define
	\begin{equation}\label{eq:mild-solution-map}
		\begin{aligned}
			\calT(v)(t,x)
			:={}&P_tu_0(x)
			+\int_{\calZ_t}\rho_{t-s}(x,y)b(v(s,y))\dd s\dd y\\
			&+\int_{\calZ_t}\rho_{t-s}(x,y)\sigma(v(s,y))
			\widetilde L(\dd s,\dd y),
		\end{aligned}
	\end{equation}
	and set $\calT(v)(0,x):=u_0(x)$. Both integrals are well defined.
	The drift integral converges absolutely, since the first term of
	$\calV(v)(t,x)$ is finite, $\widetilde\P$-a.s. The stochastic
	integral exists by \eqref{eq:reconstructed-noise-substitution}, since
	the second term,
	$\int_{\calZ_t}\rho_{t-s}(x,y)^\alpha|\sigma(v(s,y))|^\alpha\dd s\dd y$,
	is finite, $\widetilde\P$-a.s. Thus $\calT(v)(t,x)$ is a finite
	random variable for each fixed $(t,x)$.
	
	We next choose a predictable version of $\calT(v)$. By
	\cite{ChongVolterra}*{Lemma~6.2}, the stochastic integral in
	\eqref{eq:mild-solution-map} has a predictable modification, that is,
	a predictable field that agrees with it $\widetilde\P$-a.s.\ at
	each fixed $(t,x)$. Since that lemma is stated on the whole space, we
	extend $\widetilde L$ and $\sigma(v)$ by zero for $y\notin D$, and the
	kernel $\rho_{t-s}(x,y)$ by zero for $y\notin D$ and for $s\ge t$; we
	apply the lemma with the identity function as coefficient. The
	extended noise is a L\'evy basis with control measure
	$\mathbf1_D(y)\dd s\dd y$. The same argument, with Lebesgue measure in
	place of $\widetilde L$, gives a predictable version of the drift
	integral. Adding the deterministic function $P_tu_0(x)$, we obtain a
	predictable modification of $\calT(v)$. Setting it equal to zero
	where it is not finite changes it only on a null set at each fixed
	$(t,x)$. From now on, $\calT(v)$ denotes this modification, which
	belongs to $\calP(\widetilde{\bbB})$.
	
	By Lemma~\ref{lem:mild-kernel-integrability}, the field
	$v=\widetilde u$ satisfies these assumptions, so the predictable field
	$\calT(\widetilde u)$ is well defined. It remains to show that it
	is a representative of the weak solution.
	
	\subsection{The regularized mild identity}
	\label{subsec:regularized-mild-operator}
	
	For $\varepsilon>0$ and a field $v\in\calP(\widetilde{\bbB})$ satisfying
	$\int_{\calZ_T}(|b(v)|+|\sigma(v)|^\alpha)\dd s\dd y<\infty$
	$\widetilde\P$-a.s., define
	\begin{equation}\label{eq:regularized-mild-map}
		\begin{aligned}
			\calT_\varepsilon(v)(t,x)
			:={}&P_{t+\varepsilon}u_0(x)
			+\int_{\calZ_t}\rho_{t-s+\varepsilon}(x,y)
			b(v(s,y))\dd s\dd y\\
			&+\int_{\calZ_t}\rho_{t-s+\varepsilon}(x,y)
			\sigma(v(s,y))\widetilde L(\dd s,\dd y).
		\end{aligned}
	\end{equation}
	Only the semigroup times are shifted; the integration domain remains
	$\calZ_t$, and the coefficients are evaluated at $v(s,y)$.
	Since $0\le\rho_{t-s+\varepsilon}(x,y)\le(2\pi\varepsilon)^{-1/2}$
	for $s\le t$, both integrals are well defined at every fixed $(t,x)$.
	The field $\widetilde u$ satisfies the stated integrability condition
	by \eqref{eq:weak-integrability}.
	
	\begin{lemma}\label{lem:regularized-mild-identity}
		For every fixed $\varepsilon>0$ and $(t,x)\in(0,T]\times D$,
		\begin{equation}\label{eq:regularized-mild-identity}
			\calT_\varepsilon(\widetilde u)(t,x)
			=(P_\varepsilon\widetilde u(t))(x),
			\qquad\widetilde\P\text{-a.s.}
		\end{equation}
		Here $P_\varepsilon$ acts on the $\H_{-r}$-valued state
		$\widetilde u(t)$. The exceptional event may depend on
		$\varepsilon,t,x$.
	\end{lemma}
	\begin{proof}
		Fix $\varepsilon>0$, $t>0$ and $x\in D$.
		
		\smallskip
		\noindent\textit{Step 1. Apply an integrating factor to each weak identity.}
		Since $\tfrac12e_k''=-\lambda_ke_k$, the weak equation gives
		\[
		\dd\langle\widetilde u(s),e_k\rangle
		=-\lambda_k\langle\widetilde u(s),e_k\rangle\dd s
		+\langle b(\widetilde u(s,\cdot)),e_k\rangle\dd s
		+\dd\widetilde M^{e_k}(s).
		\]
		The drift pairing is defined for almost every $s$, which suffices
		for its time integral. The process $\widetilde M^{e_k}$ is a local
		martingale. Subsection~\ref{subsec:stable-noise-reconstruction} gives
		$\widetilde M^{e_k}(a)=\int_{\calZ_a}e_k(y)
		\sigma(\widetilde u(s,y))\widetilde L(\dd s,\dd y)$.
		The semimartingale product rule with the deterministic factor
		$e^{\lambda_ks}$ therefore yields
		\begin{equation}\label{eq:limit-coefficient-variation}
			\begin{aligned}
				\langle\widetilde u(t),e_k\rangle
				={}&e^{-\lambda_kt}\langle u_0,e_k\rangle\\
				&+\int_{\calZ_t}e^{-\lambda_k(t-s)}e_k(y)
				b(\widetilde u(s,y))\dd s\dd y\\
				&+\int_{\calZ_t}e^{-\lambda_k(t-s)}e_k(y)
				\sigma(\widetilde u(s,y))\widetilde L(\dd s,\dd y).
			\end{aligned}
		\end{equation}
		Writing the last term as a space--time integral uses the
		multiplication property of stochastic integration: it holds first
		for step functions of time, then for the factor $e^{-\lambda_k (t-s)}$ by uniform approximation. 
		
		\smallskip
		\noindent\textit{Step 2. Sum finitely many identities.}
		For $m\ge1$, put
		\[
		\rho_v^{(m)}(x,y):=\sum_{k=1}^m
		e^{-\lambda_kv}e_k(x)e_k(y),\qquad v>0,
		\qquad
		S_{\varepsilon,m}(t,x):=\sum_{k=1}^m
		e^{-\lambda_k\varepsilon}e_k(x)\langle\widetilde u(t),e_k\rangle.
		\]
		Multiplying \eqref{eq:limit-coefficient-variation} by
		$e^{-\lambda_k\varepsilon}e_k(x)$ and summing over $k\le m$ gives
		\begin{equation}\label{eq:regularized-finite-identity}
			\begin{aligned}
				S_{\varepsilon,m}(t,x)
				={}&\int_D\rho_{t+\varepsilon}^{(m)}(x,y)u_0(y)\dd y\\
				&+\int_{\calZ_t}\rho_{t-s+\varepsilon}^{(m)}(x,y)
				b(\widetilde u(s,y))\dd s\dd y\\
				&+\int_{\calZ_t}\rho_{t-s+\varepsilon}^{(m)}(x,y)
				\sigma(\widetilde u(s,y))\widetilde L(\dd s,\dd y)
				=:\calT_{\varepsilon,m}(\widetilde u)(t,x).
			\end{aligned}
		\end{equation}
		Taking the intersection
		of the probability-one events for the countably many $k$, the
		identities \eqref{eq:regularized-finite-identity} hold for every
		$m$ on one event of probability one.
		
		\smallskip
		\noindent\textit{Step 3. Pass to the limit on the right-hand side.}
		By $|e_k|\le\sqrt2$ and \eqref{eq:heat-kernel},
		\begin{equation}\label{eq:regularized-kernel-tail}
			\sup_{\substack{v\ge\varepsilon\\x,y\in D}}
			|\rho_v^{(m)}(x,y)-\rho_v(x,y)|
			\le 2\sum_{k>m}e^{-\lambda_k\varepsilon}
			=:d_{m,\varepsilon}\longrightarrow0 \quad\text{as } m\to\infty.
		\end{equation}
		The differences in the initial and drift terms are bounded,
		respectively, by
		\[
		d_{m,\varepsilon}\|u_0\|_{\L^1(D)},
		\qquad
		d_{m,\varepsilon}\int_{\calZ_t}
		|b(\widetilde u(s,y))|\dd s\dd y.
		\]
		Both tend to zero, the second $\widetilde\P$-a.s., by
		\eqref{eq:weak-integrability}.
		
		For the stochastic term define, on $\calZ_t$,
		\[
		H_{\varepsilon,m}^{t,x}(s,y)
		:=\bigl(\rho_{t-s+\varepsilon}^{(m)}(x,y)
		-\rho_{t-s+\varepsilon}(x,y)\bigr)
		\sigma(\widetilde u(s,y)).
		\]
		This is predictable, and
		\begin{equation}\label{eq:regularized-stochastic-error}
			\int_{\calZ_t}|H_{\varepsilon,m}^{t,x}(s,y)|^\alpha\dd s\dd y
			\le d_{m,\varepsilon}^\alpha
			\int_{\calZ_t}|\sigma(\widetilde u(s,y))|^\alpha\dd s\dd y
			\longrightarrow0,\qquad\widetilde\P\text{-a.s.}
		\end{equation}
		Equation~\eqref{eq:reconstructed-noise-substitution} identifies its
		integral with the $\scrM$-integral of the kernel difference, and
		the corresponding value $\widetilde A^{H}(T)$ of its integrand $H$
		against $\scrM$ is exactly the
		left side of \eqref{eq:regularized-stochastic-error}.
		Thus, \eqref{eq:elementary-integral-continuity}, extended by the
		construction to integrable fields, gives, for $\delta,R>0$,
		\[
		\begin{aligned}
			&\widetilde\P\left(
			\left|\int_{\calZ_t}H_{\varepsilon,m}^{t,x}(s,y)
			\widetilde L(\dd s,\dd y)\right|
			>\delta\right)\\
			&\quad\le
			\widetilde\P\left(
			\int_{\calZ_t}|H_{\varepsilon,m}^{t,x}|^\alpha\dd s\dd y\ge R
			\right)
			+\frac{\calC_{5,1}(\alpha)R^{1/\alpha}}{\delta}.
		\end{aligned}
		\]
		First let $m\to\infty$ with $R$ fixed, and then let $R\downarrow0$.
		The stochastic difference tends to zero in probability. Combining
		the three terms proves
		\begin{equation}\label{eq:regularized-operator-partial-limit}
			\calT_{\varepsilon,m}(\widetilde u)(t,x)
			\longrightarrow\calT_\varepsilon(\widetilde u)(t,x)
			\quad\text{in }\widetilde\P\text{-probability}.
		\end{equation}
		This step uses only the almost sure bound \eqref{eq:weak-integrability},
		and no moments of $\widetilde u$.
		
		\smallskip
		\noindent\textit{Step 4. Pass to the limit on the left-hand side.}
		The spectral definition of $P_\varepsilon$ and Cauchy--Schwarz give
		\[
		\left|S_{\varepsilon,m}(t,x)
		-(P_\varepsilon\widetilde u(t))(x)\right|
		\le\|\widetilde u(t)\|_{\H_{-r}}
		\left(2\sum_{k>m}\lambda_k^r
		e^{-2\lambda_k\varepsilon}\right)^{1/2}
		\longrightarrow0,\qquad\widetilde\P\text{-a.s.}
		\]
		The series converges since $\lambda_k=\pi^2k^2/2$ and
		$\varepsilon>0$. 
		
		\smallskip
		\noindent\textit{Step 5. Identify the two limits.}
		For every $\delta>0$, \eqref{eq:regularized-finite-identity} implies
		\[
		\begin{aligned}
			&\widetilde\P\left(
			|(P_\varepsilon\widetilde u(t))(x)
			-\calT_\varepsilon(\widetilde u)(t,x)|>\delta\right)\\
			&\quad\le
			\widetilde\P\left(
			|(P_\varepsilon\widetilde u(t))(x)-S_{\varepsilon,m}(t,x)|
			>\delta/2\right)\\
			&\qquad+
			\widetilde\P\left(
			|\calT_{\varepsilon,m}(\widetilde u)(t,x)
			-\calT_\varepsilon(\widetilde u)(t,x)|>\delta/2\right).
		\end{aligned}
		\]
		Both terms on the right tend to zero. The probability on the left
		does not depend on $m$, so it is zero. This proves
		\eqref{eq:regularized-mild-identity}.
	\end{proof}
	
	\subsection{Removing the regularization}
	\label{subsec:removing-regularization}
	
	Lemma~\ref{lem:regularized-mild-identity} identifies
	$\calT_\varepsilon(\widetilde u)(t,x)$ with
	$(P_\varepsilon\widetilde u(t))(x)$ for each fixed $\varepsilon>0$,
	but its proof uses the bound \eqref{eq:regularized-kernel-tail}, which
	depends on $\varepsilon$. To let $\varepsilon\downarrow0$, we compare
	the shifted kernel $\rho_{t-s+\varepsilon}$ with $\rho_{t-s}$
	directly. For $p\in\{1,\alpha\}$, set
	\begin{equation}\label{eq:kernel-shift-error}
		\calR_{\varepsilon,p}(t,x)
		:=\int_{\calZ_t}
		|\rho_{t-s+\varepsilon}(x,y)-\rho_{t-s}(x,y)|^p\dd s\dd y,
	\end{equation}
	and, for $v\in\calL$,
	\begin{equation}\label{eq:mild-shift-error-functional}
		\begin{aligned}
			\calV_\varepsilon(v)(t,x)
			:={}&\int_{\calZ_t}
			|\rho_{t-s+\varepsilon}(x,y)-\rho_{t-s}(x,y)|
			|b(v(s,y))|\dd s\dd y\\
			&+\int_{\calZ_t}
			|\rho_{t-s+\varepsilon}(x,y)-\rho_{t-s}(x,y)|^\alpha
			|\sigma(v(s,y))|^\alpha\dd s\dd y.
		\end{aligned}
	\end{equation}
	In these integrands we set $\rho_0(x,y):=0$ at $s=t$. This does not
	change the integrals, since the time slice $\{s=t\}$ has Lebesgue
	measure zero and $\widetilde L$ almost surely has no jump at the fixed
	time $t$. We write $\calV_{\varepsilon,n}$ for the same functional
	with $b_n,\sigma_n$ in place of $b,\sigma$.

	\begin{lemma}\label{lem:remove-mild-regularization}
		For each fixed $(t,x)\in(0,T]\times D$,
		\begin{equation}\label{eq:mild-regularization-limit}
			\calT_\varepsilon(\widetilde u)(t,x)
			\longrightarrow\calT(\widetilde u)(t,x)
			\quad\text{in }\widetilde\P\text{-probability}
			\quad\text{as }\varepsilon\downarrow0.
		\end{equation}
		More precisely, for every $\varepsilon,R>0$ and every integer
		$K\ge1$,
		\begin{equation}\label{eq:mild-shift-error-tail}
			\widetilde\P\bigl(\calV_\varepsilon(\widetilde u)(t,x)>R\bigr)
			\le\P(\tauK\le T)
			+\frac1R\sum_{p\in\{1,\alpha\}}
			\calC_{6,2}(K,p)\calR_{\varepsilon,p}(t,x),
		\end{equation}
		where $\calC_{6,2}(K,p)=(J_2^K)^{p/2}$, as in
		\eqref{eq:mild-shift-constant}. These constants do not depend on
		$n$, $\varepsilon$, $t$, $x$ or $R$.
	\end{lemma}
	\begin{proof}
		Fix $(t,x)\in(0,T]\times D$.
		
		\emph{Step 1: the kernel errors tend to zero.} Fix $a>0$. By the
		spectral representation \eqref{eq:heat-kernel},
		$\rho_{a+\varepsilon}\to\rho_a$ uniformly on $D\times D$ as
		$\varepsilon\downarrow0$. Moreover, for $p\in\{1,\alpha\}$, the
		bound $\int_D\rho_a(x,y)^p\dd y\le(2\pi a)^{-(p-1)/2}$ from the proof
		of Lemma~\ref{lem1} gives
		\[
		\int_D|\rho_{a+\varepsilon}(x,y)-\rho_a(x,y)|^p\dd y
		\le 2^{p-1}\int_D
		\bigl(\rho_{a+\varepsilon}(x,y)^p+\rho_a(x,y)^p\bigr)\dd y
		\le 2^p(2\pi a)^{-(p-1)/2}.
		\]
		The right-hand side is integrable over $a\in(0,t)$, because
		$(p-1)/2<1$. With $a=t-s$, dominated convergence gives
		\begin{equation}\label{eq:kernel-shift-error-limit}
			\calR_{\varepsilon,1}(t,x)\longrightarrow0
			\quad\text{and}\quad
			\calR_{\varepsilon,\alpha}(t,x)\longrightarrow0
			\qquad\text{as }\varepsilon\downarrow0.
		\end{equation}
		
		\emph{Step 2: the error functional tends to zero.} By
		\eqref{eq:approx-growth} and Lemma~\ref{lem1} with $q=2$, the second
		moments of $b_n(u_n^K(s,y))$ and $\sigma_n(u_n^K(s,y))$ are bounded
		by $J_2^K$, uniformly in $n,s,y$. Jensen's inequality and Tonelli's
		theorem give
		\[
		\E\,\calV_{\varepsilon,n}(u_n^K)(t,x)
		\le (J_2^K)^{1/2}\calR_{\varepsilon,1}(t,x)
		+(J_2^K)^{\alpha/2}\calR_{\varepsilon,\alpha}(t,x),
		\qquad n\ge1.
		\]
		From here, the argument is the same as in Steps~2 and~3 of the proof
		of Lemma~\ref{lem:mild-kernel-integrability}, with
		$\calV_{\varepsilon,n}$ and $\calV_\varepsilon$ in place of
		$\calV_n$ and $\calV$. We use that $u_n=u_n^K$ Leb-a.e.\ on
		$\{\tauK>T\}$, Markov's inequality, equality in law, and Fatou's
		lemma. The map $v\mapsto\calV_{\varepsilon,n}(v)(t,x)$ is Borel
		measurable on $\calL$ by the proof of
		Lemma~\ref{lem:mild-functional-measurability}, since its weights
		are nonnegative. This proves \eqref{eq:mild-shift-error-tail}. For
		fixed $R>0$, we first let $\varepsilon\downarrow0$ with $K$ fixed,
		using \eqref{eq:kernel-shift-error-limit}, and then let
		$K\to\infty$. We obtain
		\begin{equation}\label{eq:mild-shift-functional-limit}
			\calV_\varepsilon(\widetilde u)(t,x)\longrightarrow0
			\quad\text{in }\widetilde\P\text{-probability}.
		\end{equation}
		
		\emph{Step 3: convergence of the mild map.} We compare the three
		terms of $\calT_\varepsilon(\widetilde u)(t,x)$ and
		$\calT(\widetilde u)(t,x)$. For the initial term, Step~1 with
		$a=t$ gives
		\[
		|P_{t+\varepsilon}u_0(x)-P_tu_0(x)|
		\le\|u_0\|_{\L^1(D)}
		\sup_{y\in D}|\rho_{t+\varepsilon}(x,y)-\rho_t(x,y)|
		\longrightarrow0.
		\]
		The difference of the drift terms is bounded in absolute value by
		$\calV_\varepsilon(\widetilde u)(t,x)$, so it tends to zero in
		probability by \eqref{eq:mild-shift-functional-limit}. For the
		stochastic terms, put
		\[
		H_\varepsilon^{t,x}(s,y)
		:=\bigl(\rho_{t-s+\varepsilon}(x,y)-\rho_{t-s}(x,y)\bigr)
		\sigma(\widetilde u(s,y)),
		\qquad(s,y)\in\calZ_t.
		\]
		Both stochastic integrals are well defined, by
		Lemma~\ref{lem:mild-kernel-integrability} and since the shifted
		kernel is bounded by $(2\pi\varepsilon)^{-1/2}$. Moreover,
		$\int_{\calZ_t}|H_\varepsilon^{t,x}|^\alpha\dd s\dd y
		\le\calV_\varepsilon(\widetilde u)(t,x)$. The estimate used in
		Step~3 of the proof of Lemma~\ref{lem:regularized-mild-identity}
		gives, for $\delta,R>0$,
		\[
		\widetilde\P\left(
		\left|\int_{\calZ_t}H_\varepsilon^{t,x}(s,y)
		\widetilde L(\dd s,\dd y)\right|>\delta\right)
		\le
		\widetilde\P\bigl(\calV_\varepsilon(\widetilde u)(t,x)\ge R\bigr)
		+\frac{\calC_{5,1}(\alpha)R^{1/\alpha}}{\delta}.
		\]
		For fixed $R$, the first term tends to zero by
		\eqref{eq:mild-shift-functional-limit}, and the second term is small
		when $R$ is small. So the difference of the stochastic terms also
		tends to zero in probability. Together, the three terms prove
		\eqref{eq:mild-regularization-limit}.
	\end{proof}
	
	\subsection{The predictable mild representative}
	\label{subsec:predictable-mild-representative}
	Let $v:=\calT(\widetilde u)$, with the predictable modification
	chosen in Subsection~\ref{subsec:mild-kernel-map}, and set
	$v(0,x):=u_0(x)$.
	
	\emph{Step 1: $v(t,\cdot)$ is a density of $\widetilde u(t)$.} Fix
	$t\in(0,T]$. By the end of the proof of
	Theorem~\ref{thm:weak-existence}, $\widetilde u(t)\in\L^2(D)$
	with density $\widetilde u(t,\cdot)$, $\widetilde\P$-a.s. The heat
	semigroup is strongly continuous on $\L^2(D)$. Therefore
	\[
	\|P_{1/j}\widetilde u(t)-\widetilde u(t,\cdot)\|_{\L^2(D)}
	\longrightarrow0,\qquad\widetilde\P\text{-a.s.}
	\]
	Since $\operatorname{Leb}(D)=1$, the Cauchy--Schwarz inequality and bounded convergence
	give
	\[
	\widetilde\E\int_D
	\bigl(|(P_{1/j}\widetilde u(t))(x)-\widetilde u(t,x)|\wedge1\bigr)
	\dd x\longrightarrow0.
	\]
	Now fix $x\in D$. By Lemmas~\ref{lem:regularized-mild-identity}
	and~\ref{lem:remove-mild-regularization},
	$(P_{1/j}\widetilde u(t))(x)\to v(t,x)$ in probability. So
	$\widetilde\E(|(P_{1/j}\widetilde u(t))(x)-v(t,x)|\wedge1)\to0$ for
	each $x$. These expectations are bounded by $1$. By Tonelli's theorem
	and bounded convergence in $x$,
	\[
	\widetilde\E\int_D
	\bigl(|(P_{1/j}\widetilde u(t))(x)-v(t,x)|\wedge1\bigr)
	\dd x\longrightarrow0.
	\]
	Since $|a-c|\wedge1\le(|a-b|\wedge1)+(|b-c|\wedge1)$, we obtain
	$\widetilde\E\int_D(|v(t,x)-\widetilde u(t,x)|\wedge1)\dd x=0$, that
	is,
	\begin{equation}\label{eq:mild-fixed-time-density}
		\widetilde\P\bigl(
		v(t,x)=\widetilde u(t,x)\text{ for Leb-a.e. }x\in D
		\bigr)=1,\qquad t\in(0,T].
	\end{equation}
	In particular, for every $t\in(0,T]$,
	\[
	\widetilde\P\!\left(
	\langle\widetilde u(t),\varphi\rangle
	=\int_D v(t,x)\varphi(x)\dd x
	\quad\text{for all }\varphi\in\calE
	\right)=1.
	\]
	So $v$ is a density of $\widetilde u$, and condition~(i) of
	Definition~\ref{def:weak-solution} holds with the field $v$.
	
	\emph{Step 2: $v$ is also a weak solution.} Both $v$ and
	$\widetilde u$ are predictable. Integrating
	\eqref{eq:mild-fixed-time-density} in $t$ and using Tonelli's theorem
	gives
	\begin{equation}\label{eq:mild-field-spacetime-agreement}
		v(s,y)=\widetilde u(s,y)
		\quad\text{for Leb-a.e. }(s,y)\in\calZ_T,
		\qquad\widetilde\P\text{-a.s.}
	\end{equation}
	By \eqref{eq:limit-L2}, $v\in\L^2(\calZ_T)\subset\calL$,
	$\widetilde\P$-a.s. By \eqref{eq:mild-field-spacetime-agreement},
	$b(v)=b(\widetilde u)$ and $\sigma(v)=\sigma(\widetilde u)$ Leb-a.e.
	The drift integrals and condition~(iii) are therefore unchanged. For
	$\varphi\in\calE$, the field
	$\varphi[\sigma(v)-\sigma(\widetilde u)]$ has zero
	$\alpha$-integral, so its stochastic integral is zero. Since all
	processes are c\`adl\`ag, this holds for all $t$ on one event of
	probability one. So condition~(ii) holds with the field $v$, the same
	state, filtration and noise.
	
	\emph{Step 3: $v$ is a mild solution.} Fix $(t,x)\in(0,T]\times D$.
	By \eqref{eq:mild-field-spacetime-agreement},
	$\calV(v)(t,x)=\calV(\widetilde u)(t,x)$, which is finite
	almost surely by Lemma~\ref{lem:mild-kernel-integrability}. The drift
	integrals for $v$ and $\widetilde u$ coincide. The difference of the
	stochastic integrals is zero, since
	\[
	\int_{\calZ_t}\rho_{t-s}(x,y)^\alpha
	\bigl|\sigma(v(s,y))-\sigma(\widetilde u(s,y))\bigr|^\alpha
	\dd s\dd y=0,\qquad\widetilde\P\text{-a.s.}
	\]
	Therefore
	\[
	\calT(v)(t,x)=\calT(\widetilde u)(t,x)=v(t,x),
	\qquad\widetilde\P\text{-a.s.}
	\]
	So $v$ is a mild solution in the sense of
	Definition~\ref{def:mild-solution}, on the same stochastic basis and
	with the same noise. This proves Theorem~\ref{thm:mild-existence}.

	\appendix
	\section{Constants in the a priori estimates}\label{app:constants}
	We collect explicit choices for the constants in
	Lemmas~\ref{lem1}--\ref{lem:nonlinear-convergence} and
	Proposition~\ref{prop:limit-martingale}, and the integral construction
	in Subsection~\ref{subsec:weighted-stochastic-integration}, and
	Lemmas~\ref{lem:mild-kernel-integrability}
	and~\ref{lem:remove-mild-regularization},
	in the order in which they are used. Their parameter ranges and dependencies are stated in those
	results. The fixed parameters
	$T,\Cgrow,\|u_0\|_{\L^\infty(D)}$ are suppressed in the notation.
	Recall that $\lambda_1=\pi^2/2$ and
	$m_p(K)=2c_\alpha K^{p-\alpha}/(p-\alpha)$ for $p>\alpha$.
	The quantities $J_p^K$ and $D_p^K$ are defined before \eqref{eq:src},
	and $\zeta$ denotes the Riemann zeta function.
	
	\noindent\textit{Moment bound (Lemma~\ref{lem1}).}
	\begin{equation}\label{eq:C1-explicit}
		\begin{gathered}
			\eta_q:=\frac{3-q}{2},\qquad h_q:=(2\pi)^{-(q-1)/2},
			\qquad a_q:=1+3^{q-1}\|u_0\|_{\L^\infty(D)}^q,\\
			d_{q,K}:=3^{q-1}2^{q-1}(2\Cgrow)^q
			\bigl[T^{3(q-1)/2}+B_qm_q(K)h_q\bigr],\\
			\Cbound{q}{K}:=a_q E_{\eta_q}\bigl(d_{q,K}\Gamma(\eta_q)T^{\eta_q}\bigr).
		\end{gathered}
	\end{equation}
	Here $E_\eta(z):=\sum_{j=0}^\infty z^j/\Gamma(1+j\eta)$ is the
	Mittag--Leffler function, and $B_q$ is a Burkholder--Davis--Gundy
	constant, depending only on $q$, such that
	\[
	\E\sup_{s\le t}|M_s|^q\le B_q\E[M]_t^{q/2}
	\]
	for real local martingales $M$ with $M_0=0$; one may take $B_2=4$.
	
	\noindent\textit{Compensated-integral estimates.}
	Let $\alpha<q\le2$, let $H$ be a predictable $\LK$-integrable field, and
	let $I_t:=\int_{\calZ_t}H(s,y)\LK(\dd s,\dd y)$ be the c\`adl\`ag version.
	We apply \cite{CDH}*{Lemma~A.2(2)} with spatial set $E=D$, compensated
	measure $\Nt=N-\m$, where $\m(\dd s,\dd y,\dd z)=\dd s\dd y\nu_\alpha(\dd z)$,
	and $W(s,y,z)=z\ind{0<|z|\le K}$. Since
	$\int_\R|z|^q\ind{0<|z|\le K}\nu_\alpha(\dd z)=m_q(K)$, this lemma bounds
	$\E|I_t|^q$ by a constant times $m_q(K)\E\int_{\calZ_T}|H(s,y)|^q\dd s\dd y$,
	for each $t\le T$. Since $I$ is a martingale and $q>1$, Doob's maximal
	inequality extends this bound to the supremum over $t$:
	\begin{equation}\label{eq:compensated-maximal}
		\E\sup_{t\le T}|I_t|^q
		\le B_qm_q(K)\E\int_{\calZ_T}|H(s,y)|^q\dd s\dd y,
	\end{equation}
	whenever the right-hand side is finite. Here we enlarge $B_q$, if
	necessary, so that it also includes the constants of
	\cite{CDH}*{Lemma~A.2(2)} and of Doob's inequality.
	For stopping times $0\le\tau\le\tau'\le T$, the field
	$\mathbf1_{\{\tau<s\le\tau'\}}H(s,y)$ is predictable, and its integral up
	to time $t$ is $I_{t\wedge\tau'}-I_{t\wedge\tau}$. Applying
	\eqref{eq:compensated-maximal} to this field gives
	\begin{equation}\label{eq:compensated-stopped}
		\E\sup_{\tau\le t\le\tau'}|I_t-I_\tau|^q
		\le B_qm_q(K)\E\int_{\calZ_T}
		\mathbf1_{\{\tau<s\le\tau'\}}|H(s,y)|^q\dd s\dd y.
	\end{equation}
	
	\noindent\textit{Spatial and maximal bounds (Lemma~\ref{lem2}).}
	\begin{equation}\label{eq:space-sup-constants}
		\begin{aligned}
			C_{1,1}^K(q,\eta)
			&:=3^{q-1}\left[\frac{\|u_0\|_{\L^\infty(D)}^q}{q}
			+TD_q^K\lambda_1^{1-q}
			+\frac{TB_qm_q(K)D_q^K}{q}\right]
			\lambda_1^{\eta q/2-1}\zeta(2-\eta q),\\
			C_{1,2}^K(q,r)
			&:=3^{q-1}\left[\|u_0\|_{\L^\infty(D)}^q+T^qD_q^K
			+2^qB_qm_q(K)TD_q^K\right]
			\lambda_1^{-rq/2}\zeta(rq).
		\end{aligned}
	\end{equation}
	
	\noindent\textit{Time increments (Lemma~\ref{lem:time-increments}).}
	Let $\gamma$ be as in Lemma~\ref{lem:time-increments}, and set
	\begin{equation}\label{eq:time-auxiliary-constants}
		\begin{gathered}
			S_{q,r}:=\lambda_1^{-rq/2}\zeta(rq),\qquad
			V_q^K:=\frac{C_{1,2}^K(q,r)}{S_{q,r}},\\
			\frakS_{q,r,\gamma}:=
			\sup_{0<h\le1}h^{-\gamma}\sum_{k\ge1}
			\lambda_k^{-rq/2}\min(1,\lambda_kh)^q<\infty,\qquad
			N_{q,q'}^K:=2^{q/2}B_qm_q(K)(TJ_{q'}^K)^{q/q'}.
		\end{gathered}
	\end{equation}
	By \eqref{eq:space-sup-constants}, $V_q^K$ does not depend on $r$.
	We may then take
	\begin{equation}\label{eq:time-constants}
		\begin{aligned}
			C_{2,1}^K(q,q',r)
			&:=3^{q-1}\bigl[V_q^K\frakS_{q,r,\gamma}
			+(TD_q^K+2^qN_{q,q'}^K)S_{q,r}\bigr]
			+2^qC_{1,2}^K(q,r),\\
			C_{2,2}^K(q,q',r,\vartheta)
			&:=TC_{1,2}^K(q,r)
			+\frac{2C_{2,1}^K(q,q',r)T^{1+\gamma-\vartheta q}}
			{(\gamma-\vartheta q)(1+\gamma-\vartheta q)}.
		\end{aligned}
	\end{equation}
	
	\noindent\textit{Space--time tails and nonlinear convergence
		(Lemmas~\ref{lem:full-integrability-probability} and~\ref{lem:nonlinear-convergence}).}
	\begin{equation}\label{eq:tail-constants}
		C_{3,1}^K(q'):=T\Cbound{q'}{K},\qquad
		C_{3,2}(C_\Psi,\beta,q'):=(2C_\Psi)^{q'/\beta}.
	\end{equation}
	Here $\alpha<q'\le2$ and $0<\beta<q'$.
	The constant $C_{3,1}^K(q')$ depends only on
	$q',T,\Cgrow,\|u_0\|_{\L^\infty(D)}$ and $m_{q'}(K)$, and
	$C_{3,2}$ depends only on $C_\Psi,\beta,q'$.
	In particular, neither constant depends on $n$, $\Lambda_n$, or the
	thresholds $R$ and $a$. Note that $C_{3,1}^K(q')\to\infty$ as
	$K\to\infty$, since $m_{q'}(K)$ does.
	For each fixed $\varphi\in\calE$, we can take
	$C_\Psi=2^{2\alpha}\Cgrow^\alpha\|\varphi\|_{\L^\infty(D)}^\alpha$ for the
	noise term and $C_\Psi=4\Cgrow\|\varphi\|_{\L^\infty(D)}$ for the drift term.
	
	\noindent\textit{Bounded martingales (Proposition~\ref{prop:limit-martingale}).}
	For $f\in C_c^1([0,\infty);\R)$,
	\[
	\calC_{4,1}(f,\alpha,\theta)
	:=2\|f\|_{\infty}+\int_0^\infty
	|f'(a)-|\theta|^\alpha f(a)|\dd a.
	\]
	This constant bounds both $\widetilde Y_n^{f,\varphi,\theta}$ and
	$\widetilde Y^{f,\varphi,\theta}$ in
	\eqref{eq:smooth-cutoff-martingale}. With $f=\chi$, it also bounds
	the cutoff martingales used in
	Subsection~\ref{subsec:weighted-stochastic-integration} to extend
	\eqref{eq:elementary-integral-exponential}, whatever the elementary
	approximation and its integrand are. It depends only on
	$f,\alpha,\theta$, and not on $n$, $T$, $\varphi$, or $\omega$.
	
	\noindent\textit{Continuity of the elementary integral
		(Subsection~\ref{subsec:weighted-stochastic-integration}).}
	\[
	\calC_{5,1}(\alpha)
	:=\frac{2^{\,2-1/\alpha}\alpha}{\pi(\alpha-1)},
	\qquad 1<\alpha<2.
	\]
	This is the constant in \eqref{eq:elementary-integral-continuity}.
	It comes from integrating
	$\min\{2,R\theta^\alpha\}\theta^{-2}$ over $(0,\infty)$ and
	applying Doob's maximal inequality. Once the normalization of the
	stable noise is fixed, it depends only on $\alpha$, and not on
	$H$, $R$, $\varepsilon$, $n$, $K$, $r$, $T$, $\Cgrow$, $u_0$, or $\omega$.
	The bound depends on $H$ only through $\widetilde A^H(T)$. Note that
	$\calC_{5,1}(\alpha)\to\infty$ as $\alpha\downarrow1$.
	
	\noindent\textit{Integrability of the mild convolutions
		(Lemma~\ref{lem:mild-kernel-integrability}).}
	\begin{equation}\label{eq:mild-convolution-constant}
		\calC_{6,1}(K)
		:=T(J_2^K)^{1/2}
		+\frac{2(2\pi)^{-(\alpha-1)/2}}{3-\alpha}
		T^{(3-\alpha)/2}(J_2^K)^{\alpha/2},
		\qquad
		J_2^K=2(2\Cgrow)^2\bigl(1+\Cbound{2}{K}\bigr).
	\end{equation}
	This constant depends only on $\alpha,T,\Cgrow,
	\|u_0\|_{\L^\infty(D)}$ and $m_2(K)$; in particular, it depends on
	$c_\alpha$ and $K$ only through $m_2(K)$.
	It does not depend on $n$, $\Lambda_n$, $t$, $x$, $r$, or the threshold $R$.
	Note that $\calC_{6,1}(K)\to\infty$ as $K\to\infty$. This causes no
	problem, because in the proof of Lemma~\ref{lem:mild-kernel-integrability}
	we let $R\to\infty$ with $K$ fixed.
	
	\noindent\textit{Removal of the mild regularization
		(Lemma~\ref{lem:remove-mild-regularization}).}
	\begin{equation}\label{eq:mild-shift-constant}
		\calC_{6,2}(K,p):=(J_2^K)^{p/2},
		\qquad p\in\{1,\alpha\},
	\end{equation}
	where $J_2^K$ is defined in \eqref{eq:mild-convolution-constant}.
	For $p=1$, it bounds $\E|b_n(u_n^K(s,y))|$, and for $p=\alpha$, it
	bounds $\E|\sigma_n(u_n^K(s,y))|^\alpha$, uniformly in $n,s,y$.
	It depends only on $p,\alpha,T,\Cgrow,\|u_0\|_{\L^\infty(D)}$ and
	$m_2(K)$, and not on $n$, $\Lambda_n$, $\varepsilon$, $t$, $x$, $r$,
	or the threshold $R$.
	The kernel errors $\calR_{\varepsilon,p}(t,x)$ in
	\eqref{eq:mild-shift-error-tail} are deterministic and depend only on
	$\varepsilon,p,t,x$.
	As before, $\calC_{6,2}(K,p)\to\infty$ as $K\to\infty$. This causes no
	problem, because we first let $\varepsilon\downarrow0$ with $K$ fixed.

	\section{Predictable weak modification}\label{app:predictable-representatives}
	
	We use \emph{Fej\'er means} to obtain a predictable field from the
	distribution-valued state process. For $f\in \H_r$, $r\in\R$,
	define its $n$th sine Fej\'er mean by
	\[
	\fraks_n f(y)
	:=
	\sum_{k=1}^n\left(1-\frac{k}{n+1}\right)
	\langle f,e_k\rangle e_k(y),
	\qquad y\in D.
	\]
	This is the arithmetic average of the first $n+1$ sine partial
	sums, including the zero partial sum. It is well defined because
	$e_k\in \H_{-r}$, so every dual pairing $\langle f,e_k\rangle$
	is finite. By Parseval's identity,
	\[
	\|\fraks_n f-f\|_{\H_r}\longrightarrow0.
	\]
	Thus the approximation converges in $\H_r$, and hence in the sense
	of distributions. For every $f\in\L^2(D)$, \cite{Katznelson}*{Chapter~I, Theorem~3.2}, applied to the odd $2$-periodic extension of $f$, shows that its sine Fej\'er means satisfy
	\[
	\fraks_n f(y)\longrightarrow f(y)
	\quad\text{for Leb-a.e. }y\in D.
	\]
	
	For a c\`adl\`ag state process $w$, we apply this operator to
	$w(t-)$, with $w(0-):=w(0)$. The resulting fields
	$(t,y)\mapsto\fraks_n(w(t-))(y)$ are predictable with respect
	to the natural filtration of $w$. The following lemma shows that,
	when the state agrees with a space--time field, these approximations
	yield a predictable weak modification of that field.
	
	\begin{lemma}\label{lem:predictable-representative}
		Let $w$ be a random element of $\Path{\H_{-r}}$, with $r>1/2$,
		and let $v$ be a jointly measurable field such that, $\P$-a.s.,
		$v\in\L^2((0,T)\times D)$ and
		\[
		\langle w(t),e_k\rangle
		=\int_D v(t,y)e_k(y)\,\dd y
		\quad\text{for Leb-a.e. }t\in(0,T),\quad \forall k\ge1.
		\]
		Then $v$ admits a weak modification $\bar w$ that is predictable
		with respect to the natural filtration of $w$.
	\end{lemma}
	\begin{proof}
		Let $\Fil_t^w:=\sigma(w(s):0\le s\le t)$.
		Set $w(0-):=w(0)$ and define
		\[
		S_nw(t,y)
		:= \fraks_n(w(t-))(y) =
		\sum_{k=1}^n\left(1-\frac{k}{n+1}\right)
		\langle w(t-),e_k\rangle e_k(y).
		\]
		The left-limit coefficients are adapted and left-continuous,
		hence predictable. Therefore each $S_nw$ is measurable with respect
		to the predictable $\sigma$-field of $(\Fil_t^w)$ tensored with
		$\calB(D)$.
		
		Define
		\[
		\bar w(\omega,t,y)
		:=
		\begin{cases}
			\displaystyle\lim_{n\to\infty}S_nw(\omega,t,y),
			&\text{if the limit exists in }\R,\\
			0,&\text{otherwise}.
		\end{cases}
		\]
		Let $\calP_0^w$ denote the predictable $\sigma$-field associated
		with $(\Fil_t^w)$. The set
		\[
		\bigl\{(\omega,t,y) \in \Omega \times [0,T] \times D :
		\lim_{n\to\infty}S_nw(\omega,t,y)
		\text{ exists in }\R\bigr\}
		\]
		belongs to $\calP_0^w\otimes\calB(D)$.
		Thus $\bar w$ is measurable with respect to this $\sigma$-field,
		and hence is a predictable field.
		
		Fix $\omega$ such that
		$v(\omega,\cdot,\cdot)\in\L^2((0,T)\times D)$,
		$w(\omega)$ is c\`adl\`ag, and
		\[
		\langle w(\omega,t),e_k\rangle
		=\int_D v(\omega,t,y)e_k(y)\,\dd y
		\quad\text{for Leb-a.e. }t\in(0,T),\quad \forall k\ge1.
		\]
		These properties hold on an event of probability one.
		For this $\omega$, outside a Lebesgue-null set of times,
		$v(\omega,t,\cdot)\in\L^2(D)$, all the coefficient
		identities hold, and $w(\omega,t-)=w(\omega,t)$, since a c\`adl\`ag path
		has at most countably many jumps. Thus, for these times, $S_nw(t,\cdot)$ equals the weighted sine sum
		\[
		\sum_{k=1}^n\left(1-\frac{k}{n+1}\right)
		\left(\int_D v(t,z)e_k(z)\,\dd z\right)e_k(\cdot),
		\]
		called the $n$th Fej\'er mean of $v(t,\cdot)$. Applied to its odd, $2$-periodic
		extension, \cite{Katznelson}*{Chapter~I, Theorem~3.2} gives
		\[
		S_nw(t,y)\longrightarrow v(t,y)
		\quad\text{for Leb-a.e. }y\in D.
		\]
		Tonelli's theorem therefore yields
		\[
		\P\bigl(\bar w=v
		\text{ Leb-a.e. on }\calZ_T\bigr)=1.
		\]
		Hence $\bar w$ is a weak modification of $v$.
		
		The same approximations also converge in
		$\L^2((0,T)\times D)$, $\P$-a.s.
		Indeed, for the times considered above, put
		$c_k(t):=\int_D v(t,y)e_k(y)\,\dd y$. Parseval's identity gives
		\[
		\|S_nw(t,\cdot)-v(t,\cdot)\|_{\L^2(D)}^2
		=
		\sum_{k=1}^n\left(\frac{k}{n+1}\right)^2|c_k(t)|^2
		+
		\sum_{k>n}|c_k(t)|^2
		\longrightarrow0.
		\]
		This expression is bounded by
		$\|v(t,\cdot)\|_{\L^2(D)}^2$.
		Since $v\in\L^2((0,T)\times D)$, dominated convergence gives
		\[
		\|S_nw-v\|_{\L^2((0,T)\times D)}
		\longrightarrow0,
		\qquad \P\text{-a.s.}
		\]
	\end{proof}
	
	\section*{Acknowledgments}
	I am grateful to Le Chen for many helpful suggestions and insightful
	discussions, and for giving me the freedom to pursue this work during
	my postdoctoral position at Auburn University. I was also partially supported by NSF grants DMS-2246850/2443823 during the preparation of this work.

	\noindent \textbf{AI disclosure statement.}  Large language model tools were used to help with language editing, reference checking, and LaTeX formatting. The author takes full responsibility for the content of this manuscript.

	\phantomsection

\end{document}